\documentclass[graybox]{svmult}

\usepackage{amsfonts,amssymb,amscd,amstext,amsmath}

\usepackage{mathptmx}          
\usepackage{helvet}                
\usepackage{courier}               
\usepackage{makeidx}              
\usepackage{graphicx}             
\usepackage{multicol}               
\usepackage[bottom]{footmisc} 

\usepackage{indentfirst}

\usepackage[utf8]{inputenc}
\usepackage{hyperref}
\usepackage{verbatim}
\usepackage{enumerate}

\parskip=\smallskipamount

\newcommand\Acal{\mathcal{A}}

\newcommand\Ccal{\mathcal{C}}

\newcommand\Jcal{\mathcal{J}}

\newcommand\Mcal{\mathcal{M}}

\newcommand\Ocal{\mathcal{O}}

\newcommand\Rcal{\mathcal{R}}

\newcommand\Wcal{\mathcal{W}}

\newcommand\B{\mathbb{B}}
\newcommand\C{\mathbb{C}}
\newcommand\CP{\mathbb{CP}}
\renewcommand\D{\mathbb D}

\newcommand\N{\mathbb{N}}
\renewcommand\P{\mathbb{P}}
\newcommand\R{\mathbb{R}}

\newcommand\Z{\mathbb{Z}}

\newcommand\igot{\mathfrak{i}}

\renewcommand\igot{\mathfrak{i}}

\newcommand\Rgot{\mathfrak{R}}

\renewcommand\imath{\igot}

\newcommand\hra{\hookrightarrow}

\newcommand\wt{\widetilde}
\newcommand\wh{\widehat}
\newcommand\di{\partial}
\newcommand\dibar{\overline\partial}

\newcommand\dist{\mathrm{dist}}

\newcommand\supp{\mathrm{supp}}

\newcommand\Ocalc{\overline{\mathcal{O}}}
\newcommand\Ocalcl{\Ocalc_{\mathrm{loc}}}

\begin{document}

\title*{Holomorphic approximation: the legacy of Weierstrass, Runge, Oka-Weil, and Mergelyan}
\titlerunning{Holomorphic approximation}
\author{John Erik Forn{\ae}ss, Franc Forstneri{\v c}, and Erlend F.\ Wold}
\authorrunning{J.E.~Forn{\ae}ss, F.~Forstneri{\v c}, and E.F.~Wold}
\institute{
J.\ E.\ Forn{\ae}ss \at Department of Mathematical Sciences, NTNU, 7491 Trondheim, Norway \\
\email{john.fornass@ntnu.no}\\ \\
F.~Forstneri{\v c} \at Faculty of Mathematics and Physics, University of Ljubljana, 
and Institute of Mathematics, Physics and Mechanics, Jadranska 19, SI--1000 Ljubljana, Slovenia \\ 
\email{franc.forstneric@fmf.uni-lj.si} \\ \\
E.\ F.\ Wold \at Department of Mathematics, University of Oslo, Postboks 1053 Blindern, 
NO-0316 Oslo, Norway\\
\email{erlendfw@math.uio.no}
}
\maketitle

\abstract{In this paper we survey the theory of holomorphic approximation, from the classical 19th century results
of Runge and Weierstrass, continuing with the 20th century work of Oka and Weil, Mergelyan, 
Vitushkin and others, to the most recent ones on higher dimensional manifolds.   
The paper includes some new results and applications of this theory, especially to
manifold-valued maps.}

\setcounter{tocdepth}{2}
\tableofcontents


%
%

\section{Introduction}\label{sec:Intro}

%
%

The aim of this paper is to provide a review and synthesis of holomorphic approximation theory from
classical to modern. The emphasis is on recent results and applications to manifold-valued maps.

Approximation theory plays a fundamental role in complex analysis, holomorphic dynamics,
the theory of minimal surfaces in Euclidean spaces, and in many other related fields of Mathematics
and its applications. It provides an indispensable tool 
in constructions of holomorphic maps with desired
properties between complex manifolds. Applications of this theory are too numerous
to be presented properly in a short space, but we mention several of them at
appropriate places and provide references that the reader might pursue.
We are hoping that the paper will bring a new stimulus for future developments in this important area of analysis. 

Although this is largely a survey, it includes some new results, 
especially those concerning Mergelyan approximation in higher dimension (see Sect.\ \ref{sec:Mergelyan}),
and applications of these techniques to manifold-valued maps (see Sect.\ \ref{sec:manifold}).
We also mention open problems and indicate promising directions.
Proofs are outlined where possible, especially of those result which introduce major new ideas. 
More advanced results are only mentioned with references to the original sources. 
Of course we included proofs of the new results.

There exist a number of surveys on holomorphic approximation theory; see e.g.\  \cite{BoivinGauthier2001,Fuchs1968,Gaier1987,Gamelin1971,Gamelin1984,GauthierParamonov2018,Gauthier1994,GauthierHengartner1982, GauthierSharifi2017,Levenberg2006,Zalcman1968}, among others.
However, ours seems the first attempt at a unified picture, from the highlights of the classical theory to 
results in several variables and for manifold-valued maps. On the other hand, several of the surveys mentioned above
include discussions of certain finer topics of approximation theory that we do not cover here, 
also for solutions of more general elliptic partial differential equations.
It is needless to say that the higher dimensional approximation theory is much less developed and the problems
tend to be considerably more complex. It is also clear that further progress in many areas of complex analysis
and its applications hinges upon developing new and more powerful approximation techniques for 
holomorphic mappings. 

%
%
{\em Organization of the paper.}
In Sects.\ \ref{sec:Runge}--\ref{sec:Mergelyan-smooth} we review the main achievements of the classical 
approximation theory for functions on the complex plane $\C$ and on Riemann surfaces. 
Our main goal is to identify those key ideas and principles which may serve as guidelines when 
considering approximation problems in several variables and for manifold-valued maps.
We begin in Sect.\ \ref{sec:Runge} with theorems of K.\ Weierstrass,
C.\ Runge, S.\ N.\ Mergelyan, and A.\ G.\ Vitushkin. In Sect.\ \ref{sec:unbounded} 
we discuss approximation on closed 
unbounded subsets  of $\C$ and of Riemann surfaces. There are two main lines in the literature,
one following the work of T.\ Carleman on approximation in the fine topology, and another  
the work of  N.\ U.\ Arakelian  on uniform approximation. In Sect.\ \ref{sec:Mergelyan-smooth} we survey
results on $\Ccal^k$ Mergelyan approximation of smooth functions on Riemann surfaces.
The remainder of the paper is devoted to the higher dimensional theory.
In Sect.\ \ref{sec:OkaWeil} we recall the Oka-Weil approximation theorem 
on Stein manifolds and some generalizations; these are higher dimensional analogues of Runge's theorem. 
In Sect.\ \ref{sec:Mergelyan} we discuss Mergelyan and Carleman 
approximation of functions and closed forms on $\C^n$ and on Stein manifolds. 
In Sect.\ \ref{sec:manifold} we look at applications of these and other techniques
to local and global approximation problems of Runge, Mergelyan, Carleman, and Arakelian type 
for maps from Stein manifolds to more general complex manifolds;
these  are especially interesting when the target is an {\em Oka manifold}.
Subsect.\ \ref{ss:Mergelyan-manifold} contains very recent results
on Mergelyan approximation of manifold-valued maps.
In Sect.\ \ref{sec:weights} we mention some recent progress on weighted approximation in $L^2$ spaces.

%
%
{\em Notation and terminology.} 
We denote by  $\N=\{1,2,3,\ldots\}$ the natural numbers, by $\Z$ the ring of integers,
$\Z_+=\{0,1,2,\ldots\}$, and by $\R$ and $\C$ the fields of real and complex numbers, respectively.
For any $n\in \N$ we denote by $\R^n$ the $n$-dimensional real Euclidean space,  and by $\C^n$ the 
$n$-dimensional complex Euclidean space with complex coordinates $z=(z_1,\ldots,z_n)$, 
where $z_i=x_i+\imath y_i$ with $x_i,y_i\in\R$ and $\imath=\sqrt{-1}$.  We denote the
Euclidean norm by $|z|^2=\sum_{i=1}^n |z_i|^2$. Given $a\in\C$ and $r>0$, we set
$\D(a,r)=\{z\in\C: |z-a|<r\}$ and $\D=\D(0,1)$. Similarly, $\B^n$ denotes the unit ball in $\C^n$ 
and $\B^n(a,r)$ the ball centered at $a\in\C^n$ of radius $r$. The corresponding balls
in $\R^n$ are denoted $\B^n_\R$ and $\B^n_\R(a,r)$.

Let $X$ be a complex manifold. We denote by $\Ccal(X)$ and $\Ocal(X)$ the Fr{\'e}chet algebras of all 
continuous and holomorphic functions on $X$, respectively, endowed with 
the compact-open topology. Given a compact set $K$ in $X$, we denote by $\Ccal(K)$ the Banach algebra 
of all continuous complex valued functions on $K$ with the supremum norm, 
by $\Ocal(K)$ the set of all functions that are holomorphic in a neighborhood of $K$
(depending on the function), and by $\Ocalc(K)$ the uniform closure of $\{f|_K: f\in \Ocal(K)\}$
in $\Ccal(K)$. By $\Acal(K)=\Ccal(K)\cap \Ocal(\mathring K)$ we denote the set of all continuous
functions $K\to\C$ which are holomorphic in the interior $\mathring K$ of $K$.
If $r\in \Z_+\cup\{\infty\}$ we let $\Ccal^r(K)$ denote the space of all functions on $K$
which extend to $r$-times continuously differentiable functions on $X$, 
and $\Acal^r(K)=\Ccal^r(K)\cap \Ocal(\mathring K)$.
Given a complex manifold $Y$, we use the analogous notation $\Ocal(X,Y)$, $\Ocal(K,Y)$,
$\Acal^r(K,Y)$, etc., for the corresponding classes of maps into $Y$. We have the inclusions
\begin{equation}\label{eq:spaces}
	\Ocal(K,Y) \subset \Ocalc(K,Y) \subset \Acal(K,Y)\subset \Ccal (K,Y).
\end{equation}
A compact set $K$ in a complex manifold $X$ is said to be {\em $\Ocal(X)$-convex} if 
\begin{equation}\label{eq:hull}
	K=\wh K_{\Ocal(X)}:= \{p\in X: |f(p)|\le \max_{x\in K} |f(x)|\ \ \forall f\in\Ocal(X)\}.
\end{equation}
A compact $\Ocal(\C^n)$-convex set $K$ in $\C^n$ is said to be {\em polynomially convex}.
A compact set $K$ in a complex manifold $X$ is said to be a {\em Stein compact} if it admits a basis of 
open Stein neighborhoods in $X$.

%
%

\section{From Weierstrass and Runge to Mergelyan}\label{sec:Runge}

In this and the following two sections we survey the main achievements of the classical
holomorphic approximation theory. More comprehensive surveys of this subject are available in
\cite{BoivinGauthier2001,Fuchs1968,Gaier1987,Gamelin1971,Gamelin1984,GauthierParamonov2018,Gauthier1994,GauthierHengartner1982,Zalcman1968}, among other sources. 

The approximation theory for holomorphic functions has its origin in two classical theorems
 from 1885. The first one, due to K.\ Weierstrass \cite{Weierstrass1985}, concerns the approximation 
of continuous functions on compact intervals in $\R$ by polynomials.

%
%
\begin{theorem}[Weierstrass (1885), \cite{Weierstrass1985}]
\label{th:Weierstrass}
Suppose $f$ is a continuous function on a closed bounded interval $[a,b]\subset\R$. 
For every $\epsilon>0$ there exists a polynomial $p$ such that for all $x\in [a,b]$ we have 
$|f(x)-p(x)|<\epsilon$.
\end{theorem}

\begin{proof}
We use convolution with the Gaussian kernel. After extending
$f$ to a continuous function on $\R$ with compact support, we consider the 
family of entire functions
\begin{equation}\label{eq:Gauss}
       f_\epsilon(z) = \frac{1}{\epsilon \sqrt\pi} \int_{\R} f(x)e^{-(x-z)^2/{\epsilon^2}}dx,\qquad z\in\C,\ \epsilon>0.
\end{equation}
As $\epsilon \to 0$, we have that $f_\epsilon\to f$ uniformly on $\R$. Hence, the Taylor polynomials 
of $f_\epsilon$ approximate $f$ uniformly on compact intervals in $\R$. If furthermore $f$ is 
of class $\Ccal^k$, then by a change of variable $u=x-z$ and placing the derivatives on $f$
it follows that we get convergence also in the $\Ccal^k$ norm.
\qed\end{proof}

The paper by A.\ Pinkus \cite{Pinkus2000} (2000) contains a more complete survey of Weierstrass's results 
and of his impact on the theory of holomorphic approximation.
As we shall see in Subsect.\ \ref{ss:submanifolds}, the idea of using convolutions 
with the Gaussian kernel gives major approximation results 
also on certain classes of real submanifolds in complex Euclidean space $\C^n$ and, more generally, in Stein manifolds.

One line of generalizations of Weierstrass's theorem was discovered by 
M.\ Stone in 1937, \cite{Stone1937,Stone1948}. The {\em Stone-Weiestrass theorem} says that,
if $X$ is a compact Hausdorff space and $A$ is a subalgebra 
of the Banach algebra $\Ccal(X,\R)$ which contains a nonzero constant function,
then $A$ is dense in $\Ccal(X,\R)$ if and only if it separates points. 
It follows in particular that any complex valued continuous function on a compact set $K\subset\C$ 
can be uniformly approximated by polynomials in $z$ and $\bar z$. 
Stone's theorem opened a major direction of research in Banach algebras.

Another line of generalizations concerns approximation of continuous functions on curves in the
complex plane by holomorphic polynomials and rational functions. This led to
Mergelyan and Carleman theorems discussed in the sequel.

However, we must first return to the year 1885. The second of the two classical approximation 
theorems proved that year is due to C.\ Runge \cite{Runge1885}. 

%
%
\begin{theorem}[Runge (1885), \cite{Runge1885}] 
\label{th:Runge}
Every holomorphic function on an open neighborhood of a compact set 
$K$ in $\C$ can be approximated uniformly on $K$ by rational functions without poles in $K$,
and by holomorphic polynomials if $\C\setminus K$ is connected.
\end{theorem}

The maximum principle shows that the condition that $K$ does not separate
the plane is necessary for polynomial approximation  on $K$.

\begin{proof}
The simplest proof of Runge's theorem, and the one given in most textbooks on the subject
(see e.g.\ \cite[p.\ 270]{Rudin1987}), goes as follows.
Assume that $f$ is a holomorphic function on an open set $U\subset \C$ 
containing $K$. Choose a smoothly bounded domain $D$ with $K\subset D$ and $\overline D\subset U$.
By the Cauchy integral formula we have that
\[
	f(z) = \frac{1}{2\pi\imath} \int_{bD} \frac{f(\zeta)}{\zeta-z} \, d\zeta,\qquad z\in D.
\]
Approximating the integral by Riemann sums provides uniform approximation of $f$ on $K$
by linear combinations of functions $\frac{1}{a-z}$ with poles $a\in\C\setminus K$.
Assuming that $\C\setminus K$ is connected, we can  push the poles to infinity as follows.
Pick a disc $\Delta\subset \C$ containing $K$. Since $\C\setminus K$ is connected, 
there is a path $\lambda:[0,1]\to\C\setminus K$ connecting $a=\lambda(0)$ to a point 
$b=\lambda(1)\in \C\setminus \overline\Delta$.  Let $\delta=\inf \{\dist(\lambda(t),K):t\in [0,1]\}>0$. 
Choose points $a=a_0,a_1,\ldots,a_N=b \in \lambda([0,1])$
such that $|a_j-a_{j+1}|<\delta$ for $j=0,\ldots,N-1$. For $z\in K$ 
and $j=0,1,\ldots,N-1$ we then have that
\[
	\frac{1}{a_j-z}= \frac{1}{(a_{j+1}-z)-(a_{j+1}-a_j)}=
	\sum_{k=0}^\infty \frac{(a_{j+1}-a_j)^k}{(a_{j+1}-z)^{k+1}},
\]
where the geometric series converges uniformly on $K$.
It follows by a finite induction that $\frac{1}{a-z}$ is a uniform limit on $K$ of polynomials in 
$\frac{1}{b-z}$. Since $b\in \C\setminus \overline \Delta$, the function $\frac{1}{b-z}$ is a uniform limit on 
$\Delta$ of holomorphic polynomials in $z$ and the proof is complete.
If $\C\setminus K$ is not connected, a modification of this argument 
gives uniform approximations of $f$ by rational functions with poles in a given set 
$\Lambda\subset \C\setminus K$ containing a  point in 
every bounded connected component of $\C\setminus K$.

Another proof uses the {\em Cauchy-Green formula},
also called the {\em Pompeiu formula} for compactly supported function $f\in \Ccal^1_0(\C)$:
\begin{equation}\label{eq:Pompeiu}
	f(z) = \frac{1}{\pi} \int_\C \frac{\dibar f(\zeta)}{z-\zeta}\, du\, dv,  \qquad z\in \C,\ \zeta=u+\imath v.
\end{equation}
Here, $\dibar f(\zeta)= (\di f/\di \bar\zeta)(\zeta)$.
If $f$ is holomorphic in an open set $U\subset \C$ containing a compact set $K$, we choose a smooth function
$\chi\colon \C\to [0,1]$ which equals $1$ on a smaller neighborhood $V$ of $K$ and satisfies $\supp(\chi)\subset U$. 
Then, 
\[
	f(z) = \frac{1}{\pi} \int_\C \frac{\dibar \chi(\zeta) \, f(\zeta)}{z-\zeta}\, du\, dv,\qquad z\in V.
\]
Since the integrand is supported on $\supp(\dibar\chi)$ which is disjoint from $K$, approximating the 
integral  by Riemann sums shows that $f$ can be approximated uniformly on $K$ by rational functions 
with poles in $\C\setminus K$, and the proof is concluded as before.
\qed\end{proof}

We now digress for a moment to recall the main properties of the Cauchy-Green operator 
in \eqref{eq:Pompeiu} which is used in many approximation results discussed in the sequel. 

Given a compact set $K\subset \C$ and an integrable function $g$ on $K$, we set
\begin{equation}\label{eq:TK}
	T_K(g)(z) = \frac{1}{\pi} \int_K \frac{g(\zeta)}{z-\zeta}\, dudv,\qquad \zeta =u+\imath v.
\end{equation}
It is well known (see e.g.\ L.\ Ahlfors \cite[Lemma 1, p.\ 51]{Ahlfors2006} or A.\ Boivin and P.\ Gauthier
\cite[Lemma 1.5]{BoivinGauthier2001}) that for any $g\in L^p(K)$, $p>2$, $T_K(g)$ is a bounded 
continuous function on $\C$ that vanishes at infinity and satisfies the uniform H\"older condition 
with exponent $\alpha=1-2/p$; moreover, $T_K\colon L^p(K)\to \Ccal^\alpha(\C)$ is a continuous
linear operator. (A closely related operator is actually bounded from $L^p(\C)$ to $\Ccal^{1-2/p}(\C)$
without any support condition.) The key property of $T_K$ is that it solves the non\-homogeneous
Cauchy-Riemann equation, that is,
\[
	\dibar \,T_K(g)=g
\] 
holds in the sense of distributions, and in the classical sense on any open subset 
on which $g$ is of class $\Ccal^1$. In particular, $T_K(g)$ is holomorphic on $\C\setminus K$.
The optimal sup-norm estimate of $T_K(g)$ for $g\in L^\infty(K)$ 
is obtained from Mergelyan's estimate
\begin{equation}\label{eq:Mest}   
	 \int_{\zeta\in K} \frac{du dv}{|z-\zeta|} \,\, \le\,\, \sqrt{4\pi \, \mathrm{Area}(K)},
	 \qquad z\in \C,
\end{equation}
which is sharp when $K$ is the union of a closed disc centered at $z$ and a compact set of measure zero.
(See S.\ N.\ Mergelyan \cite{Mergelyan1952,Mergelyan1954} or 
A.\ Browder \cite[Lemma 3.1.1]{BrowderA1969}.)
The related Ahlfors-Beurling estimate which is also sharp is that 
\[
	|T_K(1)(z)| =  \left|\frac{1}{\pi}\int_{\zeta\in K} \frac{du dv}{z-\zeta}\right|  
	\,\,\le\,\,  \sqrt{\frac{\mathrm{Area}(K)}{\pi}},\qquad z\in\C.
\]
Another excellent source for this topic is the book of 
K.\ Astala, T.\ Iwaniec and G.\ Martin \cite{AstalaIwaniecMartin2009};
see in particular Sect.\ 4.3 therein.

Coming back to the topic of approximation, the situation becomes considerably more delicate 
when the function $f$ to be approximated 
is only continuous on $K$ and holomorphic in the interior $\mathring K$; that is,
$f\in\Acal(K)$. The corresponding approximation problem for compact sets in $\C$ 
with connected complement was solved by S.\ N.\ Mergelyan in 1951. 

%
%

\begin{theorem}[Mergelyan (1951), \cite{Mergelyan1951,Mergelyan1952,Mergelyan1954}] 
\label{th:Mergelyan1951}
If $K$ is a compact set in $\C$ with connected complement, then every function 
in $\Acal(K)$ can be approximated uniformly on $K$ by holomorphic polynomials.
\end{theorem}

Mergelyan's theorem generalizes  both Runge's and Weierstrass's theorem. 
It also contains as special cases the theorems of J.\ L.\ Walsh \cite{Walsch1926} (1926) in which 
$K$ is the closure of a Jordan domain, F.\ Hartogs and A.\ Rosenthal 
\cite{HartogsRosenthal1931} (1931) in which $K$ has Lebesgue measure zero,
M.\ Lavrentieff \cite{Lavrentieff1936} (1936) in which $K$ is nowhere dense, and  
M.\ V.\ Keldysh \cite{Keldysh1945} (1945) in which $K$ is the closure of its interior. 

In light of Runge's theorem, the main new point in Mergelyan's theorem 
is to approximate functions in $\Acal(K)$  by functions holomorphic in open neighborhoods of $K$, that is,
to show that 
\[ 
	\Acal(K) = \Ocalc(K).
\] 
If this holds, we say that $K$ (or $\Acal(K)$) enjoys the {\em Mergelyan property}. 
Hence, Mergelyan's theorem is essentially of local nature,
where {\em local} now pertains to {\em neighborhoods of $K$}. This aspect
is emphasized further by Bishop's localization theorem,  Theorem \ref{th:Bishop},
and its converse, Theorem \ref{th:localization-converse}.

Some generalizations of Mergelyan's theorem can be found in his papers
\cite{Mergelyan1952,Mergelyan1954}.  Subsequently 
to Mergelyan, another proof was given by E.\ Bishop in 1960, \cite{Bishop1960}, and 
yet another by L. Carleson in 1964, \cite{Carleson1964}.  Expositions are available in 
many sources; see for instance D.\ Gaier \cite[p.\ 97]{Gaier1987}, T.\ W.\ Gamelin \cite{Gamelin1984},
and W.\ Rudin \cite{Rudin1987}. We outline the proof and refer to the cited sources for the details. 

%
%
\smallskip
\noindent{\em Sketch of proof of Theorem \ref{th:Mergelyan1951}.}
By Tietze's extension theorem, every$f\in \Acal(K)$ extends to a continuous function 
with compact support on $\C$. Fix a number $\delta>0$. Let $\omega(\delta)$ denote the modulus of continuity of $f$.
By convolving $f$ with the function $A_\delta\colon\C\to\R_+$ 
defined by $A_\delta(z)=0$ for $|z|>\delta$ and
\[
	A_\delta(z) = \frac{3}{\pi \delta^2}\left(1-\frac{|z|^2}{\delta^2}\right)^2, \qquad 0\le |z|\le \delta,
\] 
we obtain a function $f_\delta\in \Ccal^1_0(\C)$ with compact support such that 
\[
	|f(z)-f_\delta(z)|<\omega(\delta)\quad \text{and} \quad
	\big| \frac{\di f_\delta}{\di \bar z}(z)\big|< \frac{2\omega(\delta)}{\delta},\quad z\in\C,
\]
and $f_\delta=f$ on $K_\delta = \{z\in K: \dist(z,\C\setminus K)>\delta\}$. By the Cauchy-Green formula
\eqref{eq:Pompeiu},
\[
	f_\delta(z) = \frac{1}{\pi} \int_\C \frac{\dibar f_\delta(\zeta)}{z-\zeta}\, du\, dv,  \qquad z\in \C.
\]
Next, we cover the compact set $X=\supp(\dibar f_\delta)$ 
by finitely many open discs $D_j=\D(z_j,2\delta)$ $(j=1,\ldots,n)$ with centers $z_j \in \C\setminus K$ such that 
each $D_j$ contains a compact Jordan arc $E_j\subset D_j\setminus K$ of diameter at least $2\delta$.
(Such discs $D_j$ and arcs $E_j$ exist because $\C\setminus K$ is connected.)
The main point now is to approximate the Cauchy kernel $\frac{1}{z-\zeta}$ for $z\in \C\setminus E_j$ 
and $\zeta\in D_j$ sufficiently well by a function of the form
\[
	P_j(z,\zeta) = g_j(z) + (\zeta -b_j)g_j(z)^2,
\]
where $g_j\in \Ocal(\C\setminus E_j)$ and $b_j\in \C$. This is accomplished by 
Mergelyan's lemma which says that $g_j$ and $b_j$ can be chosen such that the inequalities
\begin{equation}\label{eq:estP}
	|P_j(z,\zeta)|<\frac{50}{\delta}\quad \text{and}\quad 
	\bigg|P_j(z,\zeta)-\frac{1}{z-\zeta}\bigg|< \frac{4000\, \delta^2}{|z-\zeta|^3}
\end{equation}
hold for all $z\in \C\setminus E_j$ and $\zeta\in D_j$. 
(See also \cite[p.\ 104]{Gaier1987} or \cite[Lemma 20.2]{Rudin1987}.)  Set
\[
	X_1=X\cap \overline D_1, \qquad 
	X_j= X\cap \overline D_j\setminus (X_1\cup\ldots \cup X_{j-1})\ \ \text{for}\ \  j=2,\ldots,n.
\]
The open set $\Omega=\C\setminus\bigcup_{j=1}^n E_j$ clearly contains $K$. The function 
\[
	F_\delta(z)= \sum_{j=1}^n \frac{1}{\pi} \int_{X_j} \frac{\di f_\delta}{\di \bar \zeta}(\zeta) P_j(z,\zeta)\, du\, dv
\]
is holomorphic in $\Omega$ (since every function $P_j(z,\zeta)$ is holomorphic for $z\in \Omega$), 
and it follows from \eqref{eq:estP} that $|F_\delta(z) -  f_\delta(z)| < 6000 \omega(\delta)$ for all 
$z\in\Omega$. As $\delta\to 0$, we have that $\omega(\delta)\to 0$ and hence $F_\delta\to f$ 
uniformly on $K$.
\qed\smallskip

We now consider approximation problems on Riemann surface.
Fundamental discoveries concerning function theory on {\em open Riemann surfaces} 
were made by H.\ Behnke and K.\ Stein \cite{BehnkeStein1949} in 1949. They proved
the following extension of Runge's theorem to open Riemann surfaces 
(see \cite[Theorem 6]{BehnkeStein1949}); the case for $X$ compact 
was pointed out by H.\ L.\ Royden in 1967, \cite[Theorem 10]{Royden1967JAM}, and 
again by H.\ K{\"o}ditz and S.\ Timmann in 1975 \cite[Satz 1]{KoditzTimmann1975}.

%
%

\begin{theorem} [Runge's theorem on Riemann surfaces; 
\cite{BehnkeStein1949,Royden1967JAM,KoditzTimmann1975}] 
\label{th:Runge2}
If $K$ is a compact set in a Riemann surface $X$, then every holomorphic function $f$
on a neighborhood of $K$ can be approximated uniformly on $K$ by meromorphic functions $F$ on $X$
without poles in $K$, and by holomorphic functions on $X$ if $X\setminus K$ has no relatively compact 
connected components. 
\end{theorem}

In the papers of Royden \cite{Royden1967JAM} and K{\"o}ditz and Timmann \cite{KoditzTimmann1975}
the function $f$ is assumed to be meromorphic on a neighbourhood of $K$ 
(with at most finitely many poles on $K$), the approximating meromorphic function 
$F$ on $X$ has no poles on $K$ except those of $f$,
and its poles in $X\setminus K$ are located in a set $E$ having one point 
in each connected component of $X\setminus K$. Furthermore, Royden showed that
$F$ can be chosen to agree with $f$ to a given finite order at a given finite set of points 
in $K$. 

A relatively compact connected component of $X\setminus K$ is called a {\em hole} of $K$.
A compact set without holes in an open Riemann surface $X$ is also called a {\em Runge compact} in $X$.
The following is a corollary to Theorem \ref{th:Runge2} and the maximum principle. 

%
%
\begin{corollary}\label{cor:AB} 
Let $X$ be an open Riemann surface.
\begin{enumerate}
\item[\rm (a)]  Holomorphic functions on $X$ separate points, that is, for any pair of distinct points
$p,q\in X$ there exists $f\in\Ocal(X)$ such that $f(p)\ne f(q)$.
\item[\rm (b)] For every compact set $K$ in $X$, its $\Ocal(X)$-convex
hull $\wh K_{\Ocal(X)}$ (see \eqref{eq:hull}) is the union of $K$ and all holes of $K$ in $X$;
in particular, $\wh K_{\Ocal(X)}$ is compact.
\end{enumerate}
\end{corollary}

Conditions (a) and (b) in Corollary \ref{cor:AB} were used in 1951 by K.\ Stein \cite{Stein1951}
to introduce the class of {\em Stein manifolds} of any dimension.
(The third of Stein's axioms is a consequence of these two.)
Thus, open Riemann surfaces are the same thing as $1$-dimensional Stein manifolds. 
Theorem \ref{th:Runge2} is a special case of the 
{\em Oka-Weil theorem} on Stein manifolds; see Sect.\ \ref{sec:OkaWeil}. 

The  proof of Runge's theorem in the plane is based on 
Cauchy's integral formula.  To prove Runge's theorem on open Riemann surfaces, 
Behnke and Stein constructed Cauchy type kernels, 
the so called \emph{elementary differentials}; see \cite[Theorem 3]{BehnkeStein1949}
and Remark \ref{rem:CauchykernelX} below where additional references are given.
More precisely, on any open Riemann surface $X$ there is a 
meromorphic 1-form $\omega$ on $X_z\times X_\zeta$ which is holomorphic off the 
diagonal and which in any pair of local coordinates has an expression 
\begin{equation}\label{eq:CauchykernelX}
	\omega(z,\zeta) = \left(\frac{1}{\zeta - z} + h(z,\zeta)\right) d\zeta,
\end{equation}
with $h$ a holomorphic function. (Note that $\omega$ is a form only in the second variable
$\zeta$, but its coefficient is a meromorphic function of both variables $(z,\zeta)$.)
In particular, $\omega$ has simple poles with residues one along the diagonal of $X\times X$. 
For any $\Ccal^1$-smooth domain $\Omega\Subset X$ and  $f\in\Ccal^1(\overline\Omega)$ 
one then obtains the Cauchy-Green formula
\begin{equation}
	f(z) = \frac{1}{2\pi \imath}\int_{\partial\Omega} f(\zeta)\omega(z,\zeta) 
	- \frac{1}{2\pi \imath} \int_{\Omega}\overline\partial f(\zeta)\wedge\omega(z,\zeta).
\end{equation}
By using this formula when $f$ is holomorphic on an open neighborhood of the set 
$K$ in Theorem \ref{th:Runge2}, one can approximate $f$ by meromorphic functions with poles
on $X\setminus K$, and the rest of the argument (pushing the poles) is similar to the one 
in Theorem \ref{th:Runge}.

Note that, just as in the complex plane, if we consider $(0,1)$-forms $\alpha$ with compact support 
in $\Omega$, we get that the mapping $\alpha\mapsto T(\alpha)$, given by
\begin{equation}\label{operatorbs}
	T(\alpha)(z) = -\frac{1}{2\pi \imath} \int_{\Omega}\alpha(\zeta) \wedge\omega(z,\zeta),
\end{equation}
is a bounded linear operator satisfying $\overline\partial(T(\alpha))=\alpha$.  This will be used below where we 
give a simple proof of Bishop's localization theorem.  

A functional analytic proof of Theorem \ref{th:Runge2} using Weyl's lemma  was given by 
B.\ Malgrange \cite{Malgrange1955} in 1955; see also O.\ Forster's monograph 
\cite[Sect.\ 25]{Forster1991}.

%
%
\begin{remark}\label{rem:CauchykernelX}
H.\ Behnke and K.\ Stein constructed Cauchy type kernels on relatively compact domains in any open Riemann surface  
\cite[Theorem 3]{BehnkeStein1949}; see also H.\ Behnke and F.\ Sommer \cite[p.\ 584]{BehnkeSommer1962}. 
The existence of globally defined Cauchy kernels (\ref{eq:CauchykernelX})
was shown by S.\ Scheinberg \cite{Scheinberg1978} and 
P.\ M.\ Gauthier \cite{Gauthier1979} in 1978-79. Their proof uses the theorem 
of R.\ C.\ Gunning and R.\ Narasimhan \cite{GunningNarasimhan1967} (1967) which says that
every open Riemann surface $X$ admits a holomorphic immersion $g:X\to \C$. The pull-back by $g$ 
of the Cauchy kernel  on $\C$ is a Cauchy kernel on $X$ with the correct behavior along the 
diagonal $D=\{(z,z):z\in X\}$ (see \eqref{eq:CauchykernelX}), but with additional poles if $g$ is 
not injective. Since the diagonal $D$ has a basis of Stein neighborhoods in $X\times X$
and its complement $X\times X\setminus D$ is also Stein, one can remove
the extra poles by solving a Cousin problem. Furthermore, Gauthier and Scheinberg
found Cauchy kernels satisfying the symmetry condition $F(p,q)=-F(q,p)$.
\qed\end{remark}

Theorem \ref{th:Runge2} implies the analogous approximation result 
for meromorphic functions. Indeed, we may write a meromorphic function $f$ 
on an open neighborhood $U\subsetneq X$ 
of the compact set $K$ as the quotient $f=g/h$ of two holomorphic functions
(this follows from the Weierstrass interpolation theorem on open Riemann surfaces; 
see \cite{Florack1948,Weierstrass1886}) and apply the same result separately to $g$ and $h$.
Since meromorphic functions are precisely holomorphic maps to the Riemann sphere $\CP^1=\C\cup\{\infty\}$,
this extension of Theorem \ref{th:Runge2} has the following corollary.

\begin{corollary}\label{cor:Rsphere}
Let $K$ be a compact set in an arbitrary Riemann surface $X$. Then, every holomorphic map
from a neighborhood of $K$ to $\CP^1$ may be approximated uniformly on $K$ by holomorphic maps $X\to\CP^1$. 
\end{corollary}

In 1958, E.\ Bishop \cite{Bishop1958PJM} proved the following extension of Mergelyan's theorem.

%
%

\begin{theorem}[Bishop-Mergelyan theorem; Bishop (1958), \cite{Bishop1958PJM}] 
\label{th:Mergelyan2}
If $K$ is a compact set without holes in an open Riemann surface $X$, 
then every function in $\Acal(K)$ can be approximated uniformly on $K$ by functions in $\Ocal(X)$.

More generally, if $X$ is an arbitrary Riemann surface, $\rho$ is a metric on $X$, and there is a $c>0$ 
such that every hole of a compact subset $K\subset X$ has $\rho$-diameter at least $c$,
then every function in $\Acal(K)$ is a uniform limit of meromorphic functions on $X$ with poles off $K$.
This holds in particular if $K$ has at most finitely many holes.
\end{theorem}

Bishop's proof depends on investigation of measures on $K$ annihilating 
the algebra $\Acal(K)$. This approach was further developed by L.\ K.\ Kodama \cite{Kodama1965} in 1965.
In 1968, J.\ Garnett observed \cite[p.\ 463]{Garnett1968} that Theorem \ref{th:Mergelyan2} 
can be reduced to Merge\-lyan's theorem on polynomial approximation (see Theorem \ref{th:Mergelyan1951})
by means of the following localization theorem due to Bishop \cite{Bishop1958PJM}
(see also \cite[Theorem 5]{Kodama1965}).

%
%

\begin{theorem}[Bishop's localization theorem; (1958), \cite{Bishop1958PJM}] \label{th:Bishop}
Let $K$ be a compact set in a Riemann surface $X$ and $f\in\Ccal(K)$.  
If every point $x\in K$ has a compact neighborhood $D_x \subset X$ such that
$f|_{K\cap D_x}\in \Ocalc(K\cap D_x)$, then $f\in \Ocalc(K)$.
\end{theorem}

Let us first indicate how Theorems \ref{th:Mergelyan1951}, \ref{th:Runge2}, and \ref{th:Bishop} 
imply Theorem \ref{th:Mergelyan2}. We cover $K$ by  open coordinate discs $U_1,\ldots, U_N$
of diameter at most $c$ (the number in the second part of Theorem \ref{th:Mergelyan2};
no condition is needed for the first part). Choose closed discs $D_j\subset U_j$
for $j=1,\ldots,N$ whose interiors still cover $K$. Then, $U_j\setminus (K\cap D_j)$ 
is connected. (Indeed, every relatively compact connected component of $U_j\setminus (K\cap D_j)$                                                                                        is also a connected component of $X\setminus D_j$ of diameter $<c$, contradicting the assumption.) Since $U_j$ is 
a planar set, Theorem \ref{th:Mergelyan1951}  
implies $\Acal(K\cap D_j)=\Ocalc(K\cap D_j)$. Thus, the hypothesis of Theorem \ref{th:Bishop}  is satisfied, 
and hence $\Acal(K)=\Ocalc(K)$. Theorem \ref{th:Mergelyan2} then follows from Runge's theorem 
(see Theorem \ref{th:Runge2}).

\smallskip
\noindent{\em Proof of  Theorem \ref{th:Bishop}.}
The following simple proof, based on solving the $\dibar$-equation, was given by A.\ Sakai  \cite{Sakai1972} in 1972.    
We may assume that $f$ is continuous in a neighborhood of $K$. Cover $K$ by finitely many neighborhoods $D_j$
as in the theorem, such that the family of open sets $\mathring D_j$ is an open cover of $K$.
Let $\chi_j$ be a partition of unity with respect to this cover.    
Now by the assumption we obtain for any $\epsilon>0$ functions $f_j\in\Ccal(D_j)\cap\Ocal(K\cap D_j)$
such that $\|f_j-f\|_{\Ccal(K\cap D_j)}<\epsilon$.  Set $g:=\sum_{j=1}^m\chi_j f_j$.  
Then on some open neighborhood $U$ of $K$ we have that $\|g-f\|_{\Ccal(U)}=O(\epsilon)$ and 
\[ 
	\overline\partial g = \sum_{j=1}^m \overline\partial\chi_j\cdot f_j = 
	\sum_{j=1}^m \overline\partial\chi_j\cdot (f_j-f) = O(\epsilon).
\] 
(We have used that $\sum_{j=1}^m \overline\partial\chi_j=0$ in a neighborhood of $K$.)
Let $\chi\in\Ccal^\infty_0(U)$ be a cut-off function with $0\leq\chi \leq 1$ and $\chi\equiv 1$ near $K$.  
Then we have $\|\chi\cdot\overline\partial g\|_{\Ccal(\overline U)}=O(\epsilon)$, and so 
$T(\chi\cdot\overline\partial g)=O(\epsilon)$, where $T$ is the Cauchy-Green operator \eqref{operatorbs}.  
Hence, the function $g-T(\chi\cdot\overline\partial g)$ is holomorphic on some open neighborhood of $K$
and it approximates $f$ to a precision of order $\epsilon$ on $K$.   
\qed

%
%
\begin{remark}\label{rem:Sakai}
Sakai's proof also applies to a compact set $K$ in a higher dimensional complex manifold,
provided $K$ admits a basis of Stein neighborhoods on which one can 
solve the $\dibar$-equation with uniform estimates with a constant independent of the neighborhood.
This holds for instance when $K$ is the closure of a strongly pseudoconvex domain;
see Theorem \ref{th:MergelyanWo} on p.\ \pageref{page:MergelyanWo}.
\qed \end{remark}

%
%
\begin{remark}\label{rem:Hoffman}
It was observed by K.\ Hoffman and explained  by J.\ Garnett \cite{Garnett1968} 
in 1968 that Bishop's localization theorem 
in the plane is a simple consequence of the properties of the Cauchy transform.  
Given a function $\phi\in\Ccal^\infty_0(\C)$ with compact support and a 
bounded continuous function $f$ on $\C$, we consider the {\em Vitushkin localization operator}:
\begin{equation}\label{eq:locop}
\begin{array}{ll}
	T_\phi(f)(z) &= \displaystyle{\frac{1}{2\pi\imath}\int_\C \frac{f(\zeta)-f(z)}{\zeta-z} \, \dibar\phi(\zeta)\wedge d\zeta} \\
		&=   \displaystyle{ f(z)\phi(z) + \frac{1}{\pi}\int_\C \frac{f(\zeta)}{\zeta-z} \frac{\di\phi}{\di\bar\zeta}(\zeta) \, dudv}.
\end{array}
\end{equation} 
(We used the Cauchy-Green formula \eqref{eq:Pompeiu}.)
From properties of the operator $T_K$ \eqref{eq:TK} we see that $T_\phi(f)$ is a bounded continuous function 
on $\C$ vanishing at $\infty$, it is holomorphic where $f$ is holomorphic and in $\C\setminus \supp(\phi)$, 
and $f-T_\phi(f)$ is holomorphic in the interior of the level set $\{\phi=1\}$.
If $f$ has compact support and $\{\phi_j\}_{j=1}^N$ is a partition of unity on $\supp(f)$,
then $f=\sum_{j=1}^N T_{\phi_j}(f)$.  Finally, it follows from \eqref{eq:Mest} that
\begin{equation}\label{eq:estTphi}
	\|T_\phi(f)\|_\infty \le c_0 \delta \omega_f(\delta) \|\di \phi/\di \bar\zeta\|_\infty,
\end{equation}
where $\delta>0$ is the radius of a disc containing the support of $\phi$, $\omega_f(\delta)$
is the $\delta$-modulus of continuity of $f$, and $c_0>0$ is a universal constant.
(See T.\ Gamelin \cite[Lemma II.1.7]{Gamelin1984} or D.\ Gaier \cite[p.\ 114]{Gaier1987} for the details.)

Suppose now that $f\colon K\to\C$ satisfies the hypothesis of Theorem  \ref{th:Bishop}. 
By Tietze's theorem we may extend $f$ to a continuous function with compact support on $\C$.
Let $U_1,\ldots,U_N\subset \C$ be a finite covering of $\supp(f)$ by bounded open sets such that, setting
$K_j=K\cap \overline U_j$, we have  $f|_{K_j}\in \Ocalc(K_j)$ for each $j$. 
Let $\phi_j\in \Ccal^\infty_0(\C)$ be a smooth partition 
of unity on $\supp(f)$ with $\supp(\phi_j)\subset U_j$. By the hypothesis, given $\epsilon>0$ 
there is a holomorphic function $h_j\in \Ocal(W_j)$ on an open neighborhood of $K_j$ which is uniformly 
$\epsilon$-close to $f$ on $K_j$. Shrinking $W_j$ around $K_j$, we may assume that 
$h_j$ is $2\epsilon$-close to $f$ on $W_j$. 
Choose a smooth function $\chi_j\colon \C\to [0,1]$ which equals
one on a neighborhood $V_j\subset W_j$ of $K_j$ and has $\supp(\chi_j)\subset W_j$.
The function $\tilde h_j=\chi_j h_j +(1-\chi_j)f$ then equals $h_j$ on $V_j$ (hence is holomorphic there),
it equals $f$ on $\C\setminus W_j$, and is uniformly $2\epsilon$-close to $f$ on $\C$. The function 
$g_j=T_{\phi_j}(\tilde h_j)\in \Ccal(\C)$ is holomorphic on $V_j$ (since $g_j$ is holomorphic there) and 
on $\C\setminus \supp(\phi_j)$. Since the union of the latter two sets contains 
$K$, $g_j$ is holomorphic in a neighborhood of $K$. Furthermore, 
$g_j$ approximates $f_j=T_{\phi_j}(f)$ in view of \eqref{eq:estTphi}.
The sum $\sum_{j=1}^N g_j$ is then holomorphic in a neighborhood of $K$ 
and uniformly close to $\sum_{j=1}^N f_j=f$ on $\C$. (Further details can be found 
in Gaier \cite[pp.\ 114--118]{Gaier1987}.)

By using the Cauchy type kernels in Remark \ref{rem:CauchykernelX}, 
P.\ Gauthier \cite{Gauthier1979} and S.\ Scheinberg \cite{Scheinberg1978} adapted this approach
to extend Bishop's localization theorem to  closed (not necessarily compact) sets 
of essentially finite genus in any Riemann surface. See also Sect.\ \ref{sec:unbounded}
and in particular Theorem \ref{cor:BJTh2}.

Another proof of Mergelyan's theorem on Riemann surfaces (Theorem \ref{th:Mergelyan2})
can be found in \cite[Chapter 1.11]{JarnickiPflug2000}. It is based on a proof of Bishop's localization theorem
(Theorem \ref{th:Bishop}) which avoids the use of Cauchy type kernels on Riemann surfaces,
such as those given by Behnke and Stein in \cite{BehnkeStein1949}.
%
%
\qed\end{remark}

%
%

After Mergelyan proved his theorem on polynomial approximation 
and Bishop extended it to open Riemann surfaces (Theorem \ref{th:Mergelyan2}),
a major challenging problem was to characterize the class of compact sets $K$ in $\C$, or 
in a Riemann surface $X$, which enjoy the Mergelyan property $\Acal(K) = \Ocalc(K)$. 
In view of Runge's theorem (Theorem \ref{th:Runge2}), this is equivalent to 
approximation of functions in  $\Acal(K)$ by meromorphic functions on $X$ with poles off $K$,
and by rational functions if $X=\C$:
\begin{equation}\label{eq:AR}
	\Acal(K) \ \stackrel{?}{=} \ \Rcal(K).
\end{equation}
The study of this question led to powerful new methods in approximation theory. 
There are examples of compact sets of {\em Swiss cheese} 
type (with a sequence of holes of $K$ clustering on $K$)
for which $\Rcal(K)\subsetneq\Acal(K)$; see D.\ Gaier \cite[p.\ 110]{Gaier1987}. 
An early positive result is the theorem of F.\ Hartogs and A.\ Rosenthal 
\cite{HartogsRosenthal1931} from 1931 which states that if $K$ is a compact set in $\C$ 
with Lebesgue measure zero, then $\Ccal(K)=\Rcal(K)$.
After partial results by S.\ N.\ Mergelyan \cite{Mergelyan1952,Mergelyan1954}, 
E.\ Bishop \cite{Bishop1958PJM,Bishop1960}  and others, the problem 
was completely solved by  A.\ G.\ Vitushkin in 1966, \cite{Vitushkin1966,Vitushkin1967}. To state his theorem, 
we recall the notion of continuous capacity. Let $M$ be a subset of $\C$. Denote by $\Rgot(M)$ the set of 
all continuous functions $f$ on $\C$ with $\|f\|_\infty\le 1$ which are holomorphic outside some compact 
subset $K$ of $M$ and whose Laurent expansion at infinity is
$f(z)= \frac{c_1(f)}{z} + O\bigl(\frac{1}{z^2}\bigr)$.
The {\em continuous capacity} of $M$ is defined by
\[
	\alpha(M) = \sup\bigl\{|c_1(f)| : f\in\Rgot(M)\bigr\}.
\]

%
%

\begin{theorem}[Vitushkin (1966/1967), \cite{Vitushkin1966,Vitushkin1967}] \label{th:Vitushkin}
Let $K$ be a compact set in $\C$. Then, $\Rcal(K)=\Acal(K)$ if and only if
$\alpha(D\setminus K)=\alpha(D\setminus \mathring K)$ for every open disc $D$ in $\C$.
\end{theorem}

Vitushkin's proof relies on the localization operators (\ref{eq:locop})
which he introduced (see \cite[Ch. 2, \S 3]{Vitushkin1967}). 
Theorem \ref{th:Vitushkin} is a corollary of Vitushkin's main result
in \cite{Vitushkin1967} which provides a criterium for rational approximation of individual 
functions in $\Acal(K)$. The most advanced form of Vitushkin-type results 
is due to Paramonov \cite{Paramonov1995}.
Major results on the behavior of the (continuous) capacity and estimates of 
Cauchy integrals over curves were obtained by M.\ Mel'nikov \cite{Melnikov1969,Melnikov1995}, 
X.\ Tolsa \cite{Tolsa2003,Tolsa2005}, and Mel'nikov and Tolsa \cite{MelnikovTolsa2005}.

%
%

\section{Approximation on unbounded sets in Riemann surfaces}
\label{sec:unbounded}

It seems that the first result concerning the approximation of functions on unbounded closed
subsets of $\C$ by entire functions is the following generalization of Weierstrass's Theorem \ref{th:Weierstrass}, 
due to T.\ Carleman \cite{Carleman1927}.

%
%
\begin{theorem}[Carleman (1927), \cite{Carleman1927}] \label{th:Carleman}
Given  continuous functions $f\colon\R\to\C$ and $\epsilon\colon \R\to (0,+\infty)$,
there exists an entire function $F\in\Ocal(\C)$ such that
\begin{equation}\label{eq:Carleman}
	|F(x)-f(x)|<\epsilon(x)\quad \text{for all}\ \ x\in\R.
\end{equation}
\end{theorem}

This says that continuous functions on $\R$ can be approximated 
in the fine $\Ccal^0$ topology by restriction to $\R$ of entire functions on $\C$. 
The proof amounts to inductively applying Mergelyan's theorem on polynomial approximation 
(Theorem \ref{th:Mergelyan1951}). 

\begin{proof}
Recall that $\overline{\mathbb D}=\{z\in\C: |z|\le 1\}$. For $j\in\Z_+=\{0,1,\ldots\}$ set
\[
	K_j=j\overline{\mathbb D} \cup [-j-2,j+2],\qquad \epsilon_j = \min\{\epsilon(x): |x|\le j+2\}.
\]
Note that $\epsilon_j\ge \epsilon_{j+1}>0$ for all $j\in\Z_+$. 
We construct a sequence of continuous functions $f_j: (j+1/3)\D \cup \R\to \C$ satisfying 
the following conditions for all  $j\in \N$: 
\begin{enumerate}[\rm (a$_j$)]
\item $f_j$ is holomorphic on $(j+1/3)\D$, 
\item $f_j(x)=f(x)$ for $x\in\R$ with $|x|\ge j+2/3$, and 
\item $|f_{j}-f_{j-1}| < 2^{-j-1}\epsilon_{j-1}$ on $K_{j-1}$.
\end{enumerate}
To construct $f_0$, we pick a smooth function $\chi\colon \R\to[0,1]$ such that $\chi(x)=1$ for $|x|\le1/3$ 
and $\chi(x)=0$ for $|x|\ge 2/3$. Mergelyan's theorem (see Theorem \ref{th:Mergelyan1951}) gives
a holomorphic polynomial $h$ such that, if we define $f_0$ to equal 
$h$ on $(1/3)\D$ and set $f_0(x)=\chi h(x) + (1-\chi)f(x)$ for $|x|\geq 1/3$, then 
$f_0$ satisfies conditions $(a_0)$ and $(b_0)$, while condition $(c_0)$ is vacuous.

The inductive step $(j-1)\to j$ is as follows.
Mergelyan's theorem (see Theorem \ref{th:Mergelyan1951}) gives a holomorphic polynomial $h$ satisfying 
$|h-f_{j-1}|< 2^{-j-1}\epsilon_{j-1}$ on $K_{j-1}$.
Pick a smooth function $\chi\colon \R\to[0,1]$ such that $\chi(x)=1$ for $|x|\le j+1/3$ 
and $\chi(x)=0$ for $|x|\ge j+2/3$.  Set $f_j=h$ on $(j+1/3)\D$ 
and $f_j=\chi h+(1-\chi)f_{j-1}$ on $\R$. It is easily verified that the sequence
$f_j$ satisfies conditions (a$_j$), (b$_j$), and (c$_j$). 
In view of (b$_j$) we have $f_0=f_1=\ldots= f_{k-1}$ on $\{|x|\ge k\}$ for any $k\in\N$. 
From this and (c$_j$) it follows that the sequence $f_j$ converges to an entire function 
$F\in\Ocal(\C)$ such that for every $k\in \Z_+$ the following inequality holds on $\{x\in\R:k\le |x| \le k+1\}$:
\[ 
	|F(x)-f(x)| \le \sum_{j=0}^\infty |f_{j+1}(x)-f_j(x)|  < \sum_{j=k-1}^\infty 2^{-j-2}\epsilon_j \le 
	\epsilon_{k-1} \le \epsilon(x).
\]
This proves Theorem \ref{th:Carleman}.  
\qed\end{proof}

The above proof is easily adapted to show that every function 
$f\in\Ccal^r(\R)$ for $r\in\N$ can be approximated in the fine $\Ccal^r(\R)$ topology by restrictions to $\R$ of
entire functions, i.e., \eqref{eq:Carleman} is replaced by the stronger condition on the derivatives:
\[
	|F^{(k)}(x)-f^{(k)}(x)|<\epsilon(x)\quad \text{for all}\ x\in\R\ \text{and}\ k=0,1,\ldots,r.
\]
In 1973, L.\ Hoischen \cite{Hoischen1973} proved a similar result on $\Ccal^r$-Carleman approximation on
more general curves in the complex plane. 

When trying to adapt the proof of Carleman's theorem to more general closed sets $E\subset \C$ without holes,
a complication appears in the induction step since the union of $E$ with a closed disc may contain holes. 
Consider the following notion.

%
%
\begin{definition}\label{def:Carlemanset}
Let $D$ be a domain in $\C$. A closed subset $E$ of $D$ is a {\em Carleman set}  if each function in
$\Acal(E)$ can be approximated in the fine $\Ccal^0$ topology on $E$ by functions in $\Ocal(D)$. 
(More precisely, given $f\in \Acal(E)$ and a continuous function $\epsilon\colon E\to (0,+\infty)$,
there exists $F\in\Ocal(D)$ such that $|F(z)-f(z)|<\epsilon(z)$ for all $z\in E$.)
\end{definition}

The following characterization of Carleman sets was given by A.\ A.\ Nersesyan in 1971, 
\cite{Nersesyan1971,Nersesyan1972}.
Given a domain $D\subsetneq \CP^1$, let $V_\epsilon(bD)$ denote the set of all points having 
chordal (spherical) distance less than $\epsilon$ from the boundary $bD$.

\begin{theorem}[Nersesyan (1971/1972), \cite{Nersesyan1971,Nersesyan1972}]\label{th:Nersesyan}
A closed set $E$ in a domain $D\subsetneq \CP^1$ is a Carleman set if and only if it satisfies 
the following two conditions.
\begin{enumerate}[\rm (a)]
\item For each $\epsilon>0$ there exists a $\delta$, with $0<\delta <\epsilon$, 
such that none of the components of 
$\mathring E$ intersects both $V_\delta(bD)$ and $D\setminus V_\epsilon(bD)$. 
\item For each $\epsilon>0$ there is a $\delta>0$ such that each point of the set 
$(D\setminus E)\cup V_\delta(bD)$ 
can be connected to $bD$ by an arc lying in $(D\setminus E)\cup V_\epsilon(bD)$. 
\end{enumerate}
\end{theorem}

We now look at the related  problem of
{\em uniform approximation} of functions in the space $\Acal(E)$ by holomorphic functions on $D$.
This type of approximation was considered by N.\ U.\ Arakelian \cite{Arakelian1964,Arakelian1968,Arakelian1971}
who proved the following result characterizing {\em Arakelian sets}.

%
%
\begin{theorem}[Arakelian (1964), \cite{Arakelian1964,Arakelian1968,Arakelian1971}]\label{th:Arakelian}
Let $E$ be a closed set in a domain $D\subset \C$. The following two conditions are equivalent.
\begin{enumerate}[\rm (a)]
\item Every function in $\Acal(E)$ is a uniform limit of functions in $\Ocal(D)$. 
\item The complement $D^*\setminus E$ of $E$ in the one point compactification $D^*=D\cup\{*\}$ of $D$
is connected and locally connected.
\end{enumerate}
\end{theorem}

When $E$ is compact, condition (b) simply says that $D\setminus E$ is connected,
and in this case, (a) is Mergelyan's theorem. Note that local connectivity of $D^*\setminus E$ is a nontrivial condition 
only at the point $\{*\}=D^*\setminus D$.  This condition has a more convenient interpretation.
For simplicity, we consider the case $D=\C$. Given a closed set $F$ in $\C$, we denote by 
$H_F$ the union of all holes of $F$, 
an open set in $\C$. (Recall that a hole of $F$ is a bounded connected components of $\C\setminus F$.)

%
%
\begin{definition} [Bounded exhaustion hulls property] \label{def:BEH}
A closed set $E$ in $\C$ with connected complement has the {\em bounded exhaustion hulls property} (BEH) 
if the set $H_{E\cup \Delta}$ is bounded (relatively compact) for every closed disc $\Delta$ in $\C$. 
\end{definition}

It is well known and easily seen that the BEH property of a closed subset $E\subset \C$ is equivalent to 
$\CP^1\setminus E$ being connected and locally connected at $\{\infty\}=\CP^1\setminus \C$.
Furthermore, this property may be tested on any sequence of closed discs
(or more general compact simply connected domains) exhausting $\C$. 
For the corresponding condition in higher dimensions, see Definition \ref{def:BEHn}
on p.\ \pageref{page:def:BEHn}.

We now present a simple proof of sufficiency of condition (b) for the case $D=\C$ in Arakelian's theorem, 
due to J.-P.\ Rosay and W.\ Rudin (1989), \cite{RosayRudin1989}. 

%
%
\smallskip
\noindent {\em Proof of (b)$\Rightarrow$(a) in Theorem \ref{th:Arakelian}.}
Since the set $E\subset\C$ has the BEH property (see Def.\ \ref{def:BEH}), we can find a sequence of closed discs
$\Delta_1\subset \Delta_2\subset \cdots \subset \bigcup_{i=1}^\infty \Delta_i=\C$ such that, setting
$H_i=H_{E\cup \Delta_i}$ (the union of holes of $E\cup \Delta_i$), we have that
\[
	\Delta_i\cup \overline H_i \subset \mathring \Delta_{i+1},\quad i=1,2,\ldots.
\]
Set $E_0=E$ and $E_i=E\cup\Delta_i\cup H_i$ for $i\in \N$. Note that 
$E_i$ is a closed set with connected complement in $\C$, $E_i\subset E_{i+1}$, 
$\bigcup_{i=0}^\infty E_i=\C$, and  $E\setminus \Delta_{i+1}=E_i\setminus \Delta_{i+1}$. 

Choose a function $f=f_0\in \Acal(E)$ and a number $\epsilon>0$. We shall inductively construct a 
sequence $f_i\in \Acal(E_i)$ for $i=1,2,\ldots$ such that $|f_i-f_{i-1}|<2^{-i}\epsilon$ on $E_{i-1}$;
since the sets $E_i$ exhaust $\C$, it follows that $F=\lim_{i\to \infty}f_i$ is an entire function 
satisfying $|F-f|<\epsilon$ on $E=E_0$. Let us explain the induction step $(i-1)\to i$.
Assume that $f_{i-1}\in\Acal(E_{i-1})$. Pick a closed disc $\Delta$ such that 
$\Delta_i\cup \overline H_i\subset \Delta \subset \mathring\Delta_{i+1}$, 
and a smooth function $\chi\colon \C\to [0,1]$ satisfying $\chi=1$ on 
$\Delta$ and $\supp(\chi)\subset \Delta_{i+1}$.
Note that $E_i\cup \Delta=E_{i-1}\cup \Delta$. Since the compact set $E_{i-1} \cap \Delta_{i+1}$ has no holes,
Mergelyan's Theorem \ref{th:Mergelyan1951} furnishes a holomorphic polynomial $h$ on $\C$ satisfying  
\[
	|f_{i-1}-h|  < 2^{-i-1}\epsilon\quad \text{on\ \ $E_{i-1}\cap \Delta_{i+1}$}, 
\]
and
\begin{equation}\label{eq:estimateRR}
	\frac{1}{\pi} \int_{\zeta\in E_{i-1}} |f_{i-1}(\zeta) -h(\zeta)| \cdotp |\dibar \chi(\zeta)| \frac{dudv}{|z-\zeta|} < 
	2^{-i-1}\epsilon, \quad z\in\C.
\end{equation}
Note that the integrand is supported on $E_{i-1}\cap (\Delta_{i+1}\setminus \Delta)$, 
and hence the integral is bounded uniformly on $\C$ by the supremum of the integrand (which may be as small as desired 
by the choice of $h$) and the diameter of $\Delta_{i+1}$. Let
\[
	g(z) = \frac{1}{\pi} \int_{\zeta\in E_{i-1}} (f_{i-1}(\zeta) - h(\zeta)) \cdotp \dibar \chi(\zeta) \frac{dudv}{z-\zeta},
	\quad z\in\C,
\]
and define the next function $f_i\colon E_{i}\cup \Delta \to\C$ by setting
\begin{equation}\label{eq:fi}
	f_i=\chi h + (1-\chi)f_{i-1} + g. 
\end{equation}
Note that $g$ is continuous on $E_i\cup \Delta$, smooth on $\mathring E_i\cup\Delta$, it satisfies $\dibar g= (f_{i-1}-h)\dibar\chi$
on $\mathring E_{i-1}\cup \mathring \Delta$, and $|g|<2^{-i-1}\epsilon$ in view of \eqref{eq:estimateRR}.
Since $E_i\cup \Delta=E_{i-1}\cup \Delta$, it follows that $f_i$ is continuous on $E_i$ and  
$\dibar f_i=(h-f_{i-1})\dibar\chi + \dibar g=0$ on $\mathring E_i$. Furthermore, on $E_{i-1}$ we have 
$f_i= f_{i-1} + \chi(h-f_{i-1})+g$ and hence
\[
	|f_i-f_{i-1}| \le |\chi|\cdotp |h-f_{i-1}| + |g| < 2^{-i}\epsilon.
\]
This completes the induction step and hence proves (b)$\Rightarrow$(a) in Theorem \ref{th:Arakelian}.
\qed \smallskip

Comparing with the proof of Theorem \ref{th:Carleman}, we see that it was now necessary to solve a
$\dibar$-equation since the set $E_{i-1}\cap (\Delta_{i+1}\setminus \Delta)$, on which we glued
the approximating polynomial $h$ with $f_{i-1}$, might have nonempty interior.
This prevents us from obtaining Carleman approximation in the setting of  Theorem \ref{th:Arakelian}
without additional hypotheses on $E$ (compare with Nersesyan's Theorem \ref{th:Nersesyan}).
On the other hand, the same proof yields the following special case of Nersesyan's Theorem
on Carleman approximation which is of interest in many applications.

%
%
\begin{corollary}[On Carleman approximation]\label{cor:Nersesyan}
Assume that $E\subset \C$ is a closed set with connected complement satisfying the BEH property
(see Definition \ref{def:BEH}). If there is a disc $\Delta\subset\C$ such that $E\setminus \Delta$ has empty interior,
then every function in $\Acal(E)$ can be approximated in the fine $\Ccal^0$ topology by entire functions. 
\end{corollary}

To prove Corollary \ref{cor:Nersesyan} one follows the proof of Theorem \ref{th:Arakelian},
choosing the first disc $\Delta_1$ big enough such that $E\setminus \Delta_1$ has empty interior.
This allows us to define each function $f_i$ \eqref{eq:fi} in the sequence 
without the correction term $g$ (i.e., $g=0$).

The definition of the BEH property (see Definition \ref{def:BEH}) 
extends naturally to closed sets $E$ in an arbitrary domain 
$\Omega \subset \C$. For such sets, an obvious modification of proof of 
Theorem \ref{th:Arakelian} and Corollary \ref{cor:Nersesyan} provide approximation of functions in 
$\Acal(E)$ by functions in $\Ocal(\Omega)$ in the uniform and fine topology on $E$, respectively.

%
%
In 1976, A.\ Roth \cite{Roth1976} proved several results on uniform and Carleman approximation
of functions in $\Acal(E)$, where $E$ is a closed set in a domain $\Omega\subset\C$, 
by meromorphic functions on $\Omega$ without poles on $E$. Her results are based on the technique of 
{\em fusing rational functions}, given by the following lemma.

%
%
\begin{lemma}[Roth (1976), \cite{Roth1976}]\label{lem:fusing}
Let $K_1$, $K_2$ and $K$ be compact sets in $\CP^1$ with 
$K_1\cap K_2=\varnothing$. Then there is a constant $a=a(K_1,K_2)>0$ such that for any 
pair of rational functions $r_1,r_2$ with $|r_1(z)-r_2(z)|<\epsilon$ $(z\in K)$ 
there  is a rational function $r$ such that $|r(z)-r_j(z)|<a\epsilon$ for $z\in K\cup K_j$ for $j=1,2$.
\end{lemma}

The proof of this lemma is fairly elementary. In the special case of holomorphic functions, 
this amounts to the solution of a Cousin-I problem with bounds.
As an application, A.\ Roth proved the following result \cite[Theorem 1]{Roth1976} 
on approximation of functions in $\Acal(E)$ by meromorphic functions without poles on $E$.

%
%
\begin{theorem}[Roth (1976), \cite{Roth1976}] \label{th:Roth1}
Let $\Omega$ be open in $\C$, and let $E\subsetneq \Omega$ be a closed subset of $\Omega$.  
A function $f\in \Acal(E)$ may be uniformly approximated on $E$ by functions in
$\Mcal(\Omega)$ without poles on $E$ if and only if $f|_K\in\Rcal(K)$ 
for every compact $K\subset E$. 
\end{theorem}

The paper \cite{Roth1976} of A.\ Roth also contains results on tangential  and
Carleman approximation by meromorphic functions on closed subsets of planar domains.

The following result \cite[Theorem 2]{Roth1976} was proved by 
A.\ A.\ Nersesyan \cite{Nersesyan1972} for $\Omega=\C$;
this extends Vitushkin's theorem (Theorem \ref{th:Vitushkin}) to closed subsets of $\C$.

%
%
\begin{theorem}[Nersesyan (1972), \cite{Nersesyan1972}; Roth (1976), \cite{Roth1976}] \label{th:Roth2}
Let $E\subset \Omega$ be as in Theorem \ref{th:Roth1}.  
A necessary and sufficient condition that every function in $\Acal(E)$ can be approximated
uniformly on $E$ by meromorphic functions on $\Omega$ with poles off $E$
is that $\Rcal(E\cap K)=\Acal(E\cap K)$ holds for every closed disc $K\subset \Omega$. 
\end{theorem}

%
%

The results presented above have been generalized to
open Riemann surfaces to a certain extent, although the theory does not seem complete.
In 1975, P.\ M.\ Gauthier and W.\ Hengartner \cite{GauthierHengartner1975}  gave the following
necessary condition for uniform approximation.
(As before, $X^*$ denotes the one point compactification of $X$.)

%
%
\begin{theorem}\label{th:GauthierHengartner1975}
Let $E$ be a closed subset of a Riemann surface $X$. If every function in $\Ocalc(E)$
is a uniform limit of functions in $\Ocal(X)$, then $X^*\setminus E$ is connected and 
locally connected, i.e., $E$ is an Arakelian set.
\end{theorem}

However, an example in \cite{GauthierHengartner1975} shows that the converse does not hold
in general. In particular,  {\em Arakelian's Theorem \ref{th:Arakelian} cannot be fully generalized 
to Riemann surfaces.} Further examples to this effect can be found in \cite[p.\ 120]{BoivinGauthier2001}. 

The situation is rather different for {\em harmonic} functions: if $E$ is a closed Arakelian set 
in an open Riemann surface $X$ then every continuous function on $E$ which is
harmonic in the interior $\mathring E$ can be approximated uniformly on $E$ 
by entire harmonic functions on $X$ 
(see T.\ Bagby and P.\ M.\ Gauthier \cite[Corollary 2.5.2]{BagbyGauthier1988}).

In 1986, A.\ Boivin \cite{Boivin1986} extended Nersesyan's Theorem \ref{th:Nersesyan} 
to a characterization of sets of holomorphic Carleman approximation in open Riemann surfaces, 
and he provided a sufficient condition on sets of meromorphic Carleman approximation.

For Carleman approximation of harmonic functions, we refer to T.\ Bagby and P.\ M.\ Gauthier 
\cite[Theorem 3.2.3]{BagbyGauthier1988}.
Furthermore, in \cite{BoivinGauthierParamonov2002}, A.\ Boivin, P.\ Gauthier and P.\ Paramonov 
established new Roth, Arakelian and Carleman type theorems for solutions of a large class of 
elliptic partial differential operators $L$ with constant complex coefficients. 

We return once more to Bishop's localization theorem  (see Theorem \ref{th:Bishop}).
We have already mentioned (cf.\ Remark \ref{rem:CauchykernelX}) that in the late 1970's, 
P.\ M.\ Gauthier \cite{Gauthier1979} and S.\ Scheinberg \cite{Scheinberg1978}   
constructed on any open Riemann surface $X$ a meromorphic kernel $F(p,q)$ such that 
$F(p,q)=-F(q,p)$ and the only singularities of $F$ are simple poles with residues $+1$ on the diagonal. 
With this kernel in hand, they extended Bishop's localization theorem to  closed sets 
of essentially finite genus in any Riemann surface. (See also \cite{Sakai1972,Boivin1987}.)
The most precise results in this direction were obtained  by S.\ Scheinberg \cite{Scheinberg1979} in 1979.
Under certain restrictions on the Riemann surface $X$ and the closed set $E\subset X$, 
he completely described those sets $P\subset X\setminus E$ such that every function in $\Acal(E)$ 
may be approximated uniformly on $E$ by functions meromorphic on $X$ whose poles lie in $P$. 
His theorems provide an elegant synthesis of all previously known results of this type
and 
a  summary of localization results.

The following converse to Bishop's localization theorem on an arbitrary Riemann surface
was proved by A.\ Boivin and B.\ Jiang \cite{BoivinJiang2004} in 2004. 
Recall that a {\em closed parametric disc} in a Riemann surface $X$ is the inverse image 
$D=\phi^{-1}(\Delta)$ of a closed disc $\Delta \subset \phi(U)\subset \C$, 
where $(U,\phi)$ is a holomorphic chart on $X$.

%
%
\begin{theorem}[Boivin and Jiang (2004), Theorem 1 in \cite{BoivinJiang2004}]\label{th:localization-converse}
Let $E$ be a closed subset of a Riemann surface $X$. If $\Acal(E)=\Ocalc(E)$, 
then $\Acal(E\cap D)=\Ocalc(E\cap D)$ holds for every closed parametric disc $D\subset X$.
\end{theorem}

Their proof relies on Vitushkin localization operators \eqref{eq:locop}, adapted to Riemann surfaces by 
P.\ Gauthier \cite{Gauthier1979} and S.\ Scheinberg \cite{Scheinberg1979} 
by using the Cauchy kernels mentioned above. (See also Remark \ref{rem:CauchykernelX}.)

Note that Theorem \ref{th:localization-converse} generalizes one of the implications in Theorem \ref{th:Roth2}
to Riemann surfaces. A result of this kind does not seem available for compact sets in higher dimensional
complex manifolds. We shall discuss this question again in connection with the Mergelyan approximation
problem for manifold-valued maps (see Subsect.\ \ref{ss:Mergelyan-manifold},
in particular Definition \ref{def:SLMP} and Remark \ref{rem:SLMP}).

The following is an immediate corollary to Theorem \ref{th:localization-converse} and  
Bishop's localization theorem for closed sets in Riemann surfaces \cite{Gauthier1979,Scheinberg1979}. It provides 
an optimal version of Vitushkin's approximation theorem (see Theorem \ref{th:Vitushkin})
on Riemann surfaces.

%
%
\begin{theorem}[Boivin and Jiang (2004), Theorem 2 in \cite{BoivinJiang2004}] \label{cor:BJTh2}
Let $E$ be a closed subset of a Riemann surface $X$, and assume either that $E$ is weakly
of infinite genus (this holds in particular if $E$ is compact) or $\mathring E=\varnothing$. 
Then, the following are equivalent:
\begin{enumerate}
\item Every function in $\Acal(E)$ is a uniform limit of meromorphic functions on $X$ with poles off $E$.
\item For every closed parametric disc $D\subset X$ we have $\Acal(E\cap D)=\Ocalc(E\cap D)$.
\item For every point $x\in X$ there exists a closed parametric disc $D_x$ centred at $x$ such that
$\Acal(E\cap D_x)=\Ocalc(E\cap D_x)$.
\end{enumerate}
\end{theorem}

%
%

\section{Mergelyan's theorem for $\Ccal^r$ functions on Riemann surfaces}
\label{sec:Mergelyan-smooth}

In applications, one is often faced with the approximation problem for functions
of class $\Ccal^r$ $(r\in\N)$ on compact or closed sets in a Riemann surface. 
Such problems arise not only in complex analysis (for instance, in constructions of closed complex 
curves in complex manifolds, see \cite{DrinovecForstneric2007DMJ}, or in constructions 
of proper holomorphic embeddings of open Riemann surfaces into $\C^2$, see
\cite{ForstnericWold2009,ForstnericWold2013} and \cite[Chap.\ 9]{Forstneric2017E}), 
but also in  related areas such as the theory of minimal surfaces in Euclidean spaces $\R^n$ 
(see the recent survey \cite{AlarconForstneric2019JAMS}), the theory of
holomorphic Legendrian curves in complex contact manifolds 
(see \cite{AlarconForstneric2019IMRN,AlarconForstnericLopez2017CM}), and others.
In most geometric constructions it suffices to consider compact sets of the following type.

%
%
\begin{definition}[Admissible sets in Riemann surfaces] \label{def:admissible}
A compact set $S$ in a Riemann surface $X$ is {\em admissible} if  
it is of the form $S=K\cup M$, where $K$ is a finite union of pairwise disjoint compact domains 
with piecewise $\Ccal^1$ boundaries in $X$ and $M= S \setminus\mathring  K$ is a union 
of finitely many pairwise disjoint smooth Jordan arcs and closed Jordan curves meeting $K$ only in their endpoints 
(or not at all) and such that their intersections with the boundary $bK$ of $K$ are transverse.
\end{definition} 

Clearly, the complement $X\setminus S$ of an admissible set has at most finitely many 
connected components, and hence Theorem \ref{th:Mergelyan2} applies. 

A function $f\colon S=K\cup M\to\C$ on an admissibe set is said to be of class
$\Ccal^r(S)$ if $f|_K\in\Ccal^r(K)$ (this means that it is of class $\Ccal^r(\mathring K)$ and 
all its partial derivatives of order $\le r$ extend continuously to $K$) and $f|_M\in \Ccal^r(M)$.
Whitney's jet-extension theorem (see Theorem \ref{th:Whitney}) shows that any
$f\in \Acal^r(S)$ extends to a function $f\in \Ccal^r(X)$ which is {\em $\dibar$-flat} 
to order $r$ on $S$, meaning that 
\begin{equation}\label{eq:dibarflat}
	\lim_{x \to S}\, D^{r-1}(\dibar f)(x) = 0.	
\end{equation}
Here, $D^k$ denotes the total derivative of order $k$ (the collection of all partial derivatives
of order $\le k$). We define the $\Ccal^r(S)$ norm of $f$ as the maximum of derivatives of $f$ 
up to order $r$ at points $z\in S$, where for points $z\in M\setminus K$ we consider only the tangential derivatives. 
(This equals the $r$-jet norm on $S$ of a $\dibar$-flat extension  of $f$.)

We have the following approximation result for functions of class $\Acal^r$
on admissible sets in Riemann surfaces. 
Corollary \ref{cor:admissibleXY} in Subsection \ref{ss:Mergelyan-manifold}
gives an analogous result for manifold-valued maps.

%
%

\begin{theorem}[$\Ccal^r$ approximation on admissible sets in Riemann surfaces]  \label{th:admissibleRS}
If $S$ is an admissible set in a Riemann surface $X$, then every function $f\in \Acal^r(S)$ $(r\in\N)$ 
can be approximated in the $\Ccal^r(S)$-norm by meromorphic functions on $X$,
and by holomorphic functions if $S$ has no holes. 
\end{theorem}

\begin{proof}
We give a proof by induction on $r$, reducing it to $\Ccal^0$ approximation.
The result can also be proved by the method in the proof of Theorems \ref{th:MergelyanWo}
and \ref{th:MergelyanW} below. 

Pick an open neighborhood $\Omega\subsetneq X$ of $S$ such that there is a deformation 
retraction of $\Omega$ onto $S$. (It follows in particular that 
$S$ has no holes in $\Omega$.) It suffices to show that any function $f\in \Acal^r(S)$ 
can be approximated in $\Ccal^r(S)$ by functions holomorphic on $\Omega$; the conclusion then follows from 
Runge's theorem (Theorem \ref{th:Runge2}) and the Cauchy estimates. 
We may assume that $S$ (and hence $\Omega)$ is connected.
There are smooth closed oriented Jordan curves $C_1,\ldots,C_l\subset S$ generating
the first homology group $H_1(S,\Z)= H_1(\Omega,\Z) \cong \Z^l$ such that 
$C=\bigcup_{i=1}^l C_i$ is a compact Runge set in $\Omega$.
Let $\theta$ be a nowhere vanishing holomorphic $1$-form on $\Omega$. (Such $\theta$ exists 
by the Oka-Grauert principle, see \cite[Theorem 5.3.1]{Forstneric2017E}. 
Furthermore, by the Gunning-Narasimhan theorem \cite{GunningNarasimhan1967} there exists a 
holomorphic function $\xi\colon \Omega\to\C$ without critical points, and we may take $\theta=d\xi$.)
Consider the period map $P=(P_1,\ldots,P_l) \colon \Ccal(C)\to\C^l$ given by
\[
	P_i(h) = \int_{C_i}  h\theta,\qquad h\in \Ccal(C),\ \ i=1,\ldots, l.
\]
It is elementary to find continuous functions $h_1,\ldots,h_l\colon C\to\C$ such that $P_i(h_j)=\delta_{i,j}$ 
(Kronecker's delta). By Mergelyan's theorem (Theorem \ref{th:Mergelyan2}) we can approximate 
each $h_i$  uniformly on $C$ by a holomorphic function 
$g_i\in \Ocal(\Omega)$. Assuming that the approximations are close enough, the $l\times l$ matrix $A$ 
with the entries $P_i(g_j)$ is invertible. Replacing the vector $g=(g_1,\ldots,g_l)^t$ by $A^{-1}g$ we 
obtain $P_i(g_j)=\delta_{i,j}$. Fix an integer $r\in \Z_+$. 
Consider the function $\Phi \colon \Acal^r(S) \times S\times \C^l\to \C$ defined by
\[ 
	\Phi(h,x,t) = h(x) +\sum_{j=1}^l t_j g_j(x), 
\] 
where $h\in \Acal^r(S)$, $x\in S$, and $t=(t_1,\ldots,t_l) \in\C^l$. Then,
$
	P(\Phi(h,\cdotp,t)) = P(h)+\sum_{j=1}^l t_j P(g_j),
$ 
and hence
\[ 
	\frac{\di P_i(\Phi(h,\cdotp,t)}{\di t_j}\bigg|_{t=0} = P_i(g_j)=\delta_{i,j}, \qquad i,j=1,\ldots,l.
\] 
This period domination condition implies, in view of the implicit function theorem,  
that for every $h_0\in \Acal^r(S)$ the equation 
$
	P(\Phi(h,\cdotp,t))=P(h_0)
$ 
can be solved on $t=t(h)$ for all $h\in \Acal^r(S)$ near $h_0$, with $t(h_0)=P(h_0)$.

We can now prove the theorem by induction on $r\in\Z_+$. By Theorem \ref{th:Mergelyan2},
the result holds for $r=0$. Assume that $r\in\N$ and the theorem holds for $r-1$.
Pick $f\in \Acal^r(S)$. The function $f'(x):=df(x)/\theta(x)$ $(x\in S)$ then belongs to $\Acal^{r-1}(S)$. 
(At a point $x\in S\setminus K$ we understand $df(x)$ as the $\C$-linear extension to $T_x X$ of 
the differential of $f|_M$.)  Note that $P(f')=\bigl( \int_{C_j} df\bigr)_{\!\! j} = 0\in \C^l$. 
By the induction hypothesis, we can approximate $f'$  in $\Ccal^{r-1}(S)$ 
by holomorphic functions $h\in \Ocal(\Omega)$. If the approximation is close enough,
there is a $t=t(h)\in\C^l$ near $P(f')=0$  such that the holomorphic function
$\tilde h:=\Phi(h,\cdotp,t)$ on $\Omega$ satisfies $P(\tilde h)=0$. Fix a point $p_0\in S$ and define
$
	\tilde f(p)= \int_{p_0}^p \tilde h \theta 
$
for $p\in\Omega$. Since the holomorphic $1$-form $\tilde h \theta$ has vanishing period, the integral is
independent of the choice of a path of integration. If $p\in S$ then the path 
may be chosen to lie in $S$, and hence $\tilde f$ approximates $f$ in 
$\Ccal^r(S)$. This completes the induction step and therefore proves the theorem.
\qed\end{proof}

The following optimal approximation result for smooth functions on compact sets in $\C$
was proved by J.\ Verdera in 1986, \cite{Verdera1986PAMS}.

%
%

\begin{theorem}[Verdera (1986), \cite{Verdera1986PAMS}]\label{th:Verdera}
Let $K$ be a compact set in $\C$, and let $f$ be a compactly supported function in $\Ccal^r(\C)$, 
$r\in \N$, such that $\di f/\di \bar z$ vanishes on $K$ to order $r-1$ (see \eqref{eq:dibarflat}). 
Then, $f$ can be approximated in $\Ccal^r(\C)$ by functions which are 
holomorphic in neighborhoods of $K$.
\end{theorem}

Theorem \ref{th:Verdera} shows that the obstacles to rational approximation of
functions in $\Acal(K)$ in Vitushkin's theorem (see Theorem \ref{th:Vitushkin}) 
are no longer present when considering rational approximation of $\Ccal^r$ 
functions which are $\overline\partial$-flat of order $r$ for $r>0$. Results in the same direction, concerning rational approximation 
on compact sets in $\C$ 
in {L}ipschitz and H\"older norms, were obtained by A.\ G.\ O'Farrell during 1977--79, 
\cite{O'Farrell1977TAMS,O'Farrell1977,O'Farrell1979}.

Verdera's proof of Theorem \ref{th:Verdera} is somewhat simpler for $r\ge 2$ than for $r=1$. 
In the case $r\ge 2$, he follows Vitushkin's scheme for rational approximation, using
in particular the localization operators \eqref{eq:locop}; here is the outline. 
Fix a number $\delta>0$. Choose a covering of $\C$ by a 
countable family of discs $\Delta_j$ of radius $\delta$ such that every point $z\in\C$ 
is contained in at most 21 discs. Also, let $\phi_j\in\Ccal^\infty_0(\C)$ be a smooth 
function with values in $[0,1]$, with compact support contained in $\Delta_j$, such that 
$\sum_j\phi_j=1$  and $|D^k \phi_j|\le C\delta^{-k}$ for some absolute constant $C>0$. Set 
\[
	f_j(z)=T_{\phi_j}(f)(z)
	= \frac{1}{\pi} \int_\C \frac{f(\zeta)-f(z)}{\zeta-z} \frac{\di\phi_j(\zeta)}{\di\bar\zeta} \, dudv,\quad z\in\C.
\]
Then, $f_j$ is holomorphic on $\C\setminus \supp(\phi_j)$, $f_j=0$ if $\supp(f)\cap \Delta_j=\varnothing$,
and $f=\sum_j f_j$ (a finite sum). Let $g=\sum'_j f_j$ where the sum is over those indices $j$ for which 
$\Delta_j\cap K=\varnothing$ and $h=f-g=\sum''_j f_j$ is the sum over the remaining $j$'s. Thus, $g$ is holomorphic
in a neighborhood of $K$, and Verdera shows that the $\Ccal^r(\C)$ norm of $h$ goes to zero
as $\delta\to 0$. The analytic details are considerable, especially for $r=1$.

In conclusion, we mention that many of the results on holomorphic approximation, 
presented in this and the previous two sections, have been generalized to 
solutions of more general elliptic differential equations in various Banach space norms;
see in particular J.\ Verdera \cite{Verdera1986TAMS}, 
P.\ Paramonov and J.\ Verdera  \cite{ParamonovVerdera1994},
A.\ Boivin, P.\ Gauthier and P.\ Paramonov \cite{BoivinGauthierParamonov2002},
and P.\ Gauthier and P.\ Paramonov \cite{GauthierParamonov2018}.

%
%

\section{The Oka-Weil theorem and its generalizations}
\label{sec:OkaWeil}

The analogue of Runge's theorem (see Theorems \ref{th:Runge} and \ref{th:Runge2}) 
on Stein manifolds and Stein spaces is the following theorem due to K.\ Oka \cite{Oka1936} and 
A.\ Weil \cite{Weil1935}. All complex spaces are assumed to be reduced.

\begin{theorem}[The Oka-Weil theorem] \label{th:Oka-Weil}
If $X$ is a Stein space and $K$ is a compact $\Ocal(X)$-convex subset
of $X$, then every holomorphic function in an open neighborhood of $K$
can be approximated uniformly on $K$ by functions in $\Ocal(X)$.
\end{theorem}

\begin{proof}  
Two proofs of this result are available in the literature. The original one, due to K.\ Oka and A.\ Weil,
proceeds as follows. A compact $\Ocal(X)$-convex subset $K$ in a Stein space $X$ admits
a basis of open Stein neighborhoods of the form
\[
	P=\{x\in X: |h_1(x)|<1,\ldots, |h_N(x)|<1\}
\]
with $h_1,\ldots, h_N\in \Ocal(X)$. We may assume that the function $f\in \Ocal(K)$ 
to be approximated is holomorphic on $P$.  By adding more functions if necessary, we can ensure 
that the map $h=(h_1,\ldots,h_N)\colon X\to\C^N$ embeds $P$ onto a closed complex subvariety $A=h(P)$
of the unit polydisc $\D^N\subset \C^N$. Hence, there is a function $g\in \Ocal(A)$
such that $g\circ h=f$ on $P$. By the Oka-Cartan extension theorem
\cite[Corollary 2.6.3]{Forstneric2017E}, $g$ extends to a holomorphic function
$G$ on $\D^N$. Expanding $G$ into a power series and precomposing its Taylor polynomials
by $h$ gives a sequence of holomorphic functions on $X$ converging to $f$ uniformly on $K$.

Another approach uses the method of L.\ H{\"o}rmander for solving the $\dibar$-equation with 
$L^2$-estimates (see \cite{Hormander1965,Hormander1990}). We consider the case $X=\C^n$;
the general case reduces to this one by standard methods of Oka-Cartan theory.
Assume that  $f$ is a holomorphic function in a neighborhood $U\subset\C^n$ of $K$. 
Choose a pair of neighborhoods $W\Subset V\Subset U$
of $K$ and a smooth function $\chi\colon \C^n \to [0,1]$ such that $\chi=1$ on $\overline V$ and
$\supp(\chi) \subset U$. By choosing $W\supset K$ small enough, 
there is a nonnegative plurisubharmonic function $\rho\ge 0$ on $\C^n$ that 
vanishes on $W$ and satisfies $\rho\ge c>0$ on $U\setminus V$. Note that the smooth $(0,1)$-form 
\[
	\alpha=\dibar (\chi f)= f\,\dibar \chi = \sum_{i=1}^n \alpha_i \, d\bar z_i
\]
is supported in $U\setminus V$. H\"ormander's theory for the $\dibar$-complex 
(see \cite[Theorem 4.4.2]{Hormander1990}) 
furnishes for any $t>0$ a smooth function $h_t$ on $\C^n$ satisfying
\begin{equation}\label{eq:L2}
	\dibar h_t = \alpha \qquad  \text{and} \quad
	\int_{\C^n} \frac{ |h_{t}|^2}{(1+|z|^2)^2} \, e^{-t\rho} d\lambda
	\le \int_{\C^n} \sum_{i=1}^n |\alpha_i|^2 e^{-t\rho} d\lambda.
\end{equation}
(Here, $d\lambda$ denotes the Lebesgue measure on $\C^n$.)
As $t\to+\infty$, the right hand side approaches zero since $\rho\ge c>0$ on $\supp(\alpha) \subset U\setminus V$. 
Since $\rho|_W=0$, it follows that $\lim_{t\to 0} \|h_t\|_{L^2(W)}=0$. The interior elliptic estimates 
(see \cite[Lemma 3.2]{ForstnericLowOvrelid2001}) imply that $h_t|_K\to 0$ 
in $\Ccal^r(K)$ for every fixed $r\in\Z_+$. The functions 
\[
	f_t=\chi f-h_t : \C^n\longrightarrow \C
\] 
are then entire and converge to $f$ uniformly on $K$ as $t\to +\infty$.
\qed\end{proof}

We also have the following parametric version of the Cartan-Oka-Weil theorem 
which is useful in applications (see \cite[Theorem 2.8.4]{Forstneric2017E}).

%
%
\begin{theorem}[Cartan-Oka-Weil theorem with parameters] \label{th:COW} 
Let $X$ be a Stein space. Assume that $K$  is an $\Ocal(X)$-convex subset of $X$, 
$X'$ is a closed complex subvariety of $X$, and $P_0\subset P$ are compact Hausdorff spaces.
Let $f\colon P\times X\to \C$ be a continuous function such that
\begin{itemize}
\item[]
\begin{enumerate}
\item[\rm(a)] for every $p\in P$, $f(p,\cdotp)\colon X\to \C$ is holomorphic in a 
neighborhood of $K$ (independent of $p$) and $f(p,\cdotp)|_{X'}$ is holomorphic, and
\item[\rm(b)] $f(p,\cdotp)$ is holomorphic on $X$ for every $p\in P_0$.
\end{enumerate}
\end{itemize}
Then there exists for every $\epsilon>0$ a continuous function $F\colon P\times X \to \C$ satisfying  
the following conditions:
\begin{itemize}
\item[]
\begin{enumerate}
\item[\rm(i)]    $F_p=F(p,\cdotp)$ is holomorphic on $X$ for all $p\in P$,
\item[\rm(ii)]   $|F-f|<\epsilon$ on $P\times K$, and
\item[\rm(iii)]  $F=f$ on $(P_0\times X)\cup (P\times X')$.
\end{enumerate}
\end{itemize}
The same result holds for sections of any holomorphic vector bundle over $X$.
\end{theorem}

The proof can be obtained by any of the two schemes outlined above. For the second
one, note that there is a linear solution operator to the $\dibar$-problem \eqref{eq:L2}, 
and hence continuous dependence on the parameter comes for free.
One needs to include the interpolation condition into the scheme to take care of the interpolation condition (iii). 
We refer to \cite[Theorem 2.8.4]{Forstneric2017E} for the details.

A similar approximation theorem holds for sections of coherent analytic sheaves
over Stein spaces (see e.g.\ H.\ Grauert and R.\ Remmert \cite[p.\ 170]{GrauertRemmert1979}).

The extension of the Oka-Weil theorem to maps $X\to Y$ from a Stein space $X$ 
to more general complex manifolds $Y$ is the subject of {\em Oka theory}.
A complex manifold $Y$ for which the analogue of Theorem \ref{th:COW} holds
in the absence of topological obstructions is called an {\em Oka manifold}.
We discuss this topic in Subsect.\ \ref{ss:Oka}.

%
%
\section{Mergelyan's theorem in higher dimensions} \label{sec:Mergelyan}

As we have seen in Sects.\ \ref{sec:Runge}--\ref{sec:Mergelyan-smooth}, 
the Mergelyan approximation theory in the complex plane and on Riemann surfaces
was a highly developed subject around mid 20th century.
Around the same time, it became clear that the situation is much more complicated in higher dimensions.
For example, in 1955 J. Wermer \cite{Wermer1955} constructed an arc in $\mathbb C^3$ 
which fails to have the Mergelyan property.  This suggests that, in several variables, one has to be much more 
restrictive about the sets on which one considers Mergelyan type approximation problems.   

There are two lines of investigations in the literature: approximation on submanifolds of $\mathbb C^n$ of 
various degrees of smoothness, and approximation on closures of bounded pseudoconvex domains.  
In neither category the problem is completely understood, and even with these restrictions, 
the situation is substantially more complicated than in dimension one.
For example, R. Basener (1973), \cite{Basener1973} (generalizing a result of B.\ Cole (1968), \cite{Cole1968})
showed that Bishop's peak point criterium does not suffice even for smooth polynomially convex submanifolds 
of $\mathbb C^n$.
Even more surprisingly,  it was shown by K.\ Diederich and J.\ E.\ Forn\ae ss in 1976 \cite{DiederichFornaess1976}
that there exist bounded pseudoconvex domains with smooth boundaries in $\C^2$ on which the Mergelyan property fails. 
The picture for curves is more complete; see G.\ Stolzenberg \cite{Stolzenberg1966AM}, 
H.\ Alexander \cite{Alexander1979}, and P. \ Gauthier and E. \ Zeron \cite{GauthierZeron2002}.

In this section we outline the developments starting around the 1960's, 
give proofs in some detail in the cases of totally real manifolds and 
strongly pseudoconvex domains, and provide some new results on combinations of such sets.

%
%
\begin{definition}\label{def:totallyrealmfd}
Let $(X,J)$ be a complex manifold, and let $M\subset X$ be a $\Ccal^1$ submanifold.   
\begin{enumerate}[\rm (a)]
\item $M$ is \emph{totally real} at a point $p\in M$ if $T_pM\cap JT_pM=\{0\}$.  
If $M$ is totally real at all points, we say that $M$ is a totally real submanifold of $X$. 
\item  $M$ is a {\em stratified totally real submanifold} of $X$ if $M=\bigcup_{i=1}^l M_i$, 
with $M_{i}\subset M_{i+1}$ locally closed sets, such that $M_1$ and 
$M_{i+1}\setminus M_i$ are totally real submanifolds.  
\end{enumerate}
\end{definition}

We now introduce suitable types of sets for Mergelyan approximation. 
The following notion is a generalization of the one for Riemann surfaces 
in Definition \ref{def:admissible}. Recall that a compact set $S$ in a complex manifold $X$
is a {\em Stein compact} if $S$ admits a basis of open Stein neighborhoods in $X$.

%
%

\begin{definition}[Admissible sets]\label{def:admissible2}  
Let $S$ be a compact set in a complex manifold $X$.
\begin{enumerate}[\rm (a)]
\item   $S$ is {\em admissible} if it is  of the form $S=K\cup M$,
where $S$ and $K$ are Stein compacts and $M=S\setminus K$ is a totally real submanifold of $X$
(possibly with boundary).
\item $S$ is {\em stratified admissible} 
if instead $M=\bigcup_{i=1}^l M_i$ is a stratified totally real submanifold 
such that $S_i=K\cup M_i$ is compact for every $i=1,\ldots,l$.
\item An admissible set $S=K\cup M$ is {\em strongly admissible} if, in addition to the conditions in (a),
$K$ is the closure of a strongly pseudoconvex Stein domain, not necessarily connected.
\end{enumerate}
\end{definition}

%
%
%
\begin{remark}\label{rem:admissible}
We emphasize that, in the definition of an admissible set, it is the decomposition of $S$ into the 
union $K\cup M$ that matters, so one might think of them as pairs $(K,M)$ with the indicated properties.
We will show (see Lemma \ref{lem:KSteincompact}) that if in (a) only the set $S$ is assumed to be a Stein compact
(and $M$ is totally real), then $K$ is nevertheless automatically a Stein compact. 
It follows that if $S=K\cup M$ is a stratified admissible set and 
$M=\bigcup_{i=1}^m M_i$ is a totally real stratification, then the set 
$S_i=K\cup M_i$ is a  stratified admissible set 
for every $i$ (see Corollary \ref{cor:stratadmissible}). 
 \qed\end{remark}

%
%

\begin{remark}\label{rem:HWadmissible}
We wish to compare the class of admissible sets with those considered by 
L.\ H\"ormander and J.\ Wermer \cite{HormanderWermer1968}
and F.\ Forstneri\v c \cite[Sect.\ 3]{Forstneric2004AIF}, \cite[Sects.\ 3.7--3.8]{Forstneric2017E}.
A compact set $S$ in a complex manifold $X$ is said to be {\em holomorphically convex} if it admits a 
Stein neighborhood $\Omega\subset X$ such that $S$ is $\Ocal(\Omega)$-convex. 
Clearly, such $S$ is a Stein compact,
but the converse is false in general. Let us call a compact set $S=K\cup M$ an {\em HW set} 
(for H\"ormander and Wermer) if $S$ is holomorphically convex and $M=S\setminus K$ is a totally real 
submanifold of $X$. In the cited works, approximation results similar to those presented here
are proved on HW sets. Clearly, every HW set is admissible. By combining the 
techniques in the proof of Proposition \ref{prop:local} and Theorem \ref{th:admissible} one can 
prove the following partial converse.

%
%

\begin{proposition}\label{prop:HW-admissible}
If $S=K\cup M$ is an admissible set in complex manifold $X$ and  $U\subset X$ is a neighborhood of $K$,
there exists a Stein neighborhood $\Omega$ of $S$ such that
\[
	h(S):=\overline{\widehat S_{\Ocal(\Omega)} \setminus S} \, \subset\, U.
\]
\end{proposition}

Thus, taking $S'=\widehat S_{\Ocal(\Omega)}$, $K'=K\cup h(S)$ and $M'=M\setminus h(S)$,
we see that $S'=K'\cup M'$ is an HW set with $K'\subset U$. Thus, every
admissible set can be approximated from the outside by HW sets,  enlarging $K$ only slightly.

It was shown by L.\ H\"ormander and J.\ Wermer \cite{HormanderWermer1968} 
(see also \cite[Theorem 3.7.1]{Forstneric2017E}) that if $S=K\cup M$ is an HW set and 
$S'=K\cup M'$ is another compact set, with $M'$ a totally real submanifold,
such that $S\cap U=S'\cap U$ holds for some open neighborhood of $K$, then $S'$ is also
admissible (i.e., any such $S'$ is a Stein compact). In view of the above proposition, the same holds
true for admissible sets, i.e., this class is stable under changes of the totally real part which 
are fixed near $K$.
\qed\end{remark}

We will consider two types of approximations in higher dimensions.  On admissible sets $S=K\cup M$
we will consider Runge-Mergelyan approximation, \emph{i.e.}, we assume that the object we want to approximate 
(function, form, map, etc.) is holomorphic on a neighborhood of $K$ and continuous or smooth on $M$.  
On strongly admissible sets we will consider true Mergelyan approximation, assuming that the object 
to be approximated is of class $\Acal^r(S)$ for some $r\in\Z_+$.

%
%

\subsection{Approximation on totally real submanifolds and admissible sets}\label{ss:submanifolds}

In this section we present an optimal $\Ccal^k$-approximation result on totally 
real submanifolds.   With essentially no extra effort we get approximation results on 
stratified totally real manifolds and  on admissible sets 
(see Theorems \ref{th:admissible} and \ref{th:strat-admissible}).

There is a long history on approximation on totally real submanifolds, starting with J.\ Wermer \cite{Wermer1955} 
on curves and R.\ O.\ Wells \cite{Wells1969} on real analytic manifolds.  
The first general result on approximation on totally real  manifolds with various degrees of smoothness
is due to L.\ H{\"o}rmander and J.\ Wermer \cite{HormanderWermer1968}.  
Their work is based on $L^2$-methods for solving the $\overline\partial$-equation, and the passage from 
$L^2$ to $\Ccal^k$-estimates led to a gap between the order $m$ of smoothness
of the manifold $M$ on which the approximation takes place, and the order $k$ of the norm of 
the Banach space $\Ccal^k(M)$ in which the approximation takes place.  Subsequently, several 
authors worked on decreasing the gap between $m$ and $k$, introducing more precise integral kernel methods
for solving $\overline\partial$. The optimal result with $m=k$
was eventually obtained by M.\ Range and Y.-T.\ Siu \cite{RangeSiu1974}.  Subsequent
improvements were made by F.\ Forstneri\v{c}, E.\ L{\o}w and N.\ {\O}vrelid \cite{ForstnericLowOvrelid2001}
in 2001. They developed Henkin-type kernels  adapted to this situation and 
obtained optimal results on approximation of $\dibar$-flat functions in tubes around totally real manifolds
by holomorphic functions.  In 2009, B.\ Berndtsson \cite{Berndtsson2009} used $L^2$-theory to give a new approach 
to uniform approximation by holomorphic functions on compact zero sets of strongly plurisubharmonic 
functions.  A novel byproduct of his method is that, in the case of polynomial approximation, one gets 
a bound on the degree of the approximating polynomial in terms of the closeness of the approximation.  

We will not go into the details of the $L^2$ or the integral kernel approaches, but 
will instead present a method based on convolution with the Gaussian kernel which originates in the 
proof of Weierstrass's Theorem \ref{th:Weierstrass} on approximating continuous functions on $\R$ by 
holomorphic polynomials. This approach is perhaps the most elementary one and is particularly well suited 
for proving Runge-Mergelyan type approximation results with optimal regularity on
(strongly)  admissible sets. It seems that the first modern application 
of this method was made in 1981 by S.\ Baouendi and F.\ Treves 
\cite{BaouendiTreves1981} to obtain local approximation of Cauchy-Riemann (CR) functions
on CR submanifolds. The use of this method on totally real manifolds was developed 
further by P.\ Manne \cite{Manne1993} in 1993 to obtain 
Carleman approximation on totally real submanifolds (see also \cite{ManneWoldOvrelid2011}).

We define the bilinear form $\langle\cdotp,\cdotp\rangle$ on $\C^n$ by
\begin{equation}\label{eq:notation}
	\langle z,w\rangle = \sum_{i=1}^n z_i w_i,\qquad 
	z^2=\langle z,z\rangle = \sum_{i=1}^n z_i^2.
\end{equation}
We consider first the real subspace $\R^n$ of $\mathbb C^n$.  Recall that 
\[ 
	\int_{\R^n}e^{-x^2}dx =
	\int_{\R^n}e^{-\sum_{i=1}^n x_i^2 } dx_1\cdots dx_n  = \left(\int_\R e^{-t^2}dt\right)^n = \pi^{n/2}.
\] 
It follows that $\int_{\R^n}e^{-x^2/\epsilon^2}dx=\epsilon^n \pi^{n/2}$, so the family of functions
$\pi^{-n/2}\epsilon^{-n} e^{-x^2/\epsilon^2}$ is an approximate identity on $\R^n$.
Given $f\in\Ccal^k_0(\R^n)$, consider the entire functions
\begin{equation*}\label{eq:approximation}
	f_\epsilon(z) = \pi^{-n/2} \int_{\R^n}\frac{1}{\epsilon^n}f(x) e^{-(x-z)^2/ \epsilon^2} dx,
	\qquad z\in\C^n,\ \epsilon>0.
\end{equation*}
(See (\ref{eq:Gauss}) for $n=1$.)
It is straightforward to verify that $f_\epsilon\rightarrow f$ uniformly on $\R^n$, and 
by a change of variables $u=z-x$ we get convergence also in $\Ccal^k$ norm.   

It is remarkable that the same procedure gives local approximation in the $\Ccal^k$ norm on 
any totally real submanifold of class $\Ccal^k$. 
Recall that $\B^n_\R\subset\R^n$ is the unit ball and $\B^n_\R(\epsilon) = \epsilon \B^n_\R$
for any $\epsilon>0$.

%
%

\begin{proposition}\label{prop:local}
Let $\psi:\mathbb B_{\mathbb R}^n\rightarrow\mathbb R^n$ be a map of class $\Ccal^k$ $(k\in\N)$
with $\psi(0)=(d\psi)_0=0$, and set $\phi(x)=x+i\psi(x)\in\C^n$.  
Then there exists a number $0<\delta<1$ such that the following holds. 
Let $N\subset\mathbb B^n_{\mathbb R}$ be a closed set, and let 
$M=\phi(\mathbb B^n_{\mathbb R}(\delta)\cap N)\subset \C^n$ and 
$bM=\phi(b\mathbb B^n_{\mathbb R}(\delta)\cap N)\subset \C^n$.  
Given  $f\in\Ccal_0(M)$, there exists a family of entire functions 
$f_\epsilon\in\Ocal(\mathbb C^n)$, $\epsilon>0$, such that the following hold as $\epsilon\to 0$:
\begin{enumerate}[\rm (a)]
\item $f_\epsilon \rightarrow f$ uniformly on $M$, and
\item $f_\epsilon\rightarrow 0$ uniformly on 
$U=\{z\in\mathbb C^n : \mathrm{dist}(z,bM)<\eta\}$ for some $\eta>0$.
\end{enumerate}
Moreover, if $N$ is a $\Ccal^k$-smooth submanifold of $\B^n_\R$ and $f\in\Ccal^k_0(M)$, then
the approximation in (a) may be achieved in the $\Ccal^k$-norm on $M$.
\end{proposition}

\begin{remark}
This proposition is local.  However, 
Condition (b) and Cartan's Theorem B makes it very simple to globalize the approximation in the case that  
$M$ is a totally real piece of an admissible set (see Theorem \ref{th:admissible} below).  
\qed\end{remark}

\noindent \emph{Proof of Proposition \ref{prop:local}.}
Since functions on $N$ extend to $\mathbb B^n_{\mathbb R}$ in the appropriate classes, it is enough to prove the 
proposition in the case $N=\mathbb B^n_{\mathbb R}$.  

Note that $\phi'(x)=I+\imath \psi'(x)$.
We will need  (see H\"{o}rmander \cite[p.\ 85]{Hormander1976}) that 
if $A$ is a symmetric $n\times n$ matrix with positive definite real part, then    
\begin{equation}\label{eq:Hormander}
	\int_{\mathbb R^n} e^{-\langle Au,u\rangle} du = \pi^{n/2}(\det\, A)^{-1/2}.
\end{equation}
We shall use this with the matrix $A(x)=\phi'(x)^T\phi'(x)$ whose real part equals $\Re A(x) =  I-\psi'(x)^T \psi'(x)$. 
Since $\psi'(0)=0$, there is a number $0<\delta_0<1$ such that $\Re A(x)>0$ is positive
definite for all $x\in \B^n_\R(\delta_0)$, and $\psi$ is Lipschitz-$\alpha$ with $\alpha<1$ on $\mathbb B^n_{\mathbb R}(\delta_0)$.
By using a smooth cut-off function, we extend $\psi$ to $\R^n$ such that $\supp(\psi)\subset \B^n_\R$, 
without changing its values on $\B^n_\R(\delta_0)$. (This does not affect the lemma.)
We will show that the lemma holds for any number $\delta$ with $0<\delta<\delta_0$.

Set $M=\phi(\mathbb B^n_{\mathbb R}(\delta))$ and $M_0=\phi(\mathbb B^n_{\mathbb R}(\delta_0))$,
so $\overline M\subset M_0$. Given $f\in\Ccal^k_0(M)$, set
\begin{equation}\label{eq:M}
	f_\epsilon(z) = \frac{1}{\pi^{n/2} \epsilon^n} \int_{M} f(w) e^{-(w-z)^2/ \epsilon^2} dw, \quad z\in\C^n,
\end{equation}
where $dw=dw_1\ldots dw_n$.

We begin by showing that condition (b) holds by inspecting the integral kernel.
Writing $z=x+\imath y\in\C^n$ and $w=u+\imath v=\C^n$, we have that 
\[
	\big| e^{-(w-z)^2}\big| = e^{-\Re(w-z)^2} = e^{(y-v)^2-(x-u)^2}.
\]
For a fixed $w=u+\imath v$, let 
$	\Gamma_w=\{z=x+\imath y \in\C^n : (y-v)^2<(x-u)^2\}. $
On $\Gamma_w$, the function $e^{-(w-z)^2/\epsilon^2}$ clearly converges to zero as $\epsilon\to 0$.
Since the map $\psi$ is Lipschitz-$\alpha$ with $\alpha<1$ on $\mathbb B^n_{\mathbb R}(\delta_0)$, 
we see  that for every $w\in M_0$ we have that $M_0\setminus \{w\} \subset \Gamma_w$. Hence,
there exists and open neighborhood $U\subset\mathbb C^n$ of $bM$ with 
$U\subset\Gamma_w$ for all $w\in\supp (f)$. This establishes (b).

Let us now prove (a). Since the function $x\mapsto f(\phi(x))$ is supported in $\B^n_\R(\delta)$, we can 
extend the product of it with any other function on $\B^n_\R(\delta)$ to all of $\R^n$ by letting it vanish outside 
$\B^n_\R(\delta)$. Fix a point $z_0=\phi(x_0)\in M$ with $x_0\in \mathbb B^n_{\mathbb R}(\delta)$.  
Using the notation \eqref{eq:notation}, we have that
\begin{equation*}
  \begin{aligned}
       f_\epsilon(z_0) & = \pi^{-n/2} \int_{M} \frac{1}{\epsilon^n} f(w) e^{-(w-z_0)^2/\epsilon^2} \, dw \\ 
       & =  \pi^{-n/2} \int_{\R^n} \frac{1}{\epsilon^n} f(\phi(x)) \, e^{-(\phi(x)-\phi(x_0))^2/\epsilon^2} \det \phi'(x) \, dx  \\
       & =  \pi^{-n/2} \int_{\R^n}  f(\phi(x_0+\epsilon u)) \, e^{-\left(u + \imath(\psi(x_0+\epsilon u)-\psi(x_0))/\epsilon\right)^2}  
		\det \phi'(x_0 + \epsilon u) du. 
  \end{aligned}
\end{equation*}
(We applied the change of variable $x=x_0+\epsilon u$.) The Lipschitz condition on $\psi$ gives
\[
	\big|e^{-\left(u + \imath(\psi(x_0+\epsilon u)-\psi(x_0))/\epsilon\right)^2}\big|  \leq e^{-(1-\alpha) |u|^2} 
\]
for all $x_0\in \B^n_\R(\delta)$ and $0<\epsilon <\delta_0-\delta$. The dominated convergence theorem implies
\[
  \begin{aligned}
	\lim_{\epsilon\rightarrow 0} f_\epsilon(z_0) & =  
	\pi^{-n/2} \int_{\R^n}  f(\phi(x_0)) e^{- \langle \phi'(x_0)u,\phi'(x_0)u\rangle} \, \det\phi'(x_0)\, du \\
	 & =  \pi^{-n/2}\int_{\R^n} f(z_0) e^{-\langle \phi'(x_0)^T\phi'(x_0)u,u \rangle } \det\phi'(x_0)\, du  \\
	 & = f(z_0).
   \end{aligned}
\]
The last line follows from \eqref{eq:Hormander} applied with the matrix $A=\phi'(x_0)^T\phi'(x_0)$,
noting also that $\det A=\det\phi'(x_0)^2$. The estimates are clearly independent of 
$x_0\in \B^n_\R(\delta)$, and hence of $z_0\in M$, so the convergence $f_\epsilon\to f$ is uniform on $M$.

To get convergence in the $\Ccal^k$ norm, one replaces partial differentiation of the kernel 
in \eqref{eq:M} with respect to $z$ by partial differentiation of $f$ with respect to $w$ 
(see P.\ Manne \cite[p.\ 524]{Manne1993} for the details).
\qed 

\smallskip

We now globalize Proposition \ref{prop:local} to obtain the approximation of $\Ccal^k$ functions 
on totally real manifolds of class $\Ccal^k$ and on (stratified) admissible sets.

%
%

\begin{theorem}\label{th:admissible}
Let $S=K\cup M$ be an admissible set in a complex manifold $X$ (see Definition \ref{def:admissible2}), 
with $M$  a totally real submanifold (possibly with boundary) of class $\Ccal^k$.  
Then for any $f\in\Ccal^k(S)\cap\Ocal(K)$ there exists a sequence $f_j\in\Ocal(S)$ such that 
\begin{equation*}
	\lim_{j\rightarrow\infty}\|f_j - f\|_{\Ccal^k(S)} = 0.
\end{equation*}
\end{theorem}

\begin{proof}
We begin by considering the case when $\supp(f)$ is contained in the totally real manifold $M=S\setminus K$, 
that is, $\supp(f)\cap K=\varnothing$.
We cover the compact set $\supp(f)\subset M\setminus K$ by
finitely many open domains (coordinate balls)  $M_1,\ldots,M_m\subset M$ such that 
Proposition \ref{prop:local} holds for each $M_j$ and $\bigcup_{j=1}^m \overline M_j \subset M\setminus K$.
Let $\chi_j\in \Ccal^k_0(M_j)$ be a partition of unity on a neighborhood of $\supp(f)$,
so $f=\sum_j \chi_j f$. Clearly, it suffices to prove the theorem separately for each 
$\chi_j f$, so we assume without loss of generality that $f$ is compactly supported in $M_1$.  

Let $U\subset X$ be a neighborhood of $bM_1$ satisfying condition (b)
in Proposition \ref{prop:local}.   Let $B\subset X$ be an open set with $M_1\subset B$ and 
let $A\subset X$ be an open set containing $S\setminus M_1$, such that $A\cap B\subset U$.
Let $\Omega\subset X$ be a Stein neighborhood of $S$ with $\Omega\subset A\cup B$, 
and set $\Omega_A:=\Omega\cap A$ and $\Omega_B=\Omega\cap B$.
Consider the map $\Gamma:\Ocal(\Omega_A)\oplus\Ocal(\Omega_B)\rightarrow \Ocal(\Omega_A\cap\Omega_B)$
defined by $(f_A,f_B)\mapsto f_A-f_B$.  Then $\Gamma$ is surjective by Cartan's Theorem B, 
and so by the open mapping theorem, $\Gamma$ is an open mapping with respect to the Fr\'{e}chet 
topologies on the respective spaces.   Let now $f_\epsilon$ be a family as in Proposition \ref{prop:local}.
Then $f_\epsilon\rightarrow 0$ on $\Omega_A\cap\Omega_B$, so 
there is a sequence $F_\epsilon=(f_{A,\epsilon},f_{B,\epsilon})\in \Ocal(\Omega_A)\oplus\Ocal(\Omega_B)$
converging to zero in the Fr\'{e}chet topology, with $\Gamma(F_\epsilon)=f_\epsilon$.  Pick a sequence $\epsilon_j\rightarrow 0$, and 
set $f_j:= f_{\epsilon_j}+f_{B,\epsilon_j}$ on $\Omega_B$ and $f_j:=f_{A,\epsilon_j}$ on $\Omega_A$. 
The conclusion now follows by restricting $f_j$ to any domain $\Omega'$ with $S\subset\Omega'\Subset\Omega$.

It remains to consider the general case when the support of $f$ intersects $K$.
To this end, we note that what we have proved so far gives the following useful lemma.

%
%

\begin{lemma}\label{lem:KSteincompact}
If $S=K\cup M$ is a Stein compact in a complex manifold $X$, if $S\setminus K$ is totally real,
and $U\subset X$ is an open set containing $K$, then there exists a Stein neighborhood $\Omega\subset X$ of $S$ 
such that $\wh K_{\Ocal(\Omega)}\subset U$. In particular, $K$ is a Stein compact. 
\end{lemma}

\begin{proof}
For each point $p\in M\setminus K$ there is a disc $M_p\subset M\setminus K$ around $p$
on which Proposition \ref{prop:local} holds. 
As we have just shown, we may use Theorem \ref{th:admissible} to approximate a continuous 
function which is zero near $K$ and one at the point $p$, and so there exists a 
holomorphic function $f_p\in \Ocal(S)$ such that $|f_p|$ is as small as desired
on $K$ and $|f_p|>1/2$ on a neighborhood of $p$. By taking the sum $\rho=\sum_j |f_{p_j}|^2$ 
over finitely many such functions, we get a plurisubharmonic function $\rho\ge 0$ on a 
neighborhood $V$ of $S$ which is $>1$ on a neighborhood $W$ of the compact set 
$\overline{M\setminus U}$ and is close to $0$ on a neighborhood of $K$.
Note that $S\subset U\cup W$. By choosing a Stein neighborhood $\Omega$ of $S$
such that $\Omega \subset (U\cup W)\cap V$, it follows that $\wh K_{\Ocal(\Omega)}\subset U$.
\qed\end{proof}

We now conclude the proof of Theorem \ref{th:admissible}. Assume that the function $f\in\Ccal^k(S)$
to be approximated is holomorphic in an open set $U\supset K$. Choose a Stein neighborhood
$\Omega$ of $S$ as in Lemma \ref{lem:KSteincompact} such that $K_0:=\wh K_{\Ocal(\Omega)}\subset U$. 
Pick an $\Ocal(\Omega)$-convex compact set $K_1 \subset U$ 
containing $K_0$ in its interior. Choose  a smooth cut-off function $\chi$ supported
on $K_1$ such that $\chi=1$ on a  neighborhood $K_0$. 
By the Oka-Weil theorem (see Theorem \ref{th:Oka-Weil})
there is a sequence $g_j\in\mathcal O(\Omega)$ such that $g_j\rightarrow f$ 
uniformly on $K_1$ as $j\rightarrow\infty$.  Then,  we clearly have that
\[
	\tilde f_j:= \chi g_j + (1-\chi)f =
	g_j+ (1-\chi)(f-g_j) \rightarrow f \quad \text{as $j\rightarrow\infty$}
\]
in $\Ccal^k(S)$. As $g_j\in\Ocal(\Omega)$, it remains to approximate the functions $(1-\chi)(f-g_j)\in \Ccal^k(S)$
whose support does not intersect $K_0\supset K$, so we have our reduction.
\qed\end{proof}

For approximation on stratified admissible sets, we need the following.

\begin{corollary}\label{cor:stratadmissible}
Let $S=K\cup M$ be a stratified admissible set, with a totally real stratification 
$M=\bigcup_{i=1}^l M_i$.   Then $S_i:=K\cup M_i$ is a Stein compact (and hence
a stratified admissible set) for each $i=1,\ldots,l-1$. 
\end{corollary}

\begin{proof}
Note that the top stratum $\wt M:=M\setminus M_{l-1}$ is a totally real submanifold
and $S=S_{l-1}\cup \wt M$ is an admissible set.
Lemma \ref{lem:KSteincompact} implies that $S_{l-1}$ is a Stein compact.
The result now follows by a finite downward induction on $l$.
\qed\end{proof}

%
%
\begin{theorem}\label{th:strat-admissible}
If $S=K\cup M$ is a stratified admissible set in a complex manifold $X$, then
for any $f\in\Ccal(S)\cap\Ocal(K)$ there exists a sequence $f_j\in\Ocal(S)$ such that 
\begin{equation*}
	\lim_{j\rightarrow\infty}\|f_j - f\|_{\Ccal(S)} = 0.
\end{equation*}
\end{theorem}

\begin{proof}
By assumption there is a stratification $M=\bigcup_{i=1}^l M_i$ with $M_1$ and $M_{i+1}\setminus M_i$
totally real manifolds for $i=1,...,l-1$.  Let  $M_0=\varnothing$.
It suffices to apply Theorem \ref{th:admissible} successively
with $K_i=K\cup M_{i}$ and $S_i=K_i\cup M_{i+1}$ $(i=0,...,l-1)$.
\qed\end{proof}

\begin{remark}
It is not possible in general to obtain $\Ccal^k$-approximation on stratified totally real manifolds $M$, 
even if $M$ is itself $\Ccal^k$-smooth.  Suppose for instance that $M\subset\mathbb C^n$
is a $\Ccal^1$-smooth submanifold which is a Stein compact,  and which is totally real except at a point $p\in M$.
Then, $M$ has an obvious stratification by totally real manifolds, but it is clearly impossible to achieve
$\Ccal^1$-approximation at the point $p$ due to the Cauchy-Riemann equations.  However, one 
sees immediately that one may achieve $\Ccal^k$-approximation on compact subsets of each 
$M_{i+1}\setminus M_i$.

Theorem \ref{th:strat-admissible} holds in the more general case when $S=K\cup M$
is a Stein compact with $M=\bigcup_{i=1}^l M_i$ a {\em stratified totally real set}, meaning
that $M_1$ and each difference $M_i\setminus M_{i-1}$ $(i=2,\ldots,l$) is a locally closed totally real set. 
We refer to P.\ Manne \cite{Manne1993PhD,Manne1993} and to
E.\ L{\o}w and E.\ F.\ Wold \cite{LowWold2009} for these extensions.
\qed\end{remark}

A further generalization of Theorem \ref{th:admissible} is provided by 
Theorem \ref{th:approximation-handlebody} in Sect.\ \ref{sec:manifold}; we state it there as it
concerns manifold-valued maps.

%
%

Although holomorphically convex smooth submanifolds of $\mathbb C^n$ do not in general admit 
Mergelyan approximation, E.\ L.\ Stout \cite{Stout2006PAMS} gave the following 
general result in the real analytic setting, also allowing for varieties. 

\begin{theorem}[E.\ L.\ Stout (2006), \cite{Stout2006PAMS}]\label{th:Stout2}
Let $X$ be a Stein space. If $M\subset X$ is a compact real analytic subvariety 
such that $M=\mathrm{spec}\,\Ocal(M)$, then $\Ccal(M)=\overline\Ocal(M)$.
\end{theorem}

Recall that $M=\mathrm{spec}\,\Ocal(M)$ means that any continuous character on 
the algebra $\overline{\mathcal O}(M)$ may be represented by a unique point
evaluation on $M$. An example is if $M$ is a countable intersection of Stein domains.  
We will not give a proof of the full theorem here, but we will use Theorem \ref{th:strat-admissible}, together 
with some fundamental results due to K.\ Diederich and J.\ E.\ Forn{\ae}ss \cite{DiederichFornaess1978} and
E.\ Bishop \cite{Bishop1965}, to give a relatively short proof under the stronger assumption that 
$M$ is a Stein compact.

\smallskip
\noindent \emph{Proof of Theorem \ref{th:Stout2} under the assumption that $M$ is a Stein compact.}
Without loss of generality we may assume that $M\subset\mathbb C^n$.  It was 
proved by K.\ Diederich and J.-E.\ Forn\ae ss \cite{DiederichFornaess1978} that $M$ does not contain 
a nontrivial analytic disc.   Now, $M$ has a stratification $M=\bigcup_{i=1}^m M_i$ such 
that each difference $V_i=M_{i+1}\setminus M_i$ is a real analytic submanifold.   We claim 
that every $V_i$ is totally real outside a real analytic submanifold $\widetilde V_i$
of positive codimension.  If not, there is an open subset $U\subset V_i$
such that $U$ is a CR-manifold, and according to Bishop \cite{Bishop1965} one may attach 
families of holomorphic discs to $U$ shrinking towards any given point $p\in U$.
By the assumption about holomorphic convexity, 
the discs will eventually be contained in $U$, but this contradicts 
the result of Diederich and Forn{\ae}ss \cite{DiederichFornaess1978}.  This argument may be 
used repeatedly  to refine the initial stratification of $M$ into a stratification by totally real submanifolds,
and hence the result follows from Theorem \ref{th:strat-admissible}.
\qed

%
%

\subsection{Approximation on strongly pseudoconvex domains and on strongly admissible sets}
\label{ss:Mergelyan-domains}

As we have seen, proofs of the Mergelyan theorem in one complex variable depend heavily on 
integral representations of holomorphic or $\overline\partial$-flat functions.  The single most
important reason why the one-dimensional proofs work so well is that the Cauchy-Green kernel
\eqref{eq:Pompeiu} provides a solution to the inhomogeneous $\dibar$-equation which is uniformly bounded 
on all of $\C$ in terms of sup-norm of the data and the area of its support (see \eqref{eq:Mest}). 
This allows uniform approximation of functions in $\Acal(K)$ on any compact set $K\subset \C$
with not too rough boundary by functions in $\Ocal(K)$ (see  Vitushkin's Theorem \ref{th:Vitushkin}).
Nothing like that holds in several variables, and the question of uniform approximability is highly
sensitive to the shape of the boundary even for smoothly bounded domains. 

The idea of developing holomorphic integral kernels for domains in $\C^n$ with
comparable properties to those of the Cauchy kernel in one variable was promoted by
H.\ Grauert already around 1960; however, it took almost a decade to be
realized. The first such constructions were given in 1969 by G.\ M.\ Henkin \cite{Henkin1969}
and E.\ Ram{\'\i}rez de Arellano \cite{Arellano1969} for the class of strongly pseudoconvex domains.
These kernels provide solution operators for the $\overline\partial$-equation which are
bounded in the $\Ccal^k$ norms and  improve the regularity by $1/2$.
We state here a special case of their results for $(0,1)$-forms, 
but in a more precise form which can be found in 
the works by I.\ Lieb and M.\ Range \cite[Theorem 1]{LiebRange1980},
I.\ Lieb and J.\ Michel \cite{LiebMichel2002}, and \cite[Theorem 2.7.3]{Forstneric2017E}.
A brief historical review of the kernel method is given in \cite[pp.\ 392--393]{ForstnericLowOvrelid2001}.

Given a domain $\Omega\subset \C^n$, we denote by $\Ccal^k_{(0,1)}(\overline\Omega)$
the space of all differential $(0,1)$-forms of class $\Ccal^k$ on $\overline\Omega$.

\begin{theorem}\label{th:dibar}
If $\Omega$ is a bounded strongly pseudoconvex Stein domain with boundary 
of class $\mathcal C^k$ for  some $k\in\{2,3,...\}$ in a complex manifold $X$, 
there exists a bounded linear operator 
$T:\Ccal^0_{(0,1)}(\overline\Omega)\rightarrow\Ccal^{0}(\overline\Omega)$
satisfying the following properties:
\begin{enumerate}[\rm (i)]  
\item If $f\in\mathcal{C}^0_{0,1}(\overline\Omega)\cap\mathcal{C}^1_{0,1}(\Omega)$
and $\overline\partial f=0$, then $\overline\partial (Tf)=f$.
\item If $f\in\mathcal{C}^0_{0,1}(\overline\Omega)\cap\mathcal{C}^r_{0,1}(\Omega)$
for some $r\in\{1,...,k\}$ then 
\[
	\|Tf\|_{\mathcal C^{l,1/2}(\overline\Omega)}
	\leq C_{l,\Omega}\|f\|_{\mathcal C^l_{0,1}(\overline\Omega)}, \quad l=0,1,...,r.
\]
\end{enumerate}
Moreover, the constants $C_{l,\Omega}$ may be chosen uniformly for all domains 
sufficiently $\mathcal C^k$ close to $\overline\Omega$. 
\end{theorem}

The kernel method led to a variety of applications to function theory on strongly pseudoconvex domains.
In particular, G. Henkin (1969)  \cite{Henkin1969}, N.\ Kerzman (1971) \cite{Kerzman1971}, and I.\ Lieb (1969) 
\cite{Lieb1969} proved the Mergelyan property for strongly pseudoconvex
domains with sufficiently smooth boundary, and  J.\ E.\ Forn\ae ss (1976) \cite{Fornaess1976AJM}
improved this to domains with $\Ccal^2$ boundary. Subsequently, 
J.\ E.\ Forn\ae ss and A.\ Nagel (1977) \cite{FornaessNagel1977}
showed that the Mergelyan property holds in the presence of transverse holomorphic vector fields near the 
set of weakly pseudoconvex boundary points (the so called {\em degeneracy set}); this holds  in particular for 
any bounded pseudoconvex domain with real analytic boundary in $\mathbb C^2$.
F.\ Beatrous and M.\ Range (1980) \cite{BeatrousRange1980} proved for holomorphically convex
domains $\Omega\Subset \C^n$ with $\Ccal^2$ boundaries that a function $f\in\Acal(\Omega)$ 
can be uniformly approximated by functions in $\Ocal(\overline\Omega)$ 
if it can be approximated on a  neighborhood of the degeneracy set.   This result 
appeared earlier in the thesis of F.\ Beatrous (1978).  

On the other hand, K.\ Diederich and J.\ E.\ Forn\ae ss (1976) \cite{DiederichFornaess1976} 
found an example of a $\Ccal^\infty$ smooth pseudoconvex domain $\Omega \subset \C^2$
for which the Mergelyan property fails.  Their example is based on the presence of
a Levi-flat hypersurface in $b\Omega$ having an annular leaf with infinitesimally
nontrivial holonomy. This phenomenon was further explored by D.\ Barrett \cite{Barrett1992}
who showed in 1992 that the Bergman projection on certain Diederich-Forn\ae ss 
worm domains does not preserve smoothness as measured 
by Sobolev norms. In 1996, M.\ Christ \cite{Christ1996} obtained a substantially stronger 
result to the effect that the Bergman projection on such domains $\Omega$ does not even preserve 
$\Ccal^\infty(\Omega)$; this provided the first example of smoothly bounded pseudoconvex domains
in $\C^n$ on which the Bell-Ligocka Condition R fails.

In 2008, F.\ Forstneri{\v c} and C.\ Laurent-Thiebaut proved the Mergelyan property 
for smoothly bounded pseudoconvex domains $\Omega\Subset\C^n$ whose degeneracy set
consisting of weakly pseudoconvex boundary points $A\subset b\Omega$   
is of the form $A=\{z\in M: \rho(z)\le 0\}$, where $\rho$ is a strongly plurisubharmonic 
function in a neighborhood of $A$, $M\subset \C^n$ is a Levi-flat hypersurface 
whose Levi foliation is defined by a closed $1$-form, and $A$ is the closure of its relative 
interior in $M$ (see \cite[Theorem 1.9]{ForstnericLaurent2008}). 
The paper \cite{ForstnericLaurent2008} provides several sufficient conditions for a 
foliation to be defined by a closed $1$-form. This condition implies in particular 
that every leaf of $M$ is topologically closed and has trivial holonomy map.
On the other hand, in the worm hypersurface of  Diederich and Forn{\ae}ss 
\cite{DiederichFornaess1976} the foliation contains a leaf with nontrivial holonomy
to which other leaves spirally approach. In \cite{ForstnericLaurent2008} it was shown 
under the same hypotheses  that the $\dibar$-Neumann operator on $\Omega$ is regular. 
See E.\ Straube and M.\ Sucheston \cite{StraubeSucheston2002,StraubeSucheston2003}
for related results.

We begin with the proof of the Mergelyan property on strongly pseudoconvex domains, taking for granted the 
existence of bounded solution operators for the $\overline\partial$-equation in Theorem \ref{th:dibar}.  
The proof we present here is similar to Sakai's proof \cite{Sakai1972} discussed already 
in the proof of Theorem \ref{th:Bishop} (see Remark \ref{rem:Sakai}).

%
%

\begin{theorem}\label{th:MergelyanWo}
Let $X$ be a Stein manifold, and let $\Omega\subset X$ be a relatively compact 
strongly pseudoconvex domain of class $\mathcal C^k$ for $k\in\{2,3,...\}$. Then for any $f\in \Ccal^k(\overline\Omega)\cap\mathcal O(\Omega)$
$(k\in\Z_+)$  there exists a sequence of functions $f_m\in\mathcal O(\overline\Omega)$ such that 
$\lim_{m\rightarrow\infty} \|f_m - f\|_{\Ccal^{k}(\overline\Omega)} = 0$.
\end{theorem}

\label{page:MergelyanWo}

\begin{proof} 
Let $\rho\in\Ccal^2(U)$ be a defining function for $\Omega$ in an open set $U\supset \overline\Omega$, 
so $\Omega=\{\rho<0\}$ and $d\rho\ne 0$ on $b\Omega=\{\rho=0\}$. 
For small $t>0$, set $\Omega_t=\{\rho<t\}\subset U$ and $\overline \Omega_t=\{\rho\le t\}$.
We cover $b\Omega$ by finitely many pairs of open sets $W_j\subset V_j$, $j=1,...l,$
with flows $\phi_{j,t}(z)$ of holomorphic vector fields defined on $V_j$, such that 
\begin{equation}\label{eq:flows}
	\text{$\phi_{j,t}(W_j\cap\Omega_t)\subset\Omega$\ \ for all small $t>0$.}
\end{equation}
Such vector fields are obtained easily in local coordinates, using constant vector fields pointing 
into $\Omega$ with a suitable scaling.  Set $W_0=\Omega$, and let 
$\{\chi_j\}_{j=0}^l$ be a smooth partition of unity with respect to the cover $\{W_j\}_{j=0}^l$.
Choose  $m_0\in \N$ such that $\overline\Omega_{1/m_0}\subset \bigcup_{j=0}^l W_j$
and \eqref{eq:flows} holds for all $0\le t\le 1/m_0$.
Note that the functions $\chi_j$ have bounded $\Ccal^{k+1}(\overline\Omega_{1/m_0})$ norms.
By Whitney's theorem (see Theorem \ref{th:Whitney}) 
we may assume that $f$ is extended to a $\Ccal^k$-smooth function 
on $\overline \Omega_{1/m_0}$. For any integer $m\ge m_0$ we set 
\begin{equation}\label{eq:Uj}
	U_{m,0}=\Omega,\quad U_{m,j}=\Omega_{1/m}\cap W_j\ \ \text{for}\ \ j=1,\ldots,l, 
\end{equation}
\begin{equation}\label{eq:fj}
 	f_{m,0}=f\ \ \text{on}\ U_{m,0}=\Omega,\quad f_{m,j}(z)=f(\phi_{j,1/m}(z)),\ \ z\in U_{m,j},\ j=1,\ldots,l.
\end{equation}
Consider the function 
\[
	g_m = \sum_{j=0}^l \chi_j\, f_{m,j} \in \Ccal^k(\overline \Omega_{1/m}). 
\]
From the definition of the functions $f_{m,j}$ \eqref{eq:fj} it follows that 
\begin{equation}\label{eq:estdiff}
	\|f_{m,j}-f\|_{\Ccal^k(\overline U_{m,j})}= \omega(1/m), \quad j=1,\ldots,l,
\end{equation}
where $\omega(1/m)\rightarrow 0$ as $m\rightarrow\infty$ (here $\omega(1/m)$ is proportional 
to the modulus of continuity of the top derivative of $f$), 
and hence $\|g_m-f\|_{\Ccal^k(\overline\Omega_{1/m})}=\omega(1/m)$.

We now estimate the $\Ccal^k$-norm of $\overline\partial g_m$.  
Since $\sum_{j=0}^l \overline\partial \chi_j=0$, we have that 
\[
	\overline\partial g_m = \sum_{j=0}^l f_{m,j} \, \overline\partial \chi_j
	= \sum_{j=0}^l  (f_{m,j} - f)\, \overline\partial\chi_j,
\]
and it follows from \eqref{eq:estdiff} that 
$\|\overline\partial g_m\|_{\Ccal^k(\overline\Omega_{1/m})}=\omega(1/m)$.

As explained above, there are bounded linear operators 
$T_m:\Ccal^k_{(0,1)}(\overline\Omega_{1/m})\rightarrow \Ccal^k(\overline\Omega_{1/m})$,
with bounds independent of $m\ge m_0$ and satisfying $\overline\partial T_m(\alpha)=\alpha$
for every $\dibar$-closed $(0,1)$-form $\alpha$ on $\overline\Omega_{1/m}$.
Setting $f_m= g_m - T_m(\overline\partial g_m)\in\Ocal(\Omega_{1/m})$ 
we get that $\|f_m-f\|_{\Ccal^k(\overline\Omega_{1/m})}=\omega(1/m)$, and this completes the proof.
\qed\end{proof}

The next result gives $\Ccal^k$-approximation on strongly admissible sets.  

\begin{theorem}\label{th:MergelyanW}
Let $X$ be a complex manifold. Assume that $\Omega\Subset X$ is a strongly pseudoconvex Stein domain
of class $\mathcal C^k$ for $k\in\{2,3,...\}$, and that $S=\overline\Omega \cup M \subset X$ is a strongly admissible set. Given $f\in \Ccal(S)\cap \Acal(\Omega)$ there is a sequence $f_j\in\mathcal O(S)$ such that 
$\lim_{j\rightarrow\infty} \|f_j - f\|_{\Ccal(S)} = 0$.
Furthermore, if $M$ is a totally real manifold of class $\Ccal^k$ and $f\in\Ccal^k(S)$, 
we may choose  $f_j\in\mathcal O(S)$ such that 
$\lim_{j\rightarrow\infty} \|f_j - f\|_{\Ccal^{k}(S)} = 0$.
\end{theorem}

\begin{proof}
We follow the proof of Theorem \ref{th:MergelyanWo}, but cover also $M$ with the $W_j$'s.  
By the theorem of Whitney and Glaeser (see Theorem \ref{th:WG} in the Appendix and the remark following it), 
we may extend $T_m(\dibar g_m)$ to $\Ccal^k$ functions $h_m$ on some neighborhood
of $S$, converging to $0$ in the $\Ccal^k$-norm. 
Hence, $\tilde f_m:=g_m - h_m$ is holomorphic on $\Omega_{1/m}$ and $\tilde f_m\rightarrow f$
in $\Ccal^k(\overline\Omega)$ as $m\to\infty$, and the result follows from Theorem \ref{th:admissible}.
\qed\end{proof}

\subsection{Mergelyan approximation in $L^2$-spaces}\label{ss:L2Mergelyan}

In his thesis from 2015, S.\ Gubkin \cite{Gubkin2015} investigated Mergelyan approximation
in $L^2$ spaces of holomorphic functions on pseudoconvex domains in $\C^n$:
\[
	H^2(\Omega)=\Ocal(\Omega)\cap L^2(\Omega).
\]
The following theorem generalizes both his main results \cite[Theorems 4.2.2 and 4.3.3]{Gubkin2015};
in the first one the domain is assumed to have $\Ccal^\infty$-smooth boundary, and in the second one
it is assumed to admit a $\mathcal C^2$ plurisubharmonic defining function.
We only assume that the closure of the domain is a Stein compact.

%
%
\begin{theorem}\label{th:MergelyanWo2}
Assume that $X$ is a Stein manifold and $\Omega\Subset X$ is a relatively compact pseudoconvex domain 
with  $\Ccal^1$  boundary whose closure $\overline\Omega$ is a Stein compact.  Then for any $f\in H^2(\Omega)$
there exists a sequence $f_j\in\mathcal O(\overline\Omega)$ such that 
$\lim_{j\rightarrow\infty} \|f_j - f\|_{L^{2}(\Omega)} = 0$.
\end{theorem}

\begin{proof} 
As in the proof of Theorem \ref{th:MergelyanWo}, we find an open cover $\{W_j\}_{j=0}^l$
of $\overline\Omega_{1/m_0}$ for some $m_0\in \N$ such that \eqref{eq:flows} holds.
(This only requires that $b\Omega$ is of class $\Ccal^1$.) Let $\{\chi_j\}_{j=0}^l$ be a smooth partition of unity 
subordinate to  $\{W_j\}_{j=0}^l$. Given an integer $m\ge m_0$
we define the cover $\{U_{m,j}\}_{j=0}^l$ and the functions $(f_{m,j})_{j=0}^l$ by \eqref{eq:Uj}
and \eqref{eq:fj}, respectively. Consider the function 
\[
	g_m = \sum_{j=0}^l \chi_j\, f_{m,j} \in L^2(\Omega_{1/m}).
\]
Fix $\delta>0$. Since $\|f\|_{L^2(\Omega)}<\infty$, there exists a compact subset $K\subset\Omega$  such that 
\begin{equation}\label{eq:estfK}
	\|f\|_{L^2(\Omega\setminus K)}<\delta. 
\end{equation}
Choose a compact set $K'\subset\Omega$ such that 
\begin{equation}\label{eq:KKprime}
	K\cup \supp (\chi_0) \subset \mathring K'. 
\end{equation}
Note that $g_m\rightarrow f$ in sup-norm on $K'$ as $m\rightarrow\infty$.
Furthermore, \eqref{eq:flows} and \eqref{eq:Uj} imply 
$\phi_{j,1/m}(U_{m,j}{\setminus K'}) \subset \Omega\setminus K$
for all big enough $m$, and hence \eqref{eq:fj} and \eqref{eq:estfK} give
\begin{equation}\label{eq:estfmj}
	\text{$\|f_{m,j}\|_{L^2(U_{m,j}\setminus K')}<2\delta$\ \ \ for all $m$ big enough and all $j=1,\ldots,l$.} 
\end{equation}
(The factor $2$ comes from a change of variable; note that $\phi_{j,t}\to \mathrm{Id}$ 
as $t\to 0$.) Since this holds for every $\delta>0$, we see that $\lim_{m\to\infty}\|g_m-f\|_{L^2(\Omega)}=0$.

Next, we need to estimate $\overline\partial g_m$ on $\Omega_{1/m}$. We have that 
\[
	\overline\partial g_m = \sum_{j=0}^l  f_{m,j}\, \overline\partial \chi_j =  
	\sum_{j=0}^l (f_{m,j}-f) \, \overline\partial \chi_j,
\]
where the second expression holds on $\Omega$.
It follows that $\lim_{m\to\infty} \|\overline\partial g_m\|_{L^2(K')} = 0$.
On $\Omega\setminus K'$ we have in view of \eqref{eq:KKprime} that 
$\overline\partial g_m = \sum_{j=1}^l  f_{m,j}\, \overline\partial \chi_j$, 
and hence \eqref{eq:estfmj} gives 
\[
	\|\overline\partial g_m\|_{L^2(\Omega_{1/m}\setminus K')} <C_0\delta
\]
for some constant $C_0>0$ depending only on the partition of unity $\{\chi_j\}$. 
Since $\delta>0$ was arbitrary, it follows that 
$\lim_{m\to\infty} \|\overline\partial g_m\|_{L^2(\Omega_{1/m})}=0$.

Set $\alpha_m=\overline\partial g_m$, and let $\Omega'$ be a pseudoconvex domain with 
$\overline\Omega\subset\Omega'\subset\Omega_{1/m}$.
By H\"{o}rmander, there exists a constant $C>0$, independent of $m$ and the choice of $\Omega'$, 
such that there exists a solution $h_m$ to the equation $\overline\partial h_m=\alpha_m$ with 
$\|h_m\|_{L^2(\Omega')}\leq C\cdot\|\alpha_m\|_{L^2(\Omega')}$.
Setting $f_m= g_m - h_m$ we get that $\lim_{m\to\infty}  \|f_m-f\|_{L^2(\Omega)}=0$.
\qed\end{proof}

\begin{remark}\label{rem:Hartogstriangle}
A simple example of a pseudoconvex domain on which the $L^2$ Mergelyan property fails 
is the Hartogs triangle $H=\{(z,w)\in \C^2 : |w|<|z|<1\}$. The holomorphic function $f(z,w)=w/z$ on $H$ is bounded 
by one, and it cannot be approximated in any natural sense by holomorphic functions in neighborhoods 
of $\overline H$ since its restriction to horizontal slices $w=const$ has winding number $-1$.
Note that $\overline H$ is not a Stein compact. One can also see that the $L^2$ Mergelyan property
fails on the Diederich-Forn{\ae}ss worm domain \cite{DiederichFornaess1976}.
\qed\end{remark}

We shall consider further topics in $L^2$-approximation theory in Sect.\ \ref{sec:weights}.

%
%

\subsection{Carleman approximation in several variables}\label{ss:CarlemanCn}

Carleman approximation on the totally real affine subspace $M=\mathbb R^n\subset\mathbb C^n$ 
was proved by S.\ Scheinberg \cite{Scheinberg1976} in 1976. 
Such spaces are obviously polynomially convex, and, although less obviously so, 
they satisfy the following condition (compare with Definition \ref{def:BEH}).
For any compact set $C\subset\mathbb C^n$ we set 
\[
	h(C):=\overline{\widehat C\setminus C}.
\]

\begin{definition}\label{def:BEHn}
A closed set $M\subset\C^n$ has the {\em bounded exhaustion hulls property}
if for any polynomially convex compact set $K\subset\mathbb C^n$
there exists $R>0$ such that for any compact set $L\subset M$ we have that
\begin{equation}\label{eq:beh}
	h(K\cup L)\subset\mathbb \B^n(0,R).
\end{equation}
\end{definition}

\label{page:def:BEHn}

Clearly, it suffices to test this condition on any increasing sequence of compact sets $K_j$ 
increasing to $\C^n$. This notion extends in an obvious way to closed sets in an arbitrary complex
manifold $X$, replacing polynomial hulls by $\Ocal(X)$-convex hulls.
For closed sets $M$ in $\mathbb C$, this notion is equivalent to the one in Definition \ref{def:BEH}, 
and to the condition that $\C\P^1\setminus M$ is locally connected at infinity.
(This is precisely the condition under which  Arakelian's Theorem \ref{th:Arakelian} holds.)

To see that $M=\mathbb R^n$ has bounded exhaustion hulls in $\C^n$, we consider compact sets of the form 
\[
	K_r=\bigl\{z\in\mathbb C^n : |x_j|\leq r,\ |y_j|\leq r, \ j=1,...,n\bigr\}.
\]
Let us first look at a point $\tilde z=\tilde x+\imath \tilde y\in\C^n\setminus \R^n$ with $|\tilde x_j|>(\sqrt n +1)r$ 
for some $j$.  Consider the pluriharmonic polynomial 
\[
	f(z)=-\Re ((z-\tilde x)^2) = \sum_{i=1}^n \bigl(y_i^2-(x_i-\tilde x_i)^2\bigr),\qquad z\in\C^n. 
\]
A simple calculation shows that $f(z)<0$ holds for any point $z\in K_r$, and we clearly have 
$f\leq 0$ on $\mathbb R^n$ and $f(\tilde z)=(\tilde y)^2>0$. This shows that 
\[
	h(K_r\cup\mathbb R^n)\subset \bigl\{z\in\mathbb C^n: |x_j|\leq (\sqrt n +1)r, \ \ j=1,...,n\bigr\}.
\]  
Clearly we also have $h(K_r\cup\mathbb R^n)\subset\{z\in\mathbb C^n: |y_j|\leq r, \ j=1,...,n\}$, 
and \eqref{eq:beh} follows.

By using Theorem \ref{th:admissible} it is easy to prove the following result, which 
by the argument just given implies Scheinberg's result in \cite{Scheinberg1976}.
Fix a norm on the jet-space $\Jcal^k(X)$, and denote it by $|\cdot|_{\Ccal^k(x)}$.
Recall that an unbounded closed set $M$ in a Stein manifold $X$ is 
called $\Ocal(X)$-convex if $M$ is exhausted by an increasing sequence of compact $\Ocal(X)$-convex sets.
   
%
%
\begin{theorem}[P.\ E.\ Manne (1993), \cite{Manne1993PhD}]
Let $X$ be a Stein manifold. If $M\subset X$ is a closed totally real submanifold of 
class $\Ccal^k$ that is holomorphically convex and has bounded exhaustion hulls, then
$M$ admits $\Ccal^k$-Carleman approximation by entire functions.
\end{theorem}

\begin{proof}
For simplicity of exposition we give the proof in the case $X=\C^n$.
Since $M$ has bounded exhaustion hulls, there exists a normal exhaustion
 $\{K_j\}_{j\in\mathbb N}$ of $\C^n$
by polynomially convex compact sets such that $K_j\cup M$
is polynomially convex for each $j\in\N$.   Choose a sequence $m_j\in\mathbb N$
such that $m_j<m_{j+1}$ and $K_j\subset\mathbb \B^n(0,m_j)$ for each $j$. 
Set $M_j=M\cap\overline{{\mathbb B}^n(0,m_j)}$, and choose a function 
$\chi_j\in\Ccal^{\infty}_0(\mathbb B^n(0,m_{j+1}))$ with $\chi_j\equiv 1$ near 
$\overline {\mathbb B^n(0,m_j)}$.  
To prove the theorem we proceed by induction, making the induction hypothesis 
that we have found $f_j\in\Ccal^k(M)\cap\mathcal O(K_j\cup M_j)$ such that 
\[
	|f_j - f|_{\Ccal^k(x)}<\epsilon(x)/2,\qquad x\in M.
\]
It will be clear from the induction step how to achieve this for $j=1$.  
Theorems \ref{th:admissible} and \ref{th:Oka-Weil} furnish a sequence 
$g_{j,m}\in\mathcal O(K_{j+1}\cup M_{j+2})$ such that 
\[
	\|g_{j,m}-f_j\|_{\Ccal^k(K_j\cup M_{j+2})}\rightarrow 0\ \ \text{as $m\rightarrow\infty$}.
\]
It follows that $f_{j+1}=  g_{j,m} + (1-\chi_{j+1})(f_j - g_{j,m})$ will reproduce the 
induction hypothesis for sufficiently large $m$, and we may furthermore
achieve $\|f_{j+1}-f_j\|_{K_j}<2^{-j}$.   It follows that $f_j$  converges
uniformly on compacts in $X$ to an entire function  approximating $f$ to the desired precision.  
\qed\end{proof}

Prior to Manne's result, H. Alexander \cite{Alexander1979} generalized Carleman's
theorem \cite{Carleman1927} to smooth unbounded curves in $\mathbb C^n$ in 1979.  
By a fundamental work of G.\ Stolzenberg 
\cite{Stolzenberg1966AM}, such a curve is always polynomially convex and has 
bounded exhaustion hulls.   In 2002 P.\ M.\ Gauthier and E.\ Zeron \cite{GauthierZeron2002} 
improved Alexander's result to include  locally rectifiable curves with trivial topology.  

The situation is rather different for higher dimensional totally real manifolds. 
In 2009, E.\ F.\ Wold \cite{Wold2009} gave an example of a $\Ccal^{\infty}$
smooth totally real manifold $M\subset\mathbb C^3$ which is 
polynomially convex, but fails to have bounded exhaustion hulls.  
In 2011, P.\ E.\ Manne, N.\ {\O}vrelid and E.\ F.\ Wold \cite{ManneWoldOvrelid2011} showed 
that a totally real submanifold $M\subset\mathbb C^n$ admits $\Ccal^1$
Carleman approximation only if $M$ is has bounded exhaustion hulls.  
Motivated by the problem of proving that the product of two 
totally real Carleman continua is again a Carleman continuum, 
B.\ Magnusson and E.\ F.\ Wold \cite{MagnussonWold2016} gave in 2016 a very simple proof that 
a closed set admits $\Ccal^0$ Carleman approximation only if it has bounded exhaustion hulls. 
Hence, we have the following characterization of closed totally real 
submanifolds which admit Carleman approximation.

%
%
\begin{theorem}\label{th:CarlemanBEH}
Let $M$ be a closed totally real submanifold of class $\Ccal^k$ in a Stein manifold $X$.   
Then, $M$ admits $\Ccal^k$-Carleman approximation by entire functions if and only if $M$ 
is $\Ocal(X)$-convex and has bounded exhaustion hulls.  
\end{theorem}

On the other hand, for any closed totally real submanifold $M$ in a Stein manifold there always exists
some Stein open neighborhood $\Omega$ of $M$ with respect to which 
$M$ admits Carleman approximation, see P. Manne \cite{Manne1993}.

\begin{problem}\label{prob:}
Let $E \subset\mathbb C^n$ be a closed polynomially convex  subset with 
the bounded exhaustion hulls property (see Definition \ref{def:BEHn}).
\begin{enumerate}[\rm (a)]
\item Suppose that $k\in\Z_+$ and $f\in\Ccal^k(\C^n)$ is holomorphic in $\mathring E$ 
and $\overline\partial$-flat to order $k$ along $E$. Is $f$ uniformly approximable on $E$ by entire functions? 
A positive answer in dimension $n=1$ is given by Arakelian's Theorem \ref{th:Arakelian}.
\item Suppose further that any $f$ as in part (a) 
is approximable uniformly on every compact $K\subset E$ by entire functions.
Does it follow that $E$ is an Arakelian set?
\end{enumerate}
\end{problem}

%
%

\section{Approximation of manifold-valued maps}
\label{sec:manifold}

We now apply results of the previous sections to approximation problems of Runge,
Mergelyan and Carleman type for maps to complex manifolds more general than Euclidean spaces.
Such problems arise naturally in applications of complex analysis to geometry, dynamics
and other fields. With the exception of Runge's theorem which leads to Oka theory
and the concept of Oka manifold (see Subsect.\ \ref{ss:Oka}), this area is fairly 
unexplored and offers interesting problems.

The most natural generalization of Runge's theorem to manifold-valued maps pertains to
maps from Stein manifolds (and Stein spaces) to {\em Oka manifolds}; see Theorem \ref{th:Oka}. 
This class of manifolds was  introduced in 2009 F.\ Forstneri\v c \cite{Forstneric2009CR} 
after having proved that all natural
Oka properties that had been considered in the literature, which a given complex manifold $Y$ 
might or might not have, are pairwise equivalent. (See also \cite{Larusson2010NAMS}.) 
The simplest one, which is commonly used as the definition of the 
class of Oka manifolds, is given by Definition \ref{def:Oka} below.
Since a comprehensive account of this subject is available in the monograph \cite{Forstneric2017E} and
the introductory surveys  \cite{Forstneric2013KAWA,ForstnericLarusson2011},
we only give a brief outline in Subsect.\ \ref{ss:Oka}, focusing on the approximation theorem 
in line with the topic of this survey.  

In Subsect.\ \ref{ss:Mergelyan-manifold} we consider the Mergelyan approximation problem 
for maps $K\to Y$ from a Stein compact $K$ in a complex manifold $X$ to another manifold $Y$.
Assuming that the map is of class $\Acal(K,Y)$, 
the main question is whether it is approximable uniformly on $K$ by maps holomorphic in open
neighborhoods of $K$. (The remaining question of approximability by entire maps $X\to Y$
is the subject of Oka theory discussed in Subsect.\ \ref{ss:Oka}.)
If this holds for every $f\in \Acal(K,Y)$, we say that 
the space $\Acal(K,Y)$ enjoys the {\it Mergelyan property}. Thanks to a Stein neighborhood theorem
due to E.\ Poletsky (see Theorem \ref{th:Poletsky3.1}), it is possible to show for many classes
of Stein compacts $K$ that the Mergelyan property for functions on $K$ implies the Mergelyan
property for maps $K\to Y$ to an arbitrary complex manifold $Y$.

In Subsect. \ref{ss:Carleman-manifold} we present some recent results on 
Carleman and Arakelian type approximation of manifold-valued maps.

%
%

\subsection{Runge theorem for maps from Stein spaces to Oka manifolds} \label{ss:Oka}

Oka theory concerns the existence, approximation, and interpolation results  for  holomorphic maps from 
Stein manifolds and, more generally, Stein spaces, to complex manifolds. To avoid
topological obstructions one considers globally defined continuous or smooth maps,
and the main question is whether they can be deformed to holomorphic maps, often with 
additional approximation and interpolation conditions.
Thus, Oka theory may be understood as the theory of homotopy principle in complex analysis,
a point of view emphasised in the monographs \cite{Gromov1986,Forstneric2017E}.

The classical aspect of Oka theory is known as the {\em Oka-Grauert theory}.
It originates in K.\ Oka's paper \cite{Oka1939} from 1939 
where he proved that a holomorphic line bundle $E\to X$ over a Stein manifold $X$ is holomorphically
trivial if it is topologically trivial; the converse is obvious. This is equivalent to the problem 
of constructing a holomorphic section $X\to E$ without zeros, granted a continuous section
without zeros. (Oka only considered the case when $X$ is a domain of holomorphy in $\C^n$ 
since the notion of a 
Stein manifold was introduced only in 1951 \cite{Stein1951}; however, the same proof 
applies to Stein manifolds and, more generally, to Stein spaces.) It follows that 
holomorphic line bundles $E_1\to X$, $E_2\to X$ over a Stein manifold are holomorphically
equivalent if they are topologically equivalent; it suffices to apply Oka's theorem to the line 
bundle $E_1^{-1}\otimes E_2$. In particular, any holomorphic line bundle over an open
Riemann surface $X$ is holomorphically trivial. A cohomological proof of this result is 
obtained by applying the long exact sequence
of cohomology groups to the exponential sheaf sequence $0\to \Z\to \Ocal_X\to \Ocal_X^*\to 0$,
where $\Ocal_X^*$ is the sheaf of nonvanishing holomorphic functions and the map 
$ \Ocal_X\to \Ocal_X^*$ is given by $f\mapsto e^{2\pi \imath f}$ 
(see e.g.\ \cite[Sect.\ 5.2]{Forstneric2017E}). 

In 1958, Oka's theorem was extended by H.\ Grauert \cite{Grauert1958MA} to much more 
general fibre bundles with 
complex Lie group fibres over Stein spaces; see also H.\ Cartan \cite{Cartan1958} for an exposition.
Grauert's results apply in particular to holomorphic  vector bundles
of arbitrary rank over Stein spaces and show that their holomorphic classification agrees with the 
topological classification. The cohomological point of view is still possible by considering nonabelian 
cohomology groups with 
values in a Lie group. Surveys of Oka-Grauert theory can be found in the paper \cite{Leiterer1986}  by 
J.\ Leiterer and in the monograph \cite{Forstneric2017E} by F.\ Forstneri\v c.

The main ingredient in the proof of Grauert's theorem is a parametric version of the 
Oka-Weil approximation theorem for maps from Stein manifolds to complex 
homogeneous manifolds. More precisely, given a compact $\Ocal(X)$-convex set $K$ in a Stein space
$X$ and a continuous map $f\colon X\to Y$ to a complex homogeneous manifold $Y$ such that $f$
is holomorphic in an open neighborhood of $K$, it is possible to deform $f$ by a homotopy
$f_t\colon X\to Y$ $(t\in [0,1])$ to a holomorphic map  $f_1$ such that every map $f_t$ 
in the homotopy is holomorphic in a neighborhood of $K$ and close to the initial map $f=f_0$ on $K$.
Analogous results hold for families of maps $f_p\colon X\to Y$ depending continuously on a parameter
$p$ in a compact Hausdorff space $P$, where the homotopy is fixed for values of $p\in P_0$ 
in a closed subset $P_0$ of $P$ for which the map $f_p\colon X\to Y$ is holomorphic on all
on $X$. In other words, Theorem \ref{th:COW} holds with the target $\C$ replaced by any 
complex homogeneous manifold $Y$, provided that all maps $f_p\colon X\to Y$ $(p\in P)$ in the family
are defined and continuous on all of $X$.
This point of view on Grauert's theorem is explained in \cite[Sects.\ 5.3 and 8.2]{Forstneric2017E}.

After some advances during 1960's, most notably those of O.\ Forster and K.\ J.\ Ramspott
\cite{ForsterRamspott1966IM1,ForsterRamspott1966IM2}, a major extension of the Oka-Grauert 
theory was made by M.\ Gromov \cite{Gromov1989} in 1989.   
He showed in particular that the existence of a dominating holomorphic spray 
on a complex manifold $Y$ implies all forms of the h-principle, also called the {\em Oka principle} in 
this context, for holomorphic maps from any Stein manifold to $Y$. The subject was brought into 
an axiomatic form by F.\ Forstneri{\v c} who introduced the class of {\em Oka manifolds} 
(see \cite{Forstneric2005AIF,Forstneric2006AM,Forstneric2009CR,Forstneric2010PAMQ} and
the monograph \cite{Forstneric2017E}). 

\begin{definition}
\label{def:Oka}
A complex manifold $Y$ is an {\em Oka manifold} if every holomorphic map $K\to Y$ from a neighborhood of
any compact convex set $K\subset \C^n$ for any $n\in \N$ can be approximated uniformly
on $K$ by entire maps $\C^n\to Y$. 
\end{definition}

The following version of the Oka-Weil  for maps from Stein spaces to Oka manifolds
is a special case of \cite[Theorem 5.4.4]{Forstneric2017E}.

%
%

\begin{theorem}[Runge theorem for maps to Oka manifolds] \label{th:Oka}
Assume that $X$ is a Stein space and $Y$ is an Oka manifold. 
Let $\dist$ denote a Riemannian distance function on $Y$. 
Given a compact $\Ocal(X)$-convex subset $K$ of $X$ and a continuous map $f\colon X\to Y$ 
which is holomorphic in a neighborhood of $K$, there exists for every $\epsilon>0$ a homotopy of 
continuous maps $f_t\colon X\to Y$ $(t\in[0,1])$ such that $f_0=f$, for every $t$ the map $f_t$ 
is holomorphic on a neighborhood of $K$ and satisfies $\sup_{x\in K}\dist(f_t(x),f(x))<\epsilon$, 
and the map $f_1$ is holomorphic on $X$. 
\end{theorem}

A complex manifold $Y$ which satisfies the conclusion of Theorem \ref{th:Oka} for every triple $(X,K,f)$
is said to satisfy  the {\em basic Oka property with approximation} (see \cite[p.\ 258]{Forstneric2017E}).
A more general version of this result (see \cite[Theorem 5.4.4]{Forstneric2017E}) includes the parametric case, 
as well as interpolation (or jet interpolation) on a closed complex subvariety $X_0$ of $X$
provided all maps $f_p\colon X\to Y$ $(p\in P)$ in a given continuous compact family are holomorphic on $X_0$,
or in a neighborhood of $X_0$ when considering jet interpolation. Since a compact convex set in $\C^n$ 
is $\Ocal(\C^n)$-convex, the condition that $Y$ be an Oka manifold is clearly necessary in 
Theorem \ref{th:Oka}.

The class of Oka manifolds contains all complex homogeneous manifolds, but also many nonhomogeneous
ones. For example, if the tangent bundle $TY$ of a complex manifold $Y$ is pointwise generated by 
{\em $\C$-complete} holomorphic vector fields on $Y$ (such a manifold is called {\em flexible}
\cite{ArzhantsevFlennerKalimanKutzschebauchZaidenberg2013DMJ}), 
then $Y$ is an Oka manifold \cite[Proposition 5.6.22]{Forstneric2017E}.
For examples and properties of Oka manifolds, see  \cite[Chaps.\ 5--7]{Forstneric2017E}. 

Recently, two new characterizations of the class of Oka manifolds have been found
by Y.\ Kusakabe. In his first paper \cite{Kusakabe2017IJM}, Kusakabe showed that a complex manifold
$Y$ is Oka if (and only if) for any Stein manifold $X$, the mapping space $\Ocal(X,Y)$
is $\C$-connected. In his second paper \cite{Kusakabe2018X}, he showed that $Y$
is Oka if (and only if) $Y$ satisfies Gromov's Condition $\mathrm{Ell}_1$ \cite{Gromov1989}.
This condition means that for every holomorphic map $f:X\to Y$ there exists a dominating
holomorphic spray $F:X\times \C^N\to Y$ with $F(\cdotp,0)=f$, where the domination
property means that for any fixed $x\in X$ the differential of the map $F(x,\cdotp):\C^N\to Y$
is surjective at $0\in \C^N$. Kusakabe's second result implies that a complex
manifold $Y$ is Oka if and only if every point $y_0\in Y$ has a Zariski open Oka neighbourhood
\cite[Theorem 1.4]{Kusakabe2018X}.

In another recent direction, L.\ Studer proved a homotopy theorem for Oka property 
and extended its validity to Oka pairs of sheaves \cite{Studer2019X}, generalizing
the work of Forster and Ramspott \cite{ForsterRamspott1966IM1}.

Theorem \ref{th:Oka} has a partial analogue in the algebraic category, concerning maps from 
affine algebraic varieties to algebraically subelliptic manifolds. For the definition of the latter class, 
see \cite[Definition 5.6.13(e)]{Forstneric2017E}. The following result is \cite[Theorem 6.15.1]{Forstneric2017E}; 
the original reference is \cite[Theorem 3.1]{Forstneric2006AJM}.

%
%
\begin{theorem}
\label{th:Runge-algebraic}
Assume that $X$ is an affine algebraic variety, $Y$ is an algebraically subelliptic manifold,
and $f_0\colon X\to Y$ is a (regular) algebraic map. Given a compact $\Ocal(X)$-convex subset $K$ of $X$, 
an open set $U \subset X$ containing $K$, and a homotopy $f_t\colon U\to Y$ of holomorphic maps 
$(t\in [0,1])$, there exists for every  $\epsilon >0$ an algebraic map $F\colon X\times \C\to Y$ such that
$F(\cdotp,0)=f_0$ and 
\[
	\sup_{x\in K,\ t\in [0,1]} \dist\left(F(x,t),f_t(x)\right)<\epsilon.
\]
In particular, a holomorphic map $X\to Y$ which is homotopic to an algebraic map 
can be approximated uniformly on compacts in $X$ by algebraic maps $X\to Y$.
\end{theorem}

Simple examples show that Theorem \ref{th:Runge-algebraic} does not hold in the absolute form, 
i.e., there are examples of holomorphic maps which are not
homotopic to algebraic maps (see \cite[Examples 6.15.7 and 6.15.8]{Forstneric2017E}). 

By \cite[Proposition 6.4.5]{Forstneric2017E}, the class of algebraically subelliptic manifolds  
contains all algebraic manifolds which are Zariski locally affine (such manifolds are said to be
of {\em Class $A_0$}, see \cite[Definition 6.4.4]{Forstneric2017E}), and all  complements of 
closed algebraic subvarieties of codimension at least two in such  manifolds. 
In particular, every complex Grassmanian is algebraically subelliptic, so 
Theorem \ref{th:Runge-algebraic} includes as a special case the result of 
W.\ Kucharz \cite[Theorem 1]{Kucharz1995} from 1995.  
Another paper on this topic is due to  J.\ Bochnak and W.\ Kucharz \cite{BochnakKucharz1996}.

In conclusion, we mention another interesting Runge type approximation theorem of a rather different type,
due to A.\ Gournay \cite{Gournay2012}.  A smooth almost complex manifold
$(M,J)$ is said to satisfy the {\em double tangent property} if through almost every point $p\in M$ 
and almost every $2$-jet of $J$-holomorphic discs at $p$, there exists a $J$-holomorphic map 
$u\colon\C\P^1\to M$ having this jet as its second jet at $0\in \C\P^1$.

%
%
\begin{theorem}[A.\ Gournay (2012), \cite{Gournay2012}] \label{th:Gournay}
Let $(M,J)$ be a compact almost complex manifold satisfying the double tangent property and let $R$ 
be a compact Riemann surface. Then, for every open set $U\subset R$ and every compact $K\subset U$, 
every $J$-holomorphic map $u\colon U\to M$ which continuously extends to $R$ can be approximated uniformly on $K$ by $J$-holomorphic maps from $R$ to $M$. 
\end{theorem}

%
%

\subsection{Mergelyan theorem for manifold-valued maps}\label{ss:Mergelyan-manifold}

In this section, we consider the question for which compact sets $K$ in a complex manifold $X$
does the approximability of functions in $\Acal^r(K)$ ($r\in \Z_+$) by 
functions in $\Ocal(K)$ imply the analogous result for maps to an arbitrary complex manifold $Y$.
Such approximation problems arise naturally in many applications.

Recall that $\Acal(K,Y)$ denotes the set of all continuous maps $K\to Y$ which are holomorphic in 
$\mathring K$, and that if $r\in\N$ then $\Acal^r(K,Y)$ is the set of all maps
$f\in \Acal(K,Y)$ which admit a $\Ccal^r$ extension to an open neighborhood of $K$ in $X$. 
We say that the mapping space $\Acal(K,Y)$ has the {\it Mergelyan property} if 
\[
	\Ocalc(K,Y)  = \Acal(K,Y),
\] 
that is, every continuous map $K\to Y$ that is holomorphic in the interior $\mathring K$  is a uniform limit 
of maps that are holomorphic in open neighborhoods of $K$ in $X$. 

%
%
\begin{lemma}\label{lem:Steinnbd}
Assume that $X$ is a complex manifold and $K\subset X$ is a compact set satisfying $\Ocalc(K)  = \Acal(K)$. 
Let $Y$ be a complex manifold, and let $f\in \Acal(K,Y)$.
Then $f\in \Ocalc(K,Y)$ if one of the following conditions hold:
\begin{enumerate}[\rm (a)]
\item The image $f(K)\subset Y$ has a Stein neighborhood in $Y$. 
\item The graph $G_f = \bigl\{(x,f(x)) : x\in K\bigr\}$ has a Stein neighborhood in $X\times Y$.
\end{enumerate} 
\end{lemma}

\begin{proof}
We will give a proof of (b); the proof of (a) is essentially the same.  
Assume that $V\subset X\times Y$ is a Stein neighbourhood of $G_f$.  
By the Remmert-Bishop-Narasimhan theorem  (see \cite[Theorem 2.4.1]{Forstneric2017E})
there is a biholomorphic map $\phi\colon V\to \Sigma\subset \C^N$
onto a closed complex submanifold of a Euclidean space.
By the Docquier-Grauert theorem (see \cite[Theorem 3.3.3]{Forstneric2017E})
there is a neighborhood $\Omega\subset \C^N$ of $\Sigma$ and a holomorphic retraction 
$\rho\colon\Omega\to \Sigma$. Assuming that $\Ocalc(K)  = \Acal(K)$, we can approximate the map 
$\phi\circ f \colon K\to \Sigma\subset \C^N$ as closely as desired uniformly on $K$
by a holomorphic map $G\colon U\to \Omega\subset\C^N$ from an open neighborhood
$U\subset X$ of $K$. The map $g = pr_Y \circ \phi^{-1} \circ \rho \circ G : U \to Y$ then 
approximates $f$ on $K$.
\qed\end{proof}

Given a compact set $K$ in a complex manifold $X$ and a complex manifold $Y$, let 
\[
		\Ocalcl(K,Y)
\]
denote the set of all continuous maps $f\colon K\to Y$ which are locally approximable by 
holomorphic maps, in the sense that every point $x \in K$ has an open neighborhood $U\subset X$ 
such that $f|_{K\cap \overline U}\in \Ocalc(K\cap \overline U)$. Clearly, 
\[
	 \Ocal(K,Y) \subset \Ocalc(K,Y) \subset \Ocalcl(K,Y)\subset \Acal(K,Y).
\]
When $Y=\C$, we simply write $\Ocal(K) \subset \Ocalc(K) \subset \Ocalcl(K)\subset \Acal(K)$.
We say that the space $\Acal(K,Y)$ has the {\em local Mergelyan property} if 
\begin{equation}\label{eq:LMP}
	\Ocalcl(K,Y) = \Acal(K,Y).
\end{equation}

The following  theorem was proved by E.\ Poletsky \cite[Theorem 3.1]{Poletsky2013}.

%
%

\begin{theorem}[Poletsky (2013), \cite{Poletsky2013}] \label{th:Poletsky3.1} 
Let $K$ be a Stein compact in a complex manifold $X$, and let $Y$ be a complex manifold.
For every $f\in \Ocalcl(K,Y)$, the graph of $f$ on $K$ is a Stein compact in $X\times Y$. 
In particular, if $\Acal(K,Y)$ has the local Mergelyan property \eqref{eq:LMP}, then 
the graph of every map $f\in \Acal(K,Y)$ is a Stein compact in $X\times Y$.
\end{theorem}

Poletsky's  proof  uses the technique of {\em fusing plurisubharmonic functions}.
Rough\-ly speaking, we approximate a collection of plurisubharmonic functions
$\rho_j\colon U_j\to\R$ on open sets $U_j\subset X\times Y$ covering the graph of $f$ 
by a plurisubharmonic function $\rho$ on $U=\bigcup_j U_j$, in the sense that the sup norm 
$\|\rho-\rho_j\|_{U_j}$ for each $j$ is estimated in terms of $\max_{i,j}\|\rho_i-\rho_j\|_{U_i\cap U_j}$ and a certain 
positive constant which depends on a strongly plurisubharmonic function
$\tau$ in a Stein open neighborhood of $K$ in $X$. This fusing procedure is  rather
similar to the proof of Y.-T.\ Siu's theorem \cite{Siu1976} given by J.-P.\ Demailly \cite{Demailly1990}
and Col\c{t}oiu \cite{Coltoiu1990}. (Demailly's proof can also be found in \cite[Sect.\ 3.2]{Forstneric2017E}.)
The functions $\rho_j$ alluded to above are of the form $|f_j(x)-y|^2$, where
$(x,y)$ is a local holomorphic coordinate on $U_j=V_j\times W_j$ with $V_j\subset X$ and $W_j\subset Y$,
and $f_j\in \Ocal(U_j,Y)$ is a holomorphic map which approximates $f$ on $U_j\cap K$.
(Such local approximations exist by the hypothesis of the theorem.) 
By this technique, one finds strongly plurisubharmonic exhaustion functions 
on arbitrarily small open neighborhoods of the graph of $f$ in $X\times Y$;
by Grauert's theorem \cite{Grauert1958AM} such neighborhoods are Stein.

In the special case when the set $K$ in Theorem \ref{th:Poletsky3.1} 
is the closure of a relatively compact strongly pseudoconvex Stein domain, 
the existence of a Stein neighborhood basis of the graph of any 
map $f\in \Acal(K,Y)$ was first proved by F.\ Forstneri\v c \cite{Forstneric2007AJM} in 2007. 
His proof uses the method of gluing holomorphic sprays.
                        
Theorem \ref{th:Poletsky3.1} and  Lemma \ref{lem:Steinnbd} give the following corollary.

\begin{corollary} \label{cor:Poletsky1}
Let $K$ be a Stein compact in a complex manifold $X$. If $\Acal(K)=\Ocalc(K)$,
then $\Ocalcl(K,Y)=\Ocalc(K,Y)$ holds for any complex manifold $Y$.
\end{corollary}

\begin{proof}
Note that $\Ocalc(K,Y)\subset \Ocalcl(K,Y)$. 
Assume now that $f\in \Ocalcl(K,Y)$. By Theorem \ref{th:Poletsky3.1}, the graph of $f$ on $K$ 
admits an open Stein neighborhood in $X\times Y$. 
Assuming that $\Acal(K)=\Ocalc(K)$, Lemma \ref{lem:Steinnbd}(b) shows that $f\in \Ocalc(K,Y)$.
\qed\end{proof}

In light of Theorem \ref{th:Poletsky3.1} and Corollary \ref{cor:Poletsky1}, it is natural to ask 
when does the space $\Acal(K,Y)$ enjoy the local Mergelyan property \eqref{eq:LMP}.
To this end, we introduce the following property of a compact set in a complex manifold.

%
%
\begin{definition}  
\label{def:SLMP}
A compact set $K$ in a complex manifold $X$ enjoys the {\em strong local Mergelyan property}
if for every point $x\in K$ and  neighborhood $x\in U\subset X$ there is
a neighborhood $x\in V\subset U$ such that $\Acal(K\cap \overline V)=\Ocalc(K\cap \overline V)$.
\end{definition}

\label{page:rem:SLMP}
%
%
\begin{remark}\label{rem:SLMP}
Clearly, the strong local Mergelyan property of $K$ implies the local Mergelyan property 
$\Acal(K)=\Ocalcl(K)$ of the algebra $\Acal(K)$. However, the former property is ostensibly 
stronger since it asks for approximability of functions defined on small neighborhoods of points 
in $K$, and not only of functions in $\Acal(K)$. If $K$ has empty interior,
we have $\Acal(K)=\Ccal(K)$ and the two properties are equivalent by the Tietze 
extension theorem for continuous functions. 
 Theorem \ref{th:localization-converse} due to A.\ Boivin and B.\ Jiang  \cite[Theorem 1]{BoivinJiang2004}
shows that, for a compact set $K$ in a Riemann surface,
the Mergelyan property $\Acal(K)=\Ocalc(K)$ implies the strong local Mergelyan
property of $K$. We do not know whether the same holds for compact sets in 
higher dimensional manifolds. It is obvious that every compact set 
with boundary of class $\Ccal^1$ in any complex manifold has the strong local Mergelyan property.
Note also that the strong local Mergelyan property for functions implies the same property
for maps to an arbitrary complex manifold $Y$, for the simple reason that locally any map
has range in a local chart of $Y$ which is biholomorphic to an open subset of a Euclidean space.
This is the main use of this property in the present paper.
\qed\end{remark}
 
%
%
\begin{problem}\label{prob:SLMP}  Let $K$ be a compact set in a complex manifold $X$.
\begin{enumerate} 
\item Does $\Acal(K)=\Ocalcl(K)$ imply the strong local Mergelyan property of $K$?
\item Does $\Acal(K)=\Ocalc(K)$ imply the strong local Mergelyan property of $K$?
\end{enumerate}
\end{problem}

We have the following corollary to Theorem \ref{th:Poletsky3.1}.

%
%
\begin{corollary}\label{cor:SLMP}
Let $K$ be a compact set in a complex manifold $X$.
\begin{enumerate}[\rm (a)]
\item If $K$ has the strong local Mergelyan property (see Definition \ref{def:SLMP}), 
then $\Acal(K,Y)=\Ocalcl(K,Y)$ holds for every complex manifold $Y$.
\item   If $K$ is a Stein compact with the strong local Mergelyan property and $\Acal(K)=\Ocalc(K)$, 
then $\Acal(K,Y)=\Ocalc(K,Y)$ holds for every complex manifold $Y$. 
\item   If $K$ is a Stein compact with $\Ccal^1$ boundary such that $\Acal(K)=\Ocalc(K)$, 
then $\Acal(K,Y)=\Ocalc(K,Y)$ holds for every complex manifold $Y$. 
\end{enumerate}
\end{corollary}

\begin{proof}
(a) Let $f\in \Acal(K,Y)$. Every point $x\in K$ has an open neighborhood $U_x\subset X$ such that 
$f(K\cap \overline U_x)$ is contained in a coordinate chart $W \subset Y$ biholomorphic to an open 
subset of $\C^n$, $n=\dim Y$. Since $K$ is assumed to have the strong local Mergelyan property, there exists a 
compact relative neighborhood $K_x\subset K\cap U$ of the point $x$ in $K$ such that 
$f|_{K_x}\in \Ocalc(K_x,W)$. (See Remark \ref{rem:SLMP}.) 
This means that $f\in \Ocalcl(K,Y)$, thereby proving (a).
In case (b), Corollary \ref{cor:Poletsky1} implies $\Ocalcl(K,Y)=\Ocalc(K,Y)$,
and together with part (a) we get $\Acal(K,Y)=\Ocalc(K,Y)$.
In case (c),  the  set $K$ clearly has the strong local Mergelyan property,
so the conclusion follows from (b).
\qed\end{proof}

The following case concerning compact sets in Riemann surfaces may be of particular interest
(see \cite[Theorem 1.4]{Forstneric2019MMJ}).

%
%
\begin{corollary} \label{cor:KinRS}
If $K$ is a compact set in a Riemann surface $X$ such that $\Acal(K)=\Ocalc(K)$, 
then $\Acal(K,Y)=\Ocalc(K,Y)$ holds for any complex manifold $Y$.
This holds in particular if $X\setminus K$ has no relatively compact connected components.
\end{corollary}

\begin{proof}
Note that any compact set in a Riemann surface is a Stein compact (since every open Riemann surface
is Stein according to H.\ Behnke and K.\ Stein \cite{BehnkeStein1949}). 
According to Theorem \ref{th:localization-converse}, the hypothesis $\Acal(K)=\Ocalc(K)$ implies that 
$K$ has the strong local Mergelyan property, so the result follows from Corollary \ref{cor:SLMP}.

In the special case when $X\setminus K$ has no relatively compact connected components,
we can give a simple proof as follows. By Theorem \ref{th:Mergelyan2}, every function $f\in \Acal(K)$ 
is a uniform limit on
$K$ of functions in $\Ocal(X)$, hence $\Acal(K)=\Ocalc(K)$. Fix a point $x\in K$ and let $U\subset X$
be a coordinate neighborhood of $x$ with a biholomorphic map
$\phi\colon U\to \D\subset\C$. Pick a number $0<r<1$.
The compact set $K'=K\cap \phi^{-1}(r\overline{\mathbb D})$ 
does not have any holes in $U$. (Indeed, any such 
would also be a hole of $K$ in $X$, contradicting the hypothesis.) 
By Theorem \ref{th:Mergelyan2} it follows that $\Acal(K')=\Ocalc(K')$. 
This shows that $K$ enjoys the strong local Mergelyan property,
and hence the conclusion follows from Corollary \ref{cor:SLMP}.
\qed\end{proof}

The following consequence of Corollary \ref{cor:KinRS} and of the Oka principle (see Theorem \ref{th:Oka})
has been observed recently in \cite[Theorem 1.2]{Forstneric2019MMJ}.

%
%
\begin{corollary}[Mergelyan theorem for maps from Riemann surfaces to Oka manifolds] 
\label{cor:MergeljanOka}
If $K$ is a compact set without holes in an open Riemann surface $X$ and $Y$ is an Oka manifold,
then every continuous map $f\colon X\to Y$ which is holomorphic in $\mathring K$
can be approximated uniformly on $K$ by holomorphic maps $X\to Y$ homotopic to $f$.
\end{corollary}

It was shown by J.\ Winkelmann \cite{Winkelmann1998} in 1998 that Mergelyan's theorem 
also holds for maps from compact sets in $\mathbb C$ to the domain 
$\C^2\setminus \R^2$; this result is not covered by Corollary \ref{cor:MergeljanOka}.
His proof can be adapted to give the analogous result for maps from any open Riemann surface
to $\C^2\setminus \R^2$.

%
%
\begin{remark}\label{rem:Poletsky}
The following claim was stated by E.\ Poletsky \cite[Corollary 4.4]{Poletsky2013}:

\smallskip
\noindent (*)  {\em If  $K$ is a Stein compact in a complex manifold $X$ 
and $\Acal(K)$ has the Mergelyan property, then $\Acal(K,Y)$ has 
the Mergelyan property for any complex manifold $Y$.}
\smallskip

The proof  in \cite{Poletsky2013} tacitly assumes that under the assumptions of the corollary
the space $\Acal(K,Y)$ has the local Mergelyan property, but no explanation for this is given. 
Corollaries \ref{cor:SLMP} and \ref{cor:KinRS} above provide several sufficient conditions
for this to hold. We do not know whether (*) is true for every Stein compact  in 
a complex manifold of dimension $>1$; compare with Remark \ref{rem:SLMP} on p.\ \pageref{page:rem:SLMP}.
\qed\end{remark}

%
%
\begin{corollary}\label{cor:admissibleXY}
If $S=K\cup M$ is a strongly admissible set in a complex manifold $X$ (see Definition \ref{def:admissible2}),
then $\Acal(S,Y)=\Ocalc(S,Y)$ holds for any complex manifold $Y$. Furthermore, for each $r\in\N$,
every map $f\in \Acal^r(S,Y)$ is a $\Ccal^r(S,Y)$ limit of maps $U\to Y$ holomorphic in open 
neighborhoods $U\subset X$ of $S$.
\end{corollary}

\begin{proof}
It is clear from the definition of a strongly admissible set that for every point $x\in S$ and 
neighborhood $x\in U\subset X$ 
there is a smaller neighborhood $U_0\Subset U$ of $x$ such that the set $S_0=\overline U_0\cap S$ is also
strongly admissible. By Theorem \ref{th:MergelyanW} we have that $\Acal(S)=\Ocalc(S)$, and also 
$\Acal(S_0)=\Ocalc(S_0)$ for any $S_0$ as above. This means that $S$ has the strong
local Mergelyan property. The conclusion now follows from Corollary \ref{cor:SLMP}.
A similar argument applies to maps of class $\Acal^r(S,Y)$ for any $r\in \N$.
\qed\end{proof}

In the special case when the strongly admissible set $K=S$ is the closure of a relatively compact
strongly pseudoconvex domain, Corollary \ref{cor:admissibleXY} was proved by
F.\ Forstneri\v c \cite{Forstneric2007AJM} in 2007. His proof is different from those
above which rely on Poletsky's Theorem \ref{th:Poletsky3.1}. 
Instead it uses the method of gluing sprays, which is essentially a nonlinear version of the $\dibar$-problem. 
In the same paper, Forstneri\v c showed that many natural mapping spaces $K\to Y$ 
carry the structure of a Banach, Hilbert of Fr{\'e}chet manifold
(see \cite[Theorem 1.1]{Forstneric2007AJM} and also \cite[Theorem 8.13.1]{Forstneric2017E}). 
The following special case of the cited result is relevant to the present discussion.

%
%
\begin{theorem}
Let $K$ be a compact strongly pseudoconvex domain with $\Ccal^2$ boundary in a 
Stein manifold $X$. Then, for every $r\in\Z_+$ and any complex manifold $Y$ the space
$\Acal^r(K,Y)$ carries the structure of an infinite dimensional Banach manifold.
\end{theorem}

Further  and more precise approximation results for maps from compact strongly pseudoconvex
domains to Oka manifolds were obtained by B.\ Drinovec Drnov\v sek and F.\ Forstneri\v c
in \cite{DrinovecForstneric2008FM}.

The proof of Theorem \ref{th:admissible} in Subsect.\ \ref{ss:submanifolds} 
is easily generalized to give the following approximation result for sections of holomorphic 
submersions over admissible sets in complex spaces.  This plays a major role in the constructions
in Oka theory (in particular, in the proof of \cite[Theorem 5.4.4]{Forstneric2017E}).

%
%
%
%
\begin{theorem}\label{th:approximation-handlebody}
Assume that $X$ and $Z$ are complex spaces, $\pi\colon Z\to X$ is a holomorphic submersion, 
and $X'$ is a closed complex subvariety of $X$ containing its singular locus $X_{\rm sing}$.  
Let $S=K\cup M$ be an admissible set in $X$ (see Definition \ref{def:admissible2}), where $M\subset X\setminus X'$ is 
a compact totally real submanifold of class $\Ccal^k$ for some $k\in\N$. Given an open set $U\subset X$
containing $K$ and a section $f\colon U\cup M\to Z|_{U\cup M}$ such that $f|_U$ is holomorphic and $f|_M \in \Ccal^k(M)$,
there exist for every $s\in\N$ a sequence of open sets $V_j\supset S$ in $X$ and holomorphic sections 
$f_j\colon V_j\to Z|_{V_j}$ $(j\in\N)$ such that $f_j$ agrees with $f$ to order $s$ along  
$X'\cap V_j$ for each $j\in \N$, and $\lim_{j\to\infty} f_j|_S=f|_S$ in the $\Ccal^k(S)$-topology. 
\end{theorem}

A version of this result, with some loss of derivatives on the totally real submanifold $M$ 
(due to the use of H{\"o}rmander's $L^2$ method) and without the interpolation condition, 
is \cite[Theorem 3.1]{Forstneric2005AIF}.  (A proof also appears in \cite[Theorem 3.8.1]{Forstneric2017E}.)
The case when $Z=X\times \C$ (i.e., for functions) and without loss of derivatives was proved earlier by 
P.\ Manne \cite{Manne1993PhD} by using the convolution method (see Proposition \ref{prop:local} in 
Subsect.\ \ref{ss:submanifolds}). The general case 
is obtained from the special case for functions by following \cite[proof of Theorem 3.8.1]{Forstneric2017E},
noting also that the interpolation condition on the subvariety $X'$ is easily 
achieved by a standard application of the Oka-Cartan  theory. As always in results of this type, 
one begins by showing that the graph of the section admits a Stein neighborhood in $Z$;
see \cite[Lemma 3.8.3]{Forstneric2017E}.

Another case of interest is when $K$ is a compact set with empty interior, so 
$\Acal(K)=\Ccal(K)$. The following result is due to E.\ L.\ Stout \cite{Stout2011}.

%
%
\begin{theorem} \label{th:Stout1}
If $K$ is a compact set in a complex space $X$ such that $\Ccal(K)=\Ocalc(K)$ (hence $\mathring K=\varnothing$),
then $\Ccal(K,Y)=\Ocalc(K,Y)$ holds for any complex manifold $Y$.
\end{theorem}

Unlike in the previous results, the set $K$ in Theorem \ref{th:Stout1} need not be a Stein compact.
Special cases of Stout's theorem were obtained earlier by D.\ Chakrabarti (2007, 2008) \cite{Chakrabarti2007,Chakrabarti2008} who also obtained uniform approximation of continuous maps on arcs by 
pseudoholomorphic curves in almost complex manifolds. 

\begin{proof}
Choose a smooth embedding $\phi\colon Y\hra \R^m$ for some
$m\in\N$. Considering $\R^m$ as the real subspace of $\C^m$, the graph 
$Z = \{(y,\phi(y)) : y\in Y\} \subset Y\times \R^m\subset Y\times \C^m$ is a totally real submanifold of 
$Y\times \C^m$, so it has an open Stein neighborhood $\Omega$ in $Y\times \C^m$.
Let $\pi\colon Y\times \C^m\to Y$ denote the projection onto the first factor.
Given a continuous map $f\colon K\to Y$, the hypothesis of the theorem together with 
Lemma \ref{lem:Steinnbd} imply that the continuous map $K\ni x\mapsto F(x)=(f(x),\phi(f(x))) \in \Omega$
can be approximated by holomorphic maps $G \colon U\to \Omega$ in open
neighborhoods $U\subset X$ of $K$. The map $g=\pi\circ G\colon U\to Y$ then 
approximates $f$ on $K$.
\qed\end{proof}

%
%

\subsection{Carleman and Arakelian theorems for manifold-valued maps}\label{ss:Carleman-manifold}

In Sect.\ \ref{sec:unbounded} and Subsect.\ \ref{ss:CarlemanCn} we have considered Carleman and Arakelian
type approximation in one and several variables, respectively. In this section, we present some applications
and extensions of these results to manifold-valued maps. 

The following result has been proved recently by B.\ Chenoweth.

%
%
\begin{theorem}[Chenoweth (2019), \cite{Chenoweth2019PAMS}] \label{th:Chenoweth}
Let $X$ be a Stein manifold and $Y$ be an Oka manifold. If $K \subset X$ is a compact 
$\Ocal(X)$-convex subset and $M\subset X$ is a closed totally real submanifold of class $\Ccal^r$ 
$(r\in\N)$ with bounded exhaustion hulls (see Definition \ref{def:BEHn}) such that $K \cup M$ is $\Ocal(X)$-convex,
then for any $k\in \{0,1,\ldots,r\}$ the set $K \cup M$ admits $\Ccal^k$-Carleman approximation 
of maps $f \in \Ccal^k(X,Y)$ which are holomorphic on a neighborhood of $K$.
\end{theorem}

This  is proved by inductively applying Mergelyan's theorem for admissible sets
(see Theorem \ref{th:approximation-handlebody}), together with the Oka principle
for maps from Stein manifolds to Oka manifolds (see \cite[Theorem 5.4.4]{Forstneric2017E}
which is a more precise version of Theorem \ref{th:Oka} above). 
These two methods are intertwined at every step of the induction procedure. 
In view of Theorem \ref{th:CarlemanBEH} characterizing 
totally real submanifolds admitting Carleman approximation, the conditions 
in the theorem are optimal. 

Carleman type approximation theorems have also been proved for some special classes of
maps such as embeddings and automorphisms. Typically, proofs of such results combine methods 
of approximation theory with those from the Anders{\'e}n-Lempert theory concerning 
holomorphic automorphisms of complex Euclidean spaces and, more generally, 
of Stein manifolds with the density property.
Space limitation do not allow us to present this theory here; instead, we refer the reader to
the recent survey in \cite[Chapter 4]{Forstneric2017E}.

We have already seen that Arakelian type approximation on closed sets 
with unbounded interior is considerably more difficult than Carleman approximation. 
In fact, we are not aware of a single result of this type on subsets of  $\C^n$ for $n>1$.
However, the following extension of the classical one variable Arakelian's theorem (see Theorem \ref{th:Arakelian})
was proved by F.\ Forstneri{\v c} \cite{Forstneric2019MMJ} in 2019.

%
%
%
\begin{theorem} \label{th:Arakelian1}
If $E$ is an Arakelian set in a domain $X\subset \C$ and $Y$ is a compact complex homogeneous manifold,
then every continuous map $X\to Y$ which is holomorphic in $\mathring E$ can be approximated 
uniformly on $E$ by holomorphic maps $X\to Y$.
\end{theorem}

The scheme of proof in \cite{Forstneric2019MMJ} follows the
proof of Theorem \ref{th:Arakelian}, but with improvements from Oka theory which are needed in 
the nonlinear setting.
The proof does not apply to general Oka manifolds, not even to noncompact homogeneous manifolds.
Note that the approximation problems of Arakelian type for maps to noncompact
manifolds may crucially depend on the choice of the metrics on both spaces.

%
%

\section{Weighted Approximation in $L^2$ spaces}
 \label{sec:weights}

All approximation results considered so far were in one of the $\Ccal^k$ topologies on the respective sets.
We now present some results of a rather different kind, concerning approximation and density in weighted
$L^2$ spaces of holomorphic functions.

Let $\Omega$ be a domain in $\C^n$, and let $\phi$ be a plurisubharmonic function on $\Omega$. 
We denote by $L^2(\Omega, e^{-\phi})$ the space of measurable functions which are square integrable with 
respect to the measure $e^{-\phi}d\lambda$, where $d\lambda$ is the Lebesgue measure:
\[
	\|f\|^2_{\phi}:= \int_\Omega |f|^2 e^{-\phi}d\lambda <\infty.
\]
By $H^2(\Omega, e^{-\phi})$ we denote the space of holomorphic functions on 
$\Omega$ with finite $\phi$-norm:
\[
	H^2(\Omega, e^{-\phi}) = \bigl\{ f\in \Ocal(\Omega): \|f\|_{\phi}  <\infty\bigr\}.
\]
Note that if $\phi_1\leq \phi_2,$ then $H^2(\Omega, e^{-\phi_1})\subset H^2(\Omega, e^{-\phi_2})$ 
and the inclusion map is continuous,  in fact, norm decreasing. 

Let $z=(z_1,\ldots,z_n)$ be coordinates on $\C^n$ and $|z|^2=\sum_{i=1}^n |z_i|^2$.
Let $\phi_1\leq \phi_2\leq \cdots $ and $\phi$ be plurisubharmonic functions on $\C^n$ with
$\phi_j\rightarrow \phi$ pointwise as $j\to\infty$.  Set 
\[
	\psi_j=\phi_j+\log (1+|z|^2), \qquad \psi=\phi+\log (1+|z|^2).	
\]
Assume in addition that $\int_K e^{-\phi_1}d\lambda <\infty$ for every compact set $K\subset \C^n$.
The following theorem was proved by B.\ A.\ Taylor  in 1971; see \cite[Theorem 1.1]{Taylor1971}. 

\begin{theorem} 
(Assumptions as above.) For every $f\in H^2(\C^n, e^{-\phi})$ 
there is a sequence $f_j\in H^2(\C^n, e^{-\psi_j})$ such that $\|f_j-f\|_{\psi}\rightarrow 0$ as $j\to\infty$.
\end{theorem}

This result was improved in a recent paper by J.\ E.\ Forn{\ae}ss and J.\ Wu \cite{FornaessWu2017JGA}.

\begin{theorem}
Let $\phi_1\leq \phi_2\leq \cdots$ and $\phi$ be plurisubharmonic functions on $\C^n$ such that
$\phi_j\rightarrow \phi$ pointwise. For any $\epsilon >0,$ let $\tilde{\phi}_j=\phi_j+\epsilon \log (1+|z|^2)$
and $\tilde{\phi}=\phi+\epsilon \log(1+|z|^2)$. Then
$\bigcup_{j=1}^\infty H^2(\C^n,e^{-\tilde{\phi}_j})$ is dense in $H^2(\C^n,e^{-\tilde{\phi}})$.
\end{theorem}

\begin{question}
Let $\phi_1\leq \phi_2\leq \cdots$ and $\phi$ be plurisubharmonic functions on $\Omega\subset \C^n$ such that
$\phi_j\rightarrow \phi$. Is $\bigcup_{j=1}^\infty H^2(\Omega, e^{-\phi_j})$ dense in $H^2(\Omega, e^{-\phi})$?
\qed\end{question}

Recently, J.\ E.\ Forn{\ae}ss and J.\ Wu \cite{FornaessWu2017X} solved this problem in the case of $\Omega=\C$.

\begin{theorem}
If $\phi_1\leq \phi_2\leq \cdots$ and $\phi$ are subharmonic functions on $\C$ such that
$\phi_j\rightarrow \phi $ a.e.\ as $j\to\infty$, then $\bigcup_{j=1}^\infty H^2(\C, e^{-\phi_j})$ is 
dense in $H^2(\C,e^{-\phi})$.
\end{theorem}

This problem has a rich history in dimension one. Here one considers more general weights $w$ 
which are positive measurable functions on a domain $\Omega\subset \C$, and one defines for $1\leq p<\infty$ 
the weighted $L^p$-space of holomorphic functions:
\[
	H^p(\Omega,w)=\left \{ f\in \Ocal(\Omega):  \int_{\Omega} |f|^p w d \lambda<\infty \right\}. 
\]
The so called {\em completeness problem} is whether polynomials in $H^p(\Omega,w)$ are dense. 
There are two lines of investigation. 
One is about finding sufficient conditions on the domain and the weight in order for the polynomials to be dense in 
the weighted Hilbert space. Another one is to look at specific types of domains and ask
the same question for the weight function. These questions go back to T.\ Carleman \cite{Carleman1923}
who proved in 1923 that if $\Omega$ is a Jordan domain and $w \equiv 1$, then holomorphic polynomials
are dense in $H^2(\Omega)=L^2(\Omega)\cap \Ocal(\Omega)$.
Carleman's result was  extended by O.\ J.\ Farrell and A.\ I.\ Marku\v{s}evi\v{c} to Carath\'{e}odory domains 
(see \cite{Farrell1934,Mergelyan1962}).  
It is well known that this property need not hold for non-Carath\'{e}odory regions. 
The book by D.\ Gaier \cite{Gaier1987} (see in particular Chapter 1, Section 3) contains further 
results about $L^2$ polynomials approximation on some simply connected domains in the plane.  
For weight functions other than the identity,  L.\ I.\ Hedberg proved in 1965 \cite{Hedberg1965} that 
polynomials are dense when $\Omega$ is a Carath\'{e}odory domain, the weight function is continuous,
and it satisfies some technical condition near the boundary. For certain non-Carath\'{e}odory domains, 
the weighted polynomial approximation is usually considered under the assumption that 
the weight $w$ is essentially bounded and satisfies some additional conditions. For a more 
complete description of the history of this problem and many related references, 
see the survey by J.\ E.\ Brennan \cite{Brennan1977}. 

By using H\"ormander's $L^2$ estimate for the $\dibar$-operator,  B.\ A.\ Taylor \cite{Taylor1971} proved 
the following result which can be seen as a major breakthrough for general weighted approximation.
(See also D.\ Wohlgelernter \cite{Wohlgelernter1975}.)

\begin{theorem}[B.\ A.\ Taylor (1971), Theorem 2 in \cite{Taylor1971}] \label{th:Taylor}
If $\phi$ is a convex function on $\C^n$ such that the space $H^2(\C^n, e^{-\phi})$
contains all polynomials, then polynomials are dense in $H^2(\Omega, e^{-\phi})$. 
\end{theorem}
 
In 1976 N.\ Sibony \cite{Sibony1976} generalized Taylor's result as follows.
Given a domain $\Omega\subset \C^n$, we denote by $d_\Omega(z)$ the Euclidean distance 
of a point $z\in \Omega$ to $\C^n\setminus \Omega$. Write $\delta_0(z)=(1+|z|^2)^{-1/2}$ and
\[
	 \delta_\Omega(z) = \min \{ d_\Omega(z), \delta_0(z)\},\quad z\in\Omega.
\]

\begin{theorem} [N.\ Sibony (1976), \cite{Sibony1976}]
If $\Omega$ is an open convex domain in $\mathbb{C}^n$ and $\phi$ is a convex function on $\Omega$
satisfying
\[
	\sup_{z\in \Omega} e^{-\phi(z)} \delta^{-k}_\Omega(z) < +\infty, \qquad k\in\N,
\]
then polynomials are dense in $H^p(\Omega, e^{-\phi})$ for all $1\le p\le +\infty$.
\end{theorem}

In the same paper, Sibony  also proved the analogous result for homogeneous plurisubharmonic weights.

\begin{theorem} [N.\ Sibony (1976), \cite{Sibony1976}]
Let $\phi$ be a plurisubharmonic function on $\C^n$ which is complex homogeneous of order $\rho >0$, 
that is, $\phi(uz) = |u|^\rho \phi(z)$ for all $u\in \mathbb{C}$ and $z\in \mathbb{C}^n$. Then,
polynomials are dense in $H^2(\Omega, e^{-\phi})$.
\end{theorem}

It is well known that every convex function is plurisubharmonic, but the converse is not true. 
In view of Theorem \ref{th:Taylor} it is therefore natural to ask the following question.
Let $\phi$ be a plurisubharmonic function on a Runge domain $\Omega\subset \C^n$.
Suppose that the restrictions of polynomials to $\Omega$ belong to $H^2(\Omega, e^{-\phi})$. 
Does it follow that polynomials are dense in $H^2(\Omega, e^{-\phi})$?
Recently, S.\ Biard, J.\ E.\ Forn{\ae}ss and J.\ Wu   \cite{BiardFornaessWu2018} found a counterexample
in the plane.

\begin{theorem}
There is a subharmonic function $\phi$ on $\C$ such that all polynomials
belong to $H^2(\C, e^{-\phi})$, but polynomials are not dense in $H^2(\C, e^{-\phi})$.
\end{theorem}

They also proved the following positive result under additional conditions.

\begin{theorem}
Let $\phi$ be plurisubharmonic on a neighborhood of $\overline{\Omega}\subset \C^n$, and
suppose that $\overline{\Omega}$ is bounded, uniformly H-convex and polynomially convex. 
If $H^2(\Omega, e^{-\phi})$ contains all polynomials, then polynomials are dense in
$H^2(\Omega, e^{-\phi})$.
\end{theorem}

Recall that a compact set $K\subset \C^{n}$ is said to be {\em uniformly $H$-convex} 
if there exist a sequence $\varepsilon_{j}>0$ converging to $0$, a constant $c>1$, and a 
sequence of pseudoconvex domains $D_{j}\subset \C^n$ such that $K\subset D_{j}$ and 
\[
	\varepsilon_{j}\leq {\rm dist}(K,\C^{n}\setminus D_{j})\leq c\varepsilon_{j},
	\quad j=1,2,\ldots.
\] 
This terminology is due to E.\ M.\ \v Cirka \cite{Chirka1969} who showed that uniform H-convexity 
implies a Mergelyan-like approximation property for holomorphic functions;
however, the condition was used in  $L^2$ approximation results  already by 
L.\ H\"ormander and J.\ Wermer  \cite{HormanderWermer1968} in 1968
(see Remark \ref{rem:HWadmissible}). 
A related notion is that of a {\em strong Stein neighborhood basis} 
(which holds in particular for strongly hyperconvex domains);
we refer to the paper by S.\ {\c S}ahuto{\u g}lu \cite{Sahutoglu2012}.
It seems an open problem whether any of these conditions for the closure 
$K=\overline D$ of a smoothly bounded pseudoconvex domain $D\Subset \C^n$
implies the Mergelyan property for the algebra $\Acal(K)$.

%
%

\section{Appendix: Whitney's Extension Theorem}\label{app:Whitney}

Given a closed set $K$ in a smooth manifold $X$, the notation $f\in\Ccal^m(K)$ means that $f$ is the restriction 
to $K$ of a function in $\Ccal^m(X)$.

%
%
\begin{theorem}[Whitney (1934), \cite{Whitney1934TAMS}]\label{th:Whitney}
Let $\Omega\subset\mathbb R^n$ be a domain, and assume that there exists a constant $c\ge 1$ such that any 
two points $x,y\in\Omega$ can be joined by a curve in $\Omega$ of length less than $c|x-y|$.  
If $f\in\Ccal^m(\Omega)$ is such that all its partial derivatives of order $m$ extend continuously 
to $\overline\Omega$, then $f\in\Ccal^m(\overline\Omega)$.
\end{theorem}

In fact, a much stronger extension theorem was proved by Whitney.  To state it, we need to introduce some notation
and terminology.  

Let $K\subset\mathbb R^n$ be a compact set, and fix $m\in\mathbb N$. A collection $f=(f_\alpha)$ of functions 
$f_\alpha\in\Ccal (K)$, where $ \alpha=(\alpha_1,\ldots,\alpha_n)\in\Z^n_+$ is a multiindex
with $|\alpha|=\alpha_1+\cdots +\alpha_n\le m$, is called an $m$-jet on $K$.  
Let $\Jcal^m(K)$ denote the vector space of $m$-jets on $K$.   Set 
\[
	\|f\|_{m,K} = \max_{|\alpha|\le m} \, \sup_{x\in K}|f_\alpha(x)|.
\]
An $m$-jet $f=(f_\alpha)\in \Jcal^m(K)$ is said to be a {\em Whitney function of class $\Ccal^m$} on $K$ if 
\[
	f_\alpha(x)=\sum_{|\beta|\leq m-|\alpha|}\frac{f_{\alpha+\beta}(y)}{\beta !}(x-y)^\beta + o(|x-y|^{m-|\alpha|})
\]
holds for all $\alpha \in\Z^n_+$ with $|\alpha|\le m$ and all $x,y\in K$. 
We denote by $\Jcal^m_{\Wcal}(K)$ the space of all
Whitney functions of class $\Ccal^m$ on $K$.

%
%
\begin{theorem}[Whitney \cite{Whitney1934TAMS}, Glaeser \cite{Glaeser1958}] 
\label{th:WG}
Let $K$ be a compact set in $\R^n$. Given $f\in\Jcal^m(K)$, 
there exists $\tilde f\in\Ccal^m(\mathbb R^n)$ such that $\Jcal^m(\tilde f)|_{K}=f$ 
if and only if $f$ is a Whitney function of class $\Ccal^m$, that is, $f\in\Jcal^m_{\Wcal}(K)$.  
Furthermore, there exists a linear extension operator 
$\Lambda:\Jcal^m_{\Wcal}(K)\rightarrow \Ccal^m(\R^n)$ such that $\Jcal^m \Lambda(f)|_K=f$
for each $f\in \Jcal^m_{\Wcal}(K)$, and for every compact set $L\subset \R^n$ with $K\subset L$
there is a constant $C>0$ depending only on $K,L,m,n$ such that 
\begin{equation}\label{eq:bound}
	\|\Lambda(f)\|_{m,L} \le C \|f\|_{m,K}, \qquad f\in \Jcal^m_{\Wcal}(K).
\end{equation}
\end{theorem}

A proof of Whitney's theorem, including the extensions and 
simplifications due to Glaeser \cite{Glaeser1958}, can be found in the monograph by 
Malgrange \cite[Theorem 3.2 and Complement 3.5]{Malgrange1966}.

\begin{remark}
An inspection of the proof in \cite{Malgrange1966} 
shows that, if the set $K$ in Theorem \ref{th:WG} is the closure of a domain 
$\Omega\Subset\mathbb R^n$ with $\Ccal^m$-smooth boundary, then there are extension 
operators for all domains sufficiently close to $\Omega$ with the same bound in \eqref{eq:bound}. 
Furthermore, if $\Omega_j$ is a sequence of domains 
such that $\Omega_j\rightarrow\Omega$ in $\Ccal^m$ topology as $j\rightarrow\infty$, 
we may fix a domain $\widetilde\Omega$ containing $\overline\Omega$ and smooth maps 
$\phi_j:\widetilde\Omega\rightarrow\mathbb R^n$ such that $\phi_j(\Omega_j)=\Omega$ and 
$\phi_j\rightarrow\mathrm{Id}$ in the $\Ccal^m$-norm on $\widetilde\Omega$.
\qed \end{remark}


\subsection*{Acknowledgements}
J.\ E.\ Forn{\ae}ss and E.\ F.\ Wold are supported by the Norwegian Research Council grant number 240569.
F.\ Forstneri{\v c} is supported by the research program P1-0291 and grants J1-7256 
and J1-9104 from ARRS, Republic of Slovenia. 

We wish to thank the colleagues who gave us their remarks on a preliminary version of the paper or who
answered our questions: Antonio Alarc\'on, Severine Biard, Nihat Gogus, Ilya Kossovskiy, 
Finnur L\'arusson, Norman Levenberg, Petr Paramonov, Evgeny Poletsky, S{\"o}nmez {\c S}ahuto{\u g}lu, 
Edgar Lee Stout, Joan Verdera, Jujie Wu, and Mikhail Zaidenberg.





\end{document}